\documentclass[12pt,a4paper,reqno]{amsart}

\usepackage{xcolor}
\usepackage[all,2cell,arrow,color]{xy}

\usepackage{amssymb}
\usepackage{mathtools}
\usepackage{tensor}
\numberwithin{equation}{section}
\usepackage{graphicx} 
\usepackage{rotating}
\usepackage{pdflscape}
\definecolor{darkgreen}{rgb}{0,0.45,0} 
\usepackage[pagebackref,colorlinks,citecolor=darkgreen,linkcolor=darkgreen]{hyperref}
\usepackage{enumerate,xspace}

\usepackage{tikz}
\usetikzlibrary{intersections}
\tikzstyle{braid}=[thick]
\tikzset{arr/.style={circle,draw,inner sep=0}}
\tikzset{unit/.style={circle,fill,inner sep=0.05cm}}
\tikzset{empty/.style={inner sep=0pt, minimum size=0pt}}

\UseAllTwocells
\CompileMatrices

\definecolor{pinegreen}{rgb}{0.15,0.7,0.15}

 \definecolor{lightgrey}{rgb}{0.666666,0.666666,0.666666}

\renewcommand{\phi}{\varphi}

\newcommand{\cc}{\ensuremath{\mathcal C}\xspace}
\newcommand{\cm}{\ensuremath{\mathcal M}\xspace}
\newcommand{\cq}{\ensuremath{\mathcal Q}\xspace}
\newcommand{\mM}{\ensuremath{\mathbb M}\xspace}

\newcommand{\op}{^{\mathsf{op}}}

\newcommand{\rev}{^{\mathsf{rev}}}
\newcommand{\cop}{_{\mathsf{op}}}

\newcommand{\nto}{\nrightarrow}

  \newtheorem{proposition}[subsection]{Proposition}
  \newtheorem{lemma}[subsection]{Lemma}
  \newtheorem{corollary}[subsection]{Corollary}
  \newtheorem{theorem}[subsection]{Theorem}
  
  \theoremstyle{definition}
  \newtheorem{definition}[subsection]{Definition}

\theoremstyle{remark}
  \newtheorem{remark}[subsection]{Remark}

  \newcounter{c}
  \renewcommand{\[}{\setcounter{c}{1}$$}
  \newcommand{\etyk}[1]{\vspace{-7.4mm}$$\begin{equation}\Label{#1}
  \addtocounter{c}{1}}
  \renewcommand{\]}{\ifnum \value{c}=1 $$\else \end{equation}\fi}
\begin{document}

\title{Multiplier Hopf monoids}

\author{Gabriella B\"ohm} 
\address{Wigner Research Centre for Physics, H-1525 Budapest 114,
P.O.B.\ 49, Hungary}
\email{bohm.gabriella@wigner.mta.hu}
\author{Stephen Lack}
\address{Department of Mathematics, Macquarie University NSW 2109, Australia}
\email{steve.lack@mq.edu.au}

\date{Nov 2015, revised May 2016}
  
\begin{abstract}
The notion of {\em multiplier Hopf monoid} in any braided monoidal category is
introduced as a multiplier bimonoid whose constituent fusion morphisms are
isomorphisms. In the category of vector spaces over the complex
numbers, Van Daele's definition of multiplier Hopf algebra is
re-obtained. It is shown that the key features of multiplier Hopf algebras
(over fields) remain valid in this more general context. Namely, for a
multiplier Hopf monoid $A$, the existence of a unique antipode is proved ---
in an appropriate, multiplier-valued sense --- which is shown to be a morphism
of multiplier bimonoids from a twisted version of $A$ to $A$. For a {\em
  regular} multiplier Hopf monoid (whose twisted versions are multiplier Hopf
monoids as well) the antipode is proved to factorize through a proper
automorphism of the object $A$. Under mild further assumptions, duals in the
base category are shown to lift to the monoidal categories of modules and of
comodules over a regular multiplier Hopf monoid. Finally, the so-called
Fundamental Theorem of Hopf modules is proved --- which states an equivalence 
between the base category and the category of Hopf modules over a multiplier
Hopf monoid. 
\end{abstract}
  
\maketitle


\section{Introduction}

For an arbitrary field $k$, the vector space of $k$-valued functions on a {\em
finite} group carries a natural {\em Hopf algebra} structure. For not
necessarily finite groups, one considers the vector space of functions of {\em
finite support}. Its structure is reminiscent of that of a Hopf algebra but
there are also some important differences; for example, the (pointwise)
multiplication does not admit a unit (the constant function with value one
does not have finite support). To give an algebraic description of this
situation, Van Daele introduced in \cite{VanDaele:multiplier_Hopf} the notion
of {\em multiplier Hopf algebra}.  

Classically, a Hopf algebra is defined as a particular type of {\em bialgebra} 
with some additional properties which are, in turn, equivalent to the
existence  of a so-called {\em antipode map}. Van Daele's definition of
multiplier Hopf algebra, however, has no version without antipode --- {\em
multiplier bialgebras} do not occur in his works. This notion was introduced
later in \cite{BohmGomezTorrecillasLopezCentella:wmba} (where it appeared
together with its `weak' generalization). Some different 
structures under the same name were discussed in 
\cite{JanVerc:Multiplier} and \cite{Tim:QG}. 

For many applications, it is necessary to consider bialgebras and Hopf
algebras (as well as various related structures) internal to a braided
monoidal category; the classical case of bialgebras or Hopf algebras over a
field corresponds to working in the symmetric monoidal category of vector
spaces over that field. 
    
In \cite{BohmLack:braided_mba} we initiated a program aiming to  develop
the theory of multiplier bialgebras and multiplier Hopf algebras in braided
monoidal categories. We only considered multiplier bialgebras in
\cite{BohmLack:braided_mba}; in this paper we turn to multiplier Hopf
algebras, still in the context of braided monoidal categories, although
sometimes we need to assume that the braiding is a symmetry. 
The resulting analysis applies to the classical case of vector 
spaces, and also to some other cases of interest: to the symmetric 
monoidal categories of group-graded vector spaces, modules over a 
commutative ring, complex Hilbert spaces and continuous linear 
maps, and complete bornological vector spaces; in the case of the 
last two examples we rely heavily on \cite{BohmGom-TorLack:wmbm}.   

First we study the algebraic properties of these multiplier Hopf 
algebras. Next we turn to duality in their monoidal categories of
modules and comodules.
In the particular category of vector spaces, modules over multiplier 
Hopf algebras appeared in \cite{DrabVDaeZhang} and comodules in  
\cite{VDaeZhang:CorpMultipI,KurVDaeZhang:CorpMultipII}. 
Finally we prove a Fundamental Theorem of Hopf Modules over
multiplier Hopf algebras, stating an equivalence between the category 
of Hopf modules and the base category. This generalizes a classical 
result on Hopf algebras over fields in \cite{LarrSwe:FTHM}. A study 
of (relative) Hopf modules for multiplier Hopf algebras over fields 
can be found in \cite{VDaeZhang:HopfGalMultiplier,DeCom:GalObAlgQGr}.

Working in a {\em closed} braided monoidal category, the {\em multiplier
monoid} of a (nice enough) semigroup can be defined: see
\cite{BohmLack:cat_of_mbm}. It generalizes the construction in 
\cite{Dauns:Multiplier} of multiplier algebras of non-unital algebras 
over fields. For a multiplier Hopf monoid $A$ in a closed
braided monoidal category, the antipode can be seen as a morphism from the 
opposite of $A$ to the multiplier monoid of $A$. In more general braided
monoidal categories such a simple interpretation is not available. Still, since
many of the expected properties of multiplier Hopf monoids can be proved at
this higher level of generality, we prefer to work in this setting whenever it
is possible.

{\bf Acknowledgement.} 
We gratefully acknowledge the financial support of the Hungarian Scientific
Research Fund OTKA (grant K108384), 
as well as the Australian Research Council Discovery Grant (DP130101969) and
an ARC Future Fellowship (FT110100385).   
The second named author is grateful for the warm hospitality of his hosts 
during visits to the Wigner Research Centre in Sept-Oct 2014 and Aug-Sept 2015.


\section{Background}

\subsection{The setting}\label{sect:assumptions}

In our earlier papers \cite{BohmLack:braided_mba} and
\cite{BohmLack:cat_of_mbm} we worked in a braided monoidal category \cc. In
this paper we shall often need to suppose that it is symmetric; this will
always be explicitly stated. The composite of morphisms $f\colon A\to B$ and
$g\colon B\to C$ will be denoted by $g.f\colon A\to C$. 
Identity morphisms will be denoted by $1$. 
The monoidal product will be denoted by juxtaposition, the monoidal unit by
$I$ and the braiding by $c$. For $n$ copies of the same object $A$, we also
use the power notation $AA\dots A=A^n$.
We denote by $\overline \cc$ the same monoidal category $\cc$ with the inverse
braiding $c^{-1}_{Y,X}\colon XY\to YX$; of course in the symmetric case this
has no effect, and so $\overline\cc$ is just \cc.  

As in \cite{BohmLack:cat_of_mbm}, we suppose given a class \cq of regular
epimorphisms in \cc which is closed under composition and monoidal product,
contains the isomorphisms, and is right-cancellative in the sense that if
$s\colon A\to B$ and $t.s\colon A\to C$ are in \cq, then so is $t\colon B\to
C$. Since each $q\in\cq$ is a regular epimorphism, it is the
coequalizer of some pair of  morphisms. Finally we suppose that
this pair may be chosen in such a way that the coequalizer is preserved by
taking the monoidal product with a fixed object. The main examples of interest
are where \cq consists of the split epimorphisms and where \cq consists of the
regular epimorphisms. In the latter case, we need to suppose that the regular
epimorphisms are closed under composition and under the monoidal product.

\subsection{Examples} \label{sect:ex_regepi}
We quote from \cite{BohmGom-TorLack:wmbm} some interesting examples 
of categories in which the assumptions of Section 
\ref{sect:assumptions} hold.  

Coequalizers exist in any abelian category. Since in this case any 
epimorphism is regular, the class of regular epimorphisms is closed 
under composition.  All coequalizers are preserved by the 
monoidal product whenever the category is also closed monoidal. This 
includes the example of the category of $G$-graded
{$k$}-modules, for any group $G$ and any commutative ring 
{$k$}; in particular, when $G$ is trivial this reduces to (ungraded)
{$k$}-modules, and when {$k$} is a field, to vector spaces. 

Although the closed symmetric monoidal category of complete 
bornological vector spaces \cite{Hogbe-Nlend,Meyer,Voigt} is not 
abelian, it still has coequalizers, and composites of regular 
epimorphisms are regular epimorphisms. 

The symmetric monoidal category of complex Hilbert spaces and 
continuous linear maps \cite{Kadison/Ringrose:1983} is neither closed 
nor abelian. 
It is nonetheless the case that it has coequalizers, and these are preserved by
the monoidal product, while regular epimorphisms are closed under
composition; the latter fact follows immediately from the fact that
all regular epimorphisms are split. 

\subsection{String diagrams}
\label{sect:strings}

For some proofs we find it convenient to use the string diagrams of
\cite{JoyalStreet:GTC}. These are composed down the page, with monoidal
products being taken from left to right. The braiding $c$ and its inverse are
represented by crossings, as below; the fact that they are mutually inverse
allows calculations as on the right. 
$$
\begin{tikzpicture} 
\path
(0,1) node [empty,name=in1] {}
(0.5,1) node [empty,name=in2] {}
(0,0) node [empty,name=out1] {}
(0.5,0) node[empty,name=out2] {};
\path [braid, name path=bin1] (in1) to[out=270,in=90] (out2);
\draw[braid, name path=bin2] (in2) to[out=270,in=90] (out1);  
\fill[white,name intersections={of=bin1 and bin2}] (intersection-1) circle(0.1);
\draw[braid] (in1) to[out=270,in=90] (out2);
\end{tikzpicture}\qquad 
\begin{tikzpicture} 
\path
(0,1) node [empty,name=in1] {}
(0.5,1) node [empty,name=in2] {}
(0,0) node [empty,name=out1] {}
(0.5,0) node[empty,name=out2] {};
\path [braid, name path=bin2] (in2) to[out=270,in=90] (out1);
\draw[braid, name path=bin1] (in1) to[out=270,in=90] (out2);  
\fill[white,name intersections={of=bin1 and bin2}] (intersection-1) circle(0.1);
\draw[braid] (in2) to[out=270,in=90] (out1);
\end{tikzpicture} \qquad
\begin{tikzpicture} 
 \path 
(0.8,0.5) node[empty] {=}
(1.1,0.5) node[empty] {}
(0,1) node[empty,name=in1] {}
(0.5,1) node[empty,name=in2] {}
(0,0) node[empty, name=out1] {}
(0.5,0) node[empty, name=out2] {}
(0,0.5) node[empty, name=x1] {}
(0.5,0.5) node[empty, name=x2] {};
\path[braid, name path=bin1] (in1) to[out=270,in=90] (x2);
\draw[braid, name path=bin2] (in2) to[out=270,in=90] (x1);
\fill[white,name intersections={of=bin1 and bin2}] (intersection-1) circle(0.1);
\draw[braid] (in1) to[out=270,in=90] (x2);
\draw[braid, name path=bout2] (x1) to[out=270,in=90] (out2);
\path[braid, name path=bout1] (x2) to[out=270,in=90] (out1);
\fill[white,name intersections={of=bout1 and bout2}] (intersection-1) circle(0.1);
\draw[braid] (x2) to[out=270,in=90] (out1);
\end{tikzpicture}
\begin{tikzpicture} 
 \path 
(0.8,0.5) node[empty] {=}
(1.1,0.5) node[empty] {}
(0,1) node[empty,name=in1] {}
(0.5,1) node[empty,name=in2] {}
(0,0) node[empty, name=out1] {}
(0.5,0) node[empty, name=out2] {};
\draw[braid] (in1) to[out=270,in=90] (out1);
\draw[braid] (in2) to[out=270,in=90] (out2);
\end{tikzpicture}
\begin{tikzpicture} 
 \path 
(0,1) node[empty,name=in2] {}
(0.5,1) node[empty,name=in1] {}
(0,0) node[empty, name=out2] {}
(0.5,0) node[empty, name=out1] {}
(0,0.5) node[empty, name=x2] {}
(0.5,0.5) node[empty, name=x1] {};
\path[braid, name path=bin1] (in1) to[out=270,in=90] (x2);
\draw[braid, name path=bin2] (in2) to[out=270,in=90] (x1);
\fill[white,name intersections={of=bin1 and bin2}] (intersection-1) circle(0.1);
\draw[braid] (in1) to[out=270,in=90] (x2);
\draw[braid, name path=bout2] (x1) to[out=270,in=90] (out2);
\path[braid, name path=bout1] (x2) to[out=270,in=90] (out1);
\fill[white,name intersections={of=bout1 and bout2}] (intersection-1) circle(0.1);
\draw[braid] (x2) to[out=270,in=90] (out1);
\end{tikzpicture}
$$
If an object $V$ has a left dual $\overline{V}$, we denote the unit
$\eta\colon I\to V\overline{V}$ and counit $\epsilon\colon \overline{V}V\to I$
of the duality using a ``cap'' and ``cup'' respectively, as in the following
diagram; the triangle equations for the duality say that we may ``straighten''
strings, as in the equations to the right.  
$$
\begin{tikzpicture}
\path
(0.0,0.5) node[empty, name=out1] {}
(1.0,0.5) node[empty, name=out2] {}
(0.5,1) node[empty, name=y] {}
(0.5,0) node[empty] {$\eta$};
\draw[braid] (out1) to[out=90,in=180] (y) to[out=0,in=90] (out2);
\end{tikzpicture}\quad\quad 
\begin{tikzpicture}
\path
(0,1) node[empty, name=in1] {}
(1,1) node[empty, name=in2] {}
(0.5,0.5) node[empty, name=x] {}
(0.5,0) node[empty] {$\epsilon$};
\draw[braid] (in1) to[out=270,in=180] (x) to[out=0,in=270] (in2);
\end{tikzpicture}
\quad\quad
\begin{tikzpicture}
\path 
(1.25,0.5) node[empty] {=}
(0,1) node[empty, name=in1] {}
(1,0) node[empty, name=out1] {}
(0.25,0.25) node[empty, name=x] {}
(0.75,0.75) node[empty, name=y] {};
\draw[braid] (in1) to[out=270,in=180] (x) to[out=0,in=180] (y) to[out=0,in=90] (out1);
\draw[braid] (1.5,1) to (1.5,0);
\end{tikzpicture}\quad\quad
\begin{tikzpicture}
\path 
(1.25,0.5) node[empty] {=}
(0,0) node[empty, name=in1] {}
(1,1) node[empty, name=out1] {}
(0.25,0.75) node[empty, name=x] {}
(0.75,0.25) node[empty, name=y] {};
\draw[braid] (in1) to[out=90,in=180] (x) to[out=0,in=180] (y) to[out=0,in=270] (out1);
\draw[braid] (1.5,1) to (1.5,0);
\end{tikzpicture}
$$

When reasoning about an associative multiplication, it will often be
represented by a joining of two strings, as on the left below, while the
counit for some comultiplication --- or more generally, for a fusion
morphism --- will be represented as on the right. 

$$
\begin{tikzpicture} 
\draw[braid] (.2,1) to[out=270,in=180] (.5,.3) to[out=0,in=270] (.8,1);
\draw[braid] (.5,.3) to (.5,0);
\end{tikzpicture} \quad\quad 
\begin{tikzpicture} 
\path
(0,1) node [empty,name=in1] {}
(0,0) node [empty] {}
(0,0.5) node [unit,name=e] {};
\draw[braid] (in1) to (e);
\end{tikzpicture}
$$

\subsection{Semigroups with non-degenerate multiplication} \label{sect:nondeg}

By a {\em semigroup} in a \linebreak
monoidal category $\cc$ we mean a pair $(A,m)$ consisting of an object $A$ of
$\cc$ and a morphism $m:A^2\to A$ -- called the {\em  multiplication} -- which
obeys the associativity condition $m.m1=m.1m$. If $m$ possesses a {\em unit}
-- that is, a morphism $u:I\to A$ such that $m.u1=1=m.1u$ -- we say that
$(A,m,u)$ is a {\em monoid}. 

The multiplication -- or, alternatively, the semigroup $A$ -- is said to be {\em
non-degenerate} if for any objects $X,Y$ of $\cc$, both maps
\begin{eqnarray*}
&\cc(X,YA) \to \cc(XA,YA),\qquad
&f\mapsto \xymatrix{XA\ar[r]^-{f1}& YA^2\ar[r]^-{1m} &YA}\quad \textrm{and}\\
&\cc(X,AY) \to \cc(AX,AY),\qquad
&g\mapsto \xymatrix{AX\ar[r]^-{1g}& A^2Y\ar[r]^-{m1} &AY}
\end{eqnarray*}
are injective. Evidently, any monoid is non-degenerate.  
 
\subsection{Examples} \label{sect:ex_non-deg} 
Our requirement in Section \ref{sect:nondeg} that these maps be 
injective for {\em any} object $Y$ is quite strong and can often
be avoided. A more nuanced analysis of non-degeneracy is carried out 
in the forthcoming \cite{BohmGom-TorLack:wmbm}. It is shown there 
that for most purposes it is enough if the injectivity condition in 
Section \ref{sect:nondeg} holds for any object $X$, and any
$Y$ belonging to a given class. Possible choices of this class in 
the examples of Section \ref{sect:ex_regepi} are also investigated in 
\cite{BohmGom-TorLack:wmbm}, and we now briefly recall the results of 
that investigation. 
  
In the category of (group graded) vector spaces over a given field, 
and also in the category of complex Hilbert spaces and 
continuous linear maps \cite{Kadison/Ringrose:1983}, the condition of Section
\ref{sect:nondeg} holds for any object $Y$ whenever it holds for $Y$ 
being the base field.

In the category of modules over a commutative ring $k$, the condition 
of Section \ref{sect:nondeg} holds for all locally projective 
\cite{ZimHui} modules $Y$ provided that it holds for the regular 
module $Y=k$. 

In the category of complete bornological vector spaces 
\cite{Hogbe-Nlend,Meyer,Voigt}, the condition of Section \ref{sect:nondeg} 
holds for all complete bornological vector spaces $Y$ with the 
so-called approximation property if it holds for $Y$ taken to be the 
base field.  

\subsection{\mM-morphisms} \label{sect:mM}

If $(B,m)$ is a non-degenerate semigroup and $A$ is any object, we define an
{\em \mM-morphism} from $A$ to $B$ to consist of morphisms $f_1\colon AB\to B$
and $f_2\colon BA\to B$ in \cc for which the first diagram below 
\begin{equation} \label{eq:mM}
\xymatrix{
BAB \ar[r]^{1f_1} \ar[d]_{f_21} & 
B^2 \ar[d]^m \\ 
B^2 \ar[r]_m & 
B }\quad
\xymatrix{
AB^2 \ar[r]^{1m} \ar[d]_{f_11} & 
AB \ar[d]^{f_1} \\ 
B^2 \ar[r]_m & 
B} \quad
\xymatrix{
B^2A \ar[r]^{1f_2} \ar[d]_{m1} & 
B^2 \ar[d]^m \\ 
BA \ar[r]_{f_2} & 
B }
\end{equation}
commutes; it follows by \cite[Lemma~2.8]{BohmLack:cat_of_mbm} that the
other two diagrams also commute.  
Given such an \mM-morphism, we write $f\colon A\nto B$, and call $f_1$ and
$f_2$ the {\em components} of $f$. Using the non-degeneracy of $B$, 
\mM-morphisms $A\nto B$ are seen to be equal if either their first components
coincide or their second components coincide.

If the braided monoidal category \cc is {\em closed} (meaning that for
any object $X$ the functor $X(-):\cc\to \cc$ is left adjoint) and it {\em has
pullbacks}, then one can define the {\em multiplier monoid} $\mM(B)$ of any
non-degenerate semigroup $B$. This is a universal object characterized by the
property that $\mM$-morphisms $A\nto B$ correspond bijectively to morphisms $A
\to\mM(B)$ in \cc: see \cite[Section 2.9]{BohmLack:cat_of_mbm}. For more
general braided monoidal categories \cc, no interpretation of $\mM$-morphisms
as morphisms in \cc is possible. 

An \mM-morphism is said to be {\em dense} if its components lie in \cq. If the
domain $A$ is also a semigroup, the \mM-morphism $f\colon A\nto B$ is {\em
  multiplicative} if either of the diagrams  
\begin{equation} \label{eq:multiplicative}
\xymatrix{
A^2B \ar[r]^{1f_1} \ar[d]_{m1} & 
AB \ar[d]^{f_1} \\ 
AB \ar[r]_{f_1} & 
B}\quad
\xymatrix{
BA^2 \ar[r]^{1m} \ar[d]_{f_2 1} & 
BA \ar[d]^{f_2} \\ 
BA \ar[r]_{f_2} & 
B }
\end{equation}
commutes; it then follows by \cite[Remark~2.7]{BohmLack:cat_of_mbm} that
both do. (If the multiplier monoid $\mM(B)$ exists, then this is
equivalent to the multiplicativity of the corresponding morphism $A\to
\mM(B)$ of \cc in the literal sense.) 

If $f\colon A\nto B$ is an $\mM$-morphism and $g\colon B\nto C$ a dense
multiplicative \mM-morphism, there is an induced \mM-morphism $g\bullet
f\colon A\nto C$ with components determined by commutativity of the following
diagrams 
\begin{equation}\label{eq:bullet-composition}
\xymatrix{
ABC \ar[r]^{1g_1} \ar[d]_{f_1 1} & AC \ar[d]^{(g\bullet f)_1} \\ BC
\ar[r]_{g_1} & C }\quad 
\xymatrix{
CA \ar[d]_{(g\bullet f)_2} & CBA \ar[l]_{g_2 1} \ar[d]^{1f_2} \\ C & CB \ar[l]^{g_2} }
\end{equation}
by virtue of the fact that $1g_1$ and $g_21$ are (regular) epimorphisms. If
$f$ is also dense and multiplicative, then so is $g\bullet f$: see
\cite[Proposition 3.2]{BohmLack:cat_of_mbm}. 

If $f$, $g$, and $h$ are composable \mM-morphisms, with $g$ and $h$ dense and
multiplicative, then the associativity condition $(h\bullet g)\bullet
f=h\bullet (g\bullet f)$ holds: see \cite[Proposition
3.3]{BohmLack:cat_of_mbm}. If the semigroup $B$ is not just non-degenerate
but has multiplication in \cq, then there is a dense multiplicative
\mM-morphism $i\colon B\nto B$ with both components equal to $m$. By
\cite[Proposition 3.4]{BohmLack:cat_of_mbm}, this acts as a two-sided identity
for the composition $\bullet$. 

In particular, there is a category \cm whose objects are the non-degenerate
semigroups with multiplication in \cq, and whose morphisms are the dense
multiplicative \mM-morphisms. 

If $f\colon A\nto B$ is an \mM-morphism and $z\colon Z\to A$ a morphism in
\cc, then there is a composite $f\circ z\colon Z\nto B$ with components  
$$\xymatrix{
ZB \ar[r]^{z1} & AB \ar[r]^{f_1} & B & BA \ar[l]_{f_2} & BZ \ar[l]_{1z} }$$
and now if $y\colon Y\to Z$ is a \cc-morphism and $g\colon B\nto C$ a dense
multiplicative \mM-morphism then also $(g\bullet f)\circ z=g\bullet(f\circ z)$
and $(f\circ z)\circ y=f\circ (z.y)$, and finally $f\circ 1=f$. 

For a non-degenerate semigroup $A$, a morphism $z\colon Z\to A$
in \cc induces an \mM-morphism $z^\#\colon Z\nto A$ with components  
$$\xymatrix{
ZA \ar[r]^{z1} & A^2 \ar[r]^{m} & A & A^2\ar[l]_{m} & AZ \ar[l]_{1z}}$$
and $f\circ z= f\bullet z^\#$ for any  multiplicative and dense
\mM-morphism $f\colon A\nto B$. 
 Consequently, for \cc-morphisms $\xymatrix@C=10pt{X \ar[r]^-x & Z
\ar[r]^-z &A}$, where $A$ is a non-degenerate semigroup,
$z^\#\circ x=(z.x)^\#$. Similarly, if $f\colon A\nto B$ is an \mM-morphism
and $g\colon B\to C$ a multiplicative isomorphism in \cc, then
$(g^\#\bullet f)_1=g.f_1.1g^{-1}$. 

\begin{remark}\label{rem:f_non-degenerate}
Let $\cc$ be a braided monoidal category. Consider an arbitrary 
semigroup  $A$ in $\cc$ and a semigroup $B$ whose multiplication is
non-degenerate. Let $f\colon A\nto B$ be an \mM-morphism, and suppose
that $f_2\colon BA\to B$ is an epimorphism which is preserved by 
taking the monoidal product with any object. Then two morphisms $g'$ and
$g'':X \to BY$ (for any objects $X$ and $Y$ of $\cc$) are equal if and
only if  
\begin{equation}\label{eq:s.f}
\xymatrix{AX\ar[r]^-{1g} & 
ABY\ar[r]^-{f_11}&
 BY}
\end{equation}
does not depend on $g\in \{g',g''\}$. Indeed, if \eqref{eq:s.f} does not depend
on $g$, then since 
$$
\xymatrix{
 BAX\ar[r]^-{11g} \ar[d]_-{f_21} &
 BABY \ar[r]^-{1f_11} \ar[d]_-{f_211} 
\ar@{}[rd]|-{~\eqref{eq:mM}} &
 B^2Y \ar[d]^-{m1} \\
 BX \ar[r]_-{1g} &
 B^2Y \ar[r]_-{m1} &
 BY}
$$
commutes, since $f_21$ is an epimorphism and since $m\colon B^2\to B$ is
non-degenerate, it follows that $g'=g''$.

Symmetrically, if $f_1$ is an epimorphism which is preserved under 
taking the monoidal product with any object, then two morphisms $h'$ and
$h'':X \to YB$ are equal if and only if  
$$
\xymatrix{XA\ar[r]^-{h1} & 
 YBA\ar[r]^-{1f_2}& 
 YB}
$$
does not depend on $h\in \{h',h''\}$. 
\end{remark}

Monoid morphisms in the usual sense are the same as multiplicative
$\mM$-morphisms whose components are (necessarily split) epimorphisms:

\begin{lemma} \label{lem:unit}
Let $A$ and $B$ be monoids and $f:A \nto B$ be a multiplicative
$\mM$-morphism. Then there is a unique multiplicative morphism
$h\colon A\to B$ such
that $f=h^\#$. Moreover, the following conditions are equivalent to each
other. 
\begin{itemize}
\item[{(a)}] The component $f_1\colon AB \to B$ is an epimorphism.
\item[{(b)}] The component $f_1\colon AB \to B$ is a split epimorphism.
\item[{(c)}] The component $f_2\colon BA \to B$ is an epimorphism.
\item[{(d)}] The component $f_2\colon BA \to B$ is a split epimorphism.
\item[{(e)}] $h$ is unital; that is, a morphism of monoids.
\end{itemize}
\end{lemma}

\begin{proof}
Denote by $u$ the units of the monoids $A$ and $B$, and put
$$
h:=\xymatrix{
A \ar[r]^-{1u} &
AB \ar[r]^-{f_1} &
B}.
$$
Commutativity of the diagrams
$$
\xymatrix{
AB \ar[r]_-{1u1} \ar@{=}@/_1.2pc/[rd] \ar@/^1.2pc/[rr]^-{h1} &
AB^2 \ar[r]_-{f_11} \ar[d]_-{1m} \ar@{}[rd]|-{\eqref{eq:mM}} &
B^2 \ar[d]^-m \\
& AB \ar[r]_-{f_1} &
B}\qquad
\xymatrix{
BA \ar[r]_-{11u} \ar[d]_-{f_2} \ar@/^1.2pc/[rr]^-{1h} &
BAB \ar[r]_-{1f_1} \ar[d]_-{f_21} \ar@{}[rd]|-{\eqref{eq:mM}} &
B^2 \ar[d]^-m \\
B \ar[r]^-{1u} \ar@{=}@/_1.2pc/[rr] &
B^2 \ar[r]^-m &
B}
$$
proves $f=h^\#$. Since the multiplication of $B$ has a unit, $h$ is clearly
unique with this property. By symmetry it admits an equal expression
$h=f_2.u1$. With the help of the first diagram above, also
$$
\xymatrix{
A^2 \ar[r]_-{11u} \ar[d]_-m \ar@/^1.2pc/[rr]^-{1h} &
A^2 B \ar[r]_-{1f_1} \ar[d]_-{m1} \ar@{}[rrd]|-{\eqref{eq:multiplicative}} &
AB \ar[r]^-{h1} \ar[rd]^-{f_1} &
B^2 \ar[d]^-m \\
A \ar[r]^-{1u}
\ar@/_1.4pc/[rrr]_-h &
AB \ar[rr]^-{f_1} &&
B}
$$
is seen to commute, proving the multiplicativity of $h$. 

Concerning the further assertions, (b)$\Rightarrow$(a) and (d)$\Rightarrow$(c)
are trivial. If (e) holds then the first diagram of 
$$
\xymatrix{
B \ar[d]_-{u1} \ar[rd]^-{u1}_-{(e)} \ar@{=}@/^1.4pc/[rrd] \\
AB \ar[r]^-{h1} \ar@/_1.2pc/[rr]_-{f_1} &
B^2 \ar[r]^-m &
B}
\qquad
\xymatrix{
AB \ar[r]_-{u11} \ar[d]_-{f_1} \ar@{=}@/^1.2pc/[rr] &
A^2B \ar[r]_-{m1} \ar[d]_-{1f_1} \ar@{}[rd]|-{\eqref{eq:multiplicative}}&
AB \ar[d]^-{f_1} \\
B \ar[r]_-{u1} &
AB \ar[r]_-{f_1} &
B}
\qquad
\xymatrix{
I \ar[rr]^-u \ar[d]_-u &&
B \ar@{=}[d] \ar[ld]_-{u1} \\
A \ar[r]^-{1u} \ar@/_1.2pc/[rr]_-{h} &
AB \ar[r]^-{f_1} &
B}
$$
commutes, proving (b). A symmetric proof yields (e)$\Rightarrow$(d). Finally,
if (a) holds then commutativity of the second diagram above implies that its
bottom row is the identity morphism (so that in particular (b)
holds). It then follows that the last diagram commutes, proving (e). The
implication (c)$\Rightarrow$(e) follows symmetrically. 
\end{proof}

\subsection{Monoidal structure} \label{sect:M-monoidal}

If $(A,m)$ and $(B,m)$ are non-degenerate semigroups, then so is their 
monoidal product $(AB,mm.1c1)$: see \cite[Proposition
  4.1]{BohmLack:cat_of_mbm}. If $A$ and $B$ have multiplication in \cq, 
then so does $AB$.

If $f\colon A\nto C$ and $g\colon B\nto D$ are \mM-morphisms, then there is an
\mM-morphism $fg\colon AB\to CD$ with components  
$$\xymatrix{
ABCD \ar[r]^-{1c1} & 
ACBD \ar[r]^-{f_1g_1} & 
CD & 
CADB \ar[l]_-{f_2g_2} & 
CDAB \ar[l]_-{1c1} }$$
and this will be dense and multiplicative if $f$ and $g$ are so: see 
\cite[Proposition 4.2]{BohmLack:cat_of_mbm}. This operation is
functorial with respect to the various compositions defined in Section
\ref{sect:mM} (see also \cite[Proposition 4.3]{BohmLack:cat_of_mbm}). 

For any object $X$ and non-degenerate semigroups $A$ and $B$ in a braided
monoidal category $\cc$, consider an \mM-morphism $f\colon X \nto AB$. By
the computation 
$$
\begin{tikzpicture}[scale=.9] 
\path (1.25,1) node[arr,name=f1] {$f_1$};
\draw[braid] (.5,2) to[out=270,in=135] (f1);
\draw[braid] (1,2) to[out=270,in=180] (1.25,1.5) to[out=0,in=270] (1.5,2);
\draw[braid] (1.25,1.5) to[out=270,in=90] (f1);
\draw[braid] (2,2) to [out=270,in=45] (f1);
\path[braid,name path=u] (f1) to[out=315,in=180] (2.5,.5) to [out=0,in=270]
(3,2); 
\draw[braid] (2.5,.5) to[out=270,in=90] (2.5,0);
\draw[braid,name path=d] (f1) to[out=225,in=180] (1.5,.5) to [out=0,in=270]
(2.5,2); 
\draw[braid] (1.5,.5) to[out=270,in=90] (1.5,0);
\fill[white,name intersections={of=u and d}] (intersection-1) circle(0.1);
\draw[braid] (f1) to[out=315,in=180] (2.5,.5) to [out=0,in=270] (3,2); 
\draw (3.5,1.1) node {$\stackrel {\eqref{eq:mM}}=$};
\end{tikzpicture}
\begin{tikzpicture}[scale=.9] 
\path (1.5,.5) node[arr,name=f1] {$f_1$};
\draw[braid] (.5,2) to[out=270,in=135] (f1);
\draw[braid] (1,2) to[out=270,in=180] (1.25,1.5) to[out=0,in=270] (1.5,2);
\draw[braid,name path=d] (1.25,1.5) to[out=270,in=180] (1.5,1) to[out=0,in=270] (2.5,2);
\draw[braid] (1.5,1) to[out=270,in=90] (f1);
\path[braid,name path=u] (2,2) to[out=270,in=180] (2.5,1) to [out=0,in=270]
(3,2); 
\draw[braid] (2.5,1) to[out=270,in=45] (f1);
\fill[white,name intersections={of=u and d}] (intersection-1) circle(0.1);
\draw[braid] (2,2) to[out=270,in=180] (2.5,1) to [out=0,in=270] (3,2); 
\draw[braid] (f1) to[out=225,in=90] (1,0); 
\draw[braid] (f1) to[out=315,in=90] (2,0); 
\draw (3.5,1.1) node {$\stackrel{\textrm{(ass)}}=$};
\end{tikzpicture}
\begin{tikzpicture}[scale=.9] 
\path (1.25,.5) node[arr,name=f1] {$f_1$};
\draw[braid] (.5,2) to[out=270,in=135] (f1);
\draw[braid,name path=d2] (1,2) to[out=270,in=180] (1.25,1) to[out=0,in=270] (2,1.5);
\draw[braid,name path=d1] (1.5,2) to[out=270,in=180] (2,1.5) to[out=0,in=270] (2.5,2);
\draw[braid] (1.25,1) to[out=270,in=90] (f1);
\path[braid,name path=u] (2,2) to[out=270,in=90] (1.5,1.5) to[out=270,in=180]
(2.5,1) to [out=0,in=270] (3,2); 
\draw[braid] (2.5,1) to[out=270,in=45] (f1);
\fill[white,name intersections={of=u and d1}] (intersection-1) circle(0.1);
\fill[white,name intersections={of=u and d2}] (intersection-1) circle(0.1);
\draw[braid] (2,2) to[out=270,in=90] (1.5,1.5) to[out=270,in=180]
(2.5,1) to [out=0,in=270] (3,2); 
\draw[braid] (f1) to[out=225,in=90] (.75,0); 
\draw[braid] (f1) to[out=315,in=90] (1.75,0); 
\draw (3.5,1.1) node {$\stackrel {\eqref{eq:mM}}=$};
\end{tikzpicture}
\begin{tikzpicture}[scale=.9] 
\path (1,1.3) node[arr,name=f1] {$f_1$};
\draw[braid] (.5,2) to[out=270,in=135] (f1);
\draw[braid] (1,2) to[out=270,in=90] (f1);
\draw[braid,name path=d1] (1.5,2) to[out=270,in=180] (2,1.3) to[out=0,in=270] (2.5,2);
\path[braid,name path=u1] (2,2) to [out=270,in=45] (f1);
\path[braid,name path=u2] (f1) to[out=315,in=180] (2.25,.5) to [out=0,in=270]
(3,2); 
\draw[braid] (2.25,.5) to[out=270,in=90] (2.25,0);
\draw[braid,name path=d2] (f1) to[out=225,in=180] (1,.5) to [out=0,in=270]
(2,1.3); 
\draw[braid] (1,.5) to[out=270,in=90] (1,0);
\fill[white,name intersections={of=u1 and d1}] (intersection-1) circle(0.1);
\fill[white,name intersections={of=u2 and d2}] (intersection-1) circle(0.1);
\draw[braid] (2,2) to [out=270,in=45] (f1);
\draw[braid] (f1) to[out=315,in=180] (2.25,.5) to [out=0,in=270] (3,2); 
\draw (3.5,1.1) node {$\stackrel{\textrm{(ass)}}=$};
\end{tikzpicture}
\begin{tikzpicture}[scale=.9] 
\path (1,1.5) node[arr,name=f1] {$f_1$};
\draw[braid] (.5,2) to[out=270,in=135] (f1);
\draw[braid] (1,2) to[out=270,in=90] (f1);
\draw[braid,name path=d1] (1.5,2) to[out=270,in=90] (2,1.5) to[out=270,in=0] (1,1) to[out=180,in=225] (f1);
\path[braid,name path=u1] (2,2) to [out=270,in=45] (f1);
\path[braid,name path=u2] (f1) to[out=315,in=180] (2.25,.5) to [out=0,in=270]
(3,2); 
\draw[braid] (2.25,.5) to[out=270,in=90] (2.25,0);
\draw[braid,name path=d2] (1,1) to[out=270,in=180] (1.25,.5) to [out=0,in=270]
(2.5,2); 
\draw[braid] (1.25,.5) to[out=270,in=90] (1.25,0);
\fill[white,name intersections={of=u1 and d1}] (intersection-1) circle(0.1);
\fill[white,name intersections={of=u2 and d1}] (intersection-1) circle(0.1);
\fill[white,name intersections={of=u2 and d2}] (intersection-1) circle(0.1);
\draw[braid] (2,2) to [out=270,in=45] (f1);
\draw[braid] (f1) to[out=315,in=180] (2.25,.5) to [out=0,in=270] (3,2); 
\end{tikzpicture}
$$
and non-degeneracy of $AB$, we conclude that the first component of $f$ 
renders commutative the first diagram of 
\begin{equation}\label{eq:lin_one_leg}
\xymatrix{
XA^2B \ar[r]^-{11c^{-1}} \ar[d]_-{1m1} &
XABA \ar[r]^-{f_11} &
ABA \ar[r]^-{1c} &
A^2B \ar[d]^-{m1} \\
XAB \ar[rrr]_-{f_1} &&&
AB}
\quad 
\xymatrix{
XAB^2 \ar[r]^-{f_11} \ar[d]_-{11m} &
AB^2 \ar[d]^-{1m} \\
XAB \ar[r]_-{f_1} &
AB.}
\end{equation}
The second diagram commutes by similar considerations and there are analogous
identities for the second component.  

\subsection{Opposites} \label{sect:op}

For a more abstract approach to the calculations in this section, see
Remark~\ref{rmk:symmetry} below.  

If $(B,m)$ is a semigroup, there is a semigroup $(B, m.c^{-1} )$
which we call $B\op$; clearly it is non-degenerate if and only if $B$ is so.  

\begin{proposition}\label{prop:op}
If $f\colon A\nto B$ is an \mM-morphism, there is an \mM-morphism $f\op\colon
A\nto B\op$ with components $f\op_1$ and $f\op_2$ given by the following
composites. 
$$\xymatrix{
AB \ar[r]^-{c^{-1}} & 
BA \ar[r]^-{f_2} & 
B & 
AB \ar[l]_-{f_1} & 
BA \ar[l]_-{c^{-1}} }$$
This is dense if and only if $f$ is so. If $f$ is multiplicative,
then $f\op$ is a multiplicative \mM-morphism $f\op\colon A\op\nto B\op$.  The
passage from $f$ to $f\op$ is functorial. 
\end{proposition}

\proof
The fact that $f\op$ is an \mM-morphism follows from commutativity of the
diagram on the left below. 
$$\xymatrix{
BAB \ar[rr]^{1c^{-1}} \ar[dd]_{c^{-1}1} && 
B^2A \ar[r]^{1f_2} \ar[d]^{c^{-1}1} & 
B^2 \ar[dd]^{c^{-1}} \\
&& B^2 A \ar[d]^{1c^{-1}} \\
AB^2 \ar[r]^{1 c^{-1}} \ar[d]_{f_11} & 
AB^2 \ar[r]^{c^{-1} 1} & 
BAB  \ar[r]^{f_21} \ar[d]^{1f_1} \ar@{}[rd]|-{\eqref{eq:mM}}& 
B^2 \ar[d]^{m} \\
B^2 \ar[rr]_{c^{-1}} && 
B^2 \ar[r]_{m} & 
B }\quad 
\xymatrix{
A^2B \ar[rr]^{1c^{-1}} \ar[dd]_{c^{-1}1} && 
ABA \ar[r]^{1f_2} \ar[d]^{c^{-1}1} & 
AB \ar[dd]^{c^{-1}} \\
&& BA^2 \ar[d]^{1c^{-1}} \\ 
A^2B \ar[d]_{m1} \ar[r]^{1c^{-1}} & 
ABA \ar[r]^{c^{-1}1} & 
BA^2 \ar[r]^{f_21} \ar[d]^{1m} \ar@{}[rd]|-{\eqref{eq:multiplicative}}& 
BA \ar[d]^{f_2} \\
AB \ar[rr]_{c^{-1}} && 
BA \ar[r]_{f_2} & 
B }
$$
The fact that $f\op$ is multiplicative if $f$ is so follows from commutativity
of the diagram on the right; the corresponding fact about $f\op$ being dense
is obvious. Whenever $f$ is dense, the top rows of the following
diagrams are epimorphisms. Hence functoriality follows from the equality of
the left-bottom paths in the commutative diagrams
$$
\xymatrix{
ACB \ar[rrr]^-{1f_2} \ar[d]_-{c^{-1} 1} &&&
AC \ar[dd]^-{c^{-1}}
&
ACB \ar[rr]^-{1f_2} &&
AC \ar[ddd]^-{(f^{\mathsf{op}} \bullet g^{\mathsf{op}})_1} \\
CAB \ar[d]_-{1c^{-1}} &&&&
ABC \ar[u]^-{1c^{-1}} \ar[d]_-{c^{-1} 1} \\
CBA \ar[rrr]^-{f_2 1} \ar[d]_-{1 g_2} &&&
CA \ar[d]^-{(f\bullet g)^{\mathsf{op}}_2}
&
BAC \ar[d]_-{g_2 1} \\
CB \ar[rrr]_-{f_2} &&&
C
&
BC \ar[r]_-{c^{-1}} &
CB \ar[r]_-{f_2} &
C}
$$ 
\endproof

\begin{remark}\label{rmk:symmetry}
Write $\cc\rev$ for the monoidal category with the same underlying category 
\cc, but with the reverse monoidal structure, so that the monoidal product of
$A$ and $B$ in $\cc\rev$ is $BA$. The inverse-braiding $c^{-1}$ may be used to
make the identity functor on \cc into a strong monoidal isomorphism
$\cc\rev\to\cc$; such a functor sends semigroups to semigroups, and in this
case sends $A$ to $A\op$. Furthermore, the strong monoidal isomorphism
$(1,c^{-1})\colon \cc\rev\to\cc$ is in fact braided (if we equip $\cc\rev$
with the braiding $c$), from which it follows that the induced functor
$\cm\rev\to\cm$ is also strong monoidal, with the binary part of the strong
monoidal structure being given by $(c^{-1})^\#$. 
\end{remark}

If $A$ and $B$ are semigroups, then the multiplication of $A\op B\op$ is given
by the composite  
$$\xymatrix@C=35pt{
(AB)^2 \ar[r]^-{1c1} & 
A^2B^2 \ar[r]^-{c^{-1}c^{-1}} & 
A^2B^2 \ar[r]^-{mm} & 
AB }$$
while the multiplication of $(AB)\op$ is the composite
$$\xymatrix@C=35pt{
(AB)^2 \ar[r]^-{1c^{-1}1} & 
A^2B^2 \ar[r]^-{c^{-1}c^{-1}} & 
A^2B^2 \ar[r]^-{mm} & 
AB}$$
and clearly these are in general different, but agree if the braiding is a
symmetry. A similar argument works for \mM-morphisms, giving: 

\begin{proposition}\label{prop:monoidal-op}
If the braiding of \cc is a symmetry, then for semigroups $A$ and $B$ we have
$A\op B\op=(AB)\op$, and similarly for \mM-morphisms $f$ and $g$ we have $f\op
g\op=(fg)\op$.
\end{proposition}

In the case of an arbitrary braiding, for any \cc-morphism $z\colon
X\to Z$ and \mM-morphisms $f\colon Z\nto A$ and $g\colon Y\nto B$, the
\cc-morphism $c\colon (AB)\op \to B\op A\op$ is an isomorphism of semigroups
in \cc and the equalities 
\begin{equation} \label{eq:op-mon}
(f\circ z)\op=f\op\circ z
\quad \textrm{and}\quad
g\op f\op \circ c = c^\# \bullet (fg)\op
\end{equation}
hold. 


\section{Multiplier Hopf monoids}

\subsection{Multiplier bimonoids}\label{sect:mbm}

A {\em multiplier bimonoid} \cite{BohmLack:braided_mba} 
in a braided monoidal category $\cc$ is an object $A$ equipped with morphisms
$t_1:A^2\to A^2$, $t_2:A^2\to A^2$ and $e:A\to I$ subject to the axioms
encoded in the following commutative diagrams (see
\cite{BohmLack:braided_mba}).   
\begin{eqnarray}
\xymatrix{
A^3 \ar[r]^-{1t_1} \ar[d]_-{t_11}&
A^3 \ar[r]^-{c1} &
A^3 \ar[r]^-{1t_1} &
A^3 \ar[r]^-{c^{-1}1} &
A^3 \ar[d]^-{t_11} 
&&
A^2 \ar[rd]^-{1e} \ar[d]_-{t_1} 
\\
A^3 \ar[rrrr]_-{1t_1} &&&&
A^3
&&
A^2 \ar[r]_-{1e} &
A}\label{eq:mbm_ax_1}\\
\xymatrix{A^3 \ar[r]^-{t_21} \ar[d]_-{1t_2} &
A^3 \ar[r]^-{1c} &
A^3 \ar[r]^-{t_21} &
A^3 \ar[r]^-{1c^{-1}} &
A^3 \ar[d]^-{1t_2} 
&&
A^2 \ar[rd]^-{e1} \ar[d]_-{t_2} 
\\
A^3 \ar[rrrr]_-{t_21} &&&&
A^3
&&
A^2 \ar[r]_-{e1} &
A}\label{eq:mbm_ax_2}\\
\xymatrix{
A^3\ar[r]^-{t_21} \ar[d]_-{1t_1} &
A^3 \ar[d]^-{1t_1}
&&&&&
A^2 \ar[r]^-{t_1}\ar[d]_-{t_2} 
&
A^2 \ar[d]^-{e1} \\
A^3 \ar[r]_-{t_21} &
A^3
&&&&&
A^2 \ar[r]_-{1e} &
A}\label{eq:mbm_ax_compatibility}
\end{eqnarray}
The equalities encoded in the first diagrams of \eqref{eq:mbm_ax_1} and
\eqref{eq:mbm_ax_2} are referred to as the {\em fusion equations} and the
equalities expressed by the second diagrams are termed the {\em counit
conditions}.  

The common diagonal of the second diagram of \eqref{eq:mbm_ax_compatibility}
is an associative multiplication to be denoted by $m$. Postcomposing the equal
paths around the second diagram of \eqref{eq:mbm_ax_1} (or of
\eqref{eq:mbm_ax_2}) with $e$, we deduce the commutativity of 
\begin{equation}\label{eq:e_multiplicative}
\xymatrix{
A^2 \ar[r]^-{1e} \ar[d]_-m &
A \ar[d]^-e \\
A \ar[r]_-e &
I.}
\end{equation}
Postcomposing the equal paths around the first diagram of \eqref{eq:mbm_ax_1}
with $e11$, and postcomposing the equal paths around the first diagram of
\eqref{eq:mbm_ax_2} with $11e$, we obtain the so-called {\em short fusion
  equations} 
\begin{equation}\label{eq:short_fusion}
\xymatrix@C=18pt{
A^3\ar[d]_-{m1} \ar[r]^-{1t_1} &
A^3 \ar[r]^-{c1} &
A^3 \ar[r]^-{1t_1} &
A^3 \ar[r]^-{c^{-1}1} &
A^3 \ar[d]^-{m1}
&
A^3\ar[d]_-{1m} \ar[r]^-{t_21} &
A^3 \ar[r]^-{1c} &
A^3 \ar[r]^-{t_21} &
A^3 \ar[r]^-{1c^{-1}} &
A^3 \ar[d]^-{1m} \\
A^2 \ar[rrrr]_-{t_1} &&&&
A^2
&
A^2 \ar[rrrr]_-{t_2} &&&&
A^2 .}
\end{equation}
Postcomposing by $1e1$ the equal paths around the first diagram of
\eqref{eq:mbm_ax_1}, we obtain 
\begin{equation}\label{eq:t_1_mod_map}
\xymatrix{
A^3 \ar[r]^{t_11} \ar[d]_{1m} & A^3 \ar[d]^{1m} \\ A^2 \ar[r]_{t_1} & A^2. }
\end{equation}
Finally, postcomposing the equal paths around the first diagram of
\eqref{eq:mbm_ax_compatibility} by $1e1$, we obtain
\begin{equation}\label{eq:short_compatibility}
\xymatrix{
A^3 \ar[r]^-{t_21} \ar[d]_-{1t_1} &
A^3 \ar[d]^-{1m} \\
A^3 \ar[r]_-{m1} &
A^2 .}
\end{equation}

In this paper, we only consider multiplier bimonoids $(A,t_1,t_2,e)$
with additional properties: 
\begin{itemize}
\item {\em non-degeneracy of the multiplication} $m:=e1.t_1=1e.t_2$,
\item {\em surjectivity of the multiplication}; in the sense that $m\in \cq$,
\item {\em density of the comultiplication} \cite{BohmLack:cat_of_mbm}; that
is, the property that the components
\begin{equation}\label{eq:d_1-2}
\xymatrix@R=8pt{
d_1:=A^3\ar[r]^-{c1} &
A^3 \ar[r]^-{1t_1} &
A^3 \ar[r]^-{c^{-1}1} &
A^3 \ar[r]^-{m1} &
A^2 \quad \textrm{and}\\
d_2:=A^3\ar[r]^-{1c} &
A^3 \ar[r]^-{t_21} &
A^3 \ar[r]^-{1c^{-1}} &
A^3 \ar[r]^-{1m} &
A^2 \quad {\color{white} \textrm{and}}}
\end{equation} 
of the multiplicative \mM-morphism $d\colon A\nto A^2$ belong to \cq; 
\item {\em density of the counit}; that is, that the counit $e\colon A\to I$
is in \cq, and so can be seen as a dense multiplicative \mM-morphism $A\nto I$. 
\end{itemize} 
Under these assumptions some of the above axioms become redundant, see
\cite{BohmLack:braided_mba}. 

Multiplier bimonoids with these properties were identified in \cite[Theorem
5.1]{BohmLack:cat_of_mbm} with certain comonoids in the monoidal category \cm
described in Section~\ref{sect:mM}.    

Multiplier bimonoids are related to usual bimonoids as follows.

\begin{proposition} \label{prop:unit}
For a multiplier bimonoid $(A,t_1,t_2,e)$, the following assertions are
equivalent.
\begin{itemize}
\item[{(a)}] The multiplication $m:=e1.t_1=1e.t_2$ admits some unit $u$; and
  the counit $e$, as well as the morphisms $d_1$ and $d_2$ of
  \eqref{eq:d_1-2}, are epimorphisms. 
\item[{(b)}] There is a bimonoid $A$ with some monoid structure $(m,u)$, some
  comultiplication $h$, and the given counit $e$, such that $t_1=1m.h1$ and
  $t_2=m1.1h$. 
\end{itemize}
\end{proposition}

\begin{proof}
Assume that (b) holds. Then both expressions $e1.t_1$ and $1e.t_2$ are equal
to the multiplication $m$ of the monoid $A$, which has a unit $u$ by definition.
Being unital, the counit $e$ is a split epimorphism. 
The left-bottom path of the commutative diagram 
$$
\xymatrix{
A^3 \ar[r]^-{h11} \ar[d]_-{1c^{-1}} &
A^4 \ar[rr]^-{1c1} \ar[d]^-{11c^{-1}} &&
A^4 \ar[dd]^-{mm} \\
A^3 \ar[r]^-{h11} \ar@/_1.5pc/[rd]_-{t_11} &
A^4 \ar[d]^-{1m1} \\
& A^3 \ar[r]_-{1c} &
A^3 \ar[r]_-{m1} &
A^2}
$$
is the $d_1$ of \eqref{eq:d_1-2}, proving $d=h^\#$.
So it follows by Lemma \ref{lem:unit} (e)$\Rightarrow$(a) and (c) that $d_1$
and $d_2$ are epimorphisms so that (a) holds. 

Conversely, assume that (a) holds. Then applying Lemma \ref{lem:unit}
(a)$\Rightarrow$(e) to the multiplicative $\mM$-morphism $(e,e)=e^\#:A \nto I$
we see that $e:A \to I$ is a monoid morphism. Also by Lemma \ref{lem:unit}
the multiplicative $\mM$-morphism $d:A\nto A^2$ is of the form $h^\#$ for the
unique monoid morphism  
$$
h=\xymatrix{
A \ar[r]^-{1uu} & 
A^3 \ar[r]^-{d_1} &
A^2= A \ar[r]^-{1u} &
A^2 \ar[r]^-{t_1} &
A^2}
$$ 
which admits the equal form $h=t_2.u1$. It renders commutative both diagrams
$$
\xymatrix{
A^2 \ar[r]_-{u11} \ar[d]_-{t_1} \ar@/^1.2pc/[rr]^-{h1} &
A^3 \ar[r]_-{t_21} \ar[d]_-{1t_1}
\ar@{}[rd]|-{\eqref{eq:mbm_ax_compatibility}\quad} & 
A^3 \ar[d]_-{1t_1} \ar@/^1.2pc/[dd]^-{1m} \\
A^2 \ar[r]^-{u11} \ar@{=}@/_1.5pc/[rd] &
A^3 \ar[d]^-{m1} \ar[r]^-{t_21} &
A^3 \ar[d]_-{1e1} \\
& A^2 \ar@{=}[r] &
A^2}\qquad
\xymatrix{
A^2 \ar[r]_-{11u} \ar[d]_-{t_2}\ar@/^1.2pc/[rr]^-{1h} &
A^3 \ar[r]_-{1t_1} \ar[d]_-{t_21}
\ar@{}[rd]|-{\eqref{eq:mbm_ax_compatibility}\quad} & 
A^3 \ar[d]_-{t_21} \ar@/^1.2pc/[dd]^-{m1} \\
A^2 \ar[r]^-{11u} \ar@{=}@/_1.5pc/[rd] &
A^3 \ar[d]^-{1m} \ar[r]^-{1t_1} &
A^3 \ar[d]_-{1e1} \\
& A^2 \ar@{=}[r] &
A^2}
$$
as needed for (b) to hold. Coassociativity and counitality of $h$ follow by
the commutativity of the diagrams
\[
\xymatrix{
A \ar[r]_-{1u} \ar[d]^-{u1} \ar@/^1.2pc/[rr]^-h \ar@/_1.2pc/[dd]_-h &
A^2 \ar[r]_-{t_1} &
A^2 \ar[d]_-{u11} \ar@/^1.2pc/[dd]^-{h1} \\
A^2 \ar[r]^-{11u} \ar[d]^-{t_2} &
A^3 \ar[r]^-{1t_1} \ar[d]_-{t_21}
\ar@{}[rd]|-{\eqref{eq:mbm_ax_compatibility}\quad} &
A^3 \ar[d]_-{t_21} \\
A^2 \ar[r]^-{11u} \ar@/_1.2pc/[rr]_-{1h} &
A^3 \ar[r]^-{1t_1} &
A^3}\qquad
\xymatrix@C=15pt{
&& A  \ar@{=}[dd] \ar[dl]^-{1u} \ar[rd]_-{u1} 
\ar@/_1.4pc/[lldd]_-h \ar@/^1.4pc/[rrdd]^-{h} \\
& A^2 \ar[rd]^-m \ar[ld]^-{t_1} &&
A^2 \ar[rd]_-{t_2} \ar[ld]_-m \\
A^2 \ar[rr]_-{e1} &&
A  &&
A^2 . \ar[ll]^-{1e}}
\]
\end{proof}

\subsection{Twisting multiplier bimonoids} \label{claim:twist}
It is immediate from the definition that any multiplier bimonoid
$(A,t_1,t_2,e)$ determines another multiplier bimonoid $(A,c.t_2.c^{-1},$
$c.t_1.c^{-1},e)$, and indeed this construction can be iterated (although
if the braiding is a symmetry then further iterates do not yield anything
new).

We can also understand this, in the case when the assumptions in Section
\ref{sect:mbm} hold, in terms of the corresponding comonoids in \cm. By the
strong monoidality of the functor $(-)\op \colon \cm\rev\to\cm$, if
$C=(A,d,e)$ is a comonoid in \cm, then $A\op$ also becomes a comonoid with
comultiplication    
$$\xymatrix@C=7pt{
A\op \ar[rr]^-{d\op} &
{\raisebox{5pt}{${}_{{}_{/}}$}} &
(AA)\op \ar[rr]^{c^\#} &
{\raisebox{5pt}{${}_{{}_{/}}$}} & 
A\op A\op }$$
and counit  $e\op\colon A\op\nto I$. The components of the comultiplication
are given by the composites  
$$\xymatrix{
A^3 \ar[r]^{1c^{-1}} & A^3 \ar[r]^{d\op_1} & A^2 \ar[r]^{c} & 
A^2 & A^2 \ar[l]_{c} & A^3 \ar[l]_{d\op_2} & A^3 . \ar[l]_{c^{-1}1} }$$

Regard now a multiplier bimonoid $A$  (having the properties listed in Section
\ref{sect:mbm}) as a comonoid in \cm; then $d_1$ and $d_2$ are  defined as  in 
\eqref{eq:d_1-2}. The multiplication for the twisted multiplier bimonoid is
given by 
$$\xymatrix{
A^2 \ar[r]^{c^{-1}} & 
A^2 \ar[r]^{t_2} \ar[d]_{m} & 
A^2 \ar[r]^{c} \ar[d]^{1e} & 
A^2 \ar[d]^{e1} \\
&
A \ar@{=}[r] &
A \ar@{=}[r] & 
A }$$
and so the underlying semigroup is $A\op$.

The counit is just $e\colon A\to I$, which we can equally regard as
$e\op\colon A\op\nto I$.  
The first component of the comultiplication is given by the upper composite of
the diagram  
$$\xymatrix{
A^3 \ar[r]^{1c^{-1}} & 
A^3 \ar[r]^{c^{-1}1} \ar[dd]_{d\op_1} & 
A^3 \ar[r]^{t_21} \ar[d]^{1c^{-1}} & 
A^3 \ar[r]^{c1} \ar[d]^{1c^{-1}} & 
A^3 \ar[r]^{1c} & 
A^3 \ar[d]^{c^{-1}1} \\
&& A^3 \ar[d]^{d_2} & 
A^3 \ar[d]^{1m} && 
A^3 \ar[d]^{m1} \\
& A^2 \ar@{=}[r] & 
A^2 \ar@{=}[r] & 
A^2 \ar[rr]_{c} && 
A^2 }$$
which commutes by naturality of the braiding and definition of $d_2$ and
$d\op_1$. But the lower composite is the first component of $c^\# \bullet
d\op$ which is therefore the comultiplication. 

Thus the twisted multiplier bimonoid corresponds to reversing both the
multiplication and the comultiplication, as one might expect.  

\subsection{Morphisms between multiplier bimonoids} \label{claim:morphism} 
Let $A$ and $B$ be multiplier bimonoids in a braided monoidal category
$\cc$ having the properties listed in Section \ref{sect:mbm}. Following
\cite{BohmLack:cat_of_mbm}, a morphism between them is defined to be  a
morphism between the corresponding comonoids in the category \cm of
Section~\ref{sect:mM}. Explicitly, this means a  dense and
multiplicative \mM-morphism $f\colon A\nto B$ whose components render
commutative the following diagrams.
\begin{equation} \label{eq:mbm_morphism}
 \xymatrix@C=20pt{
 AB \ar[r]^-{f_1} \ar[rdd]_-{ee} &
 B \ar[dd]^-{e} 
&
 A^2 B^2 \ar[r]^-{1f_11} \ar[d]_-{t_2t_1} &
 AB^2 \ar[d]^-{1t_1} 
&
 BA \ar[r]^-{f_2} \ar[rdd]_-{ ee} &
 B \ar[dd]^-{ e} 
&
 B^2A^2 \ar[r]^-{1f_21}\ar[d]_-{ t_2t_1} &
 B^2A \ar[d]^-{ t_21} \\
&&
 A^2B^2 \ar[d]_-{1c1} &
 AB^2 \ar[d]^-{f_11}
&&&
 B^2A^2 \ar[d]_-{1c1} &
 B^2A \ar[d]^-{1f_2}\\
&I
&
 (AB)^2 \ar[r]_-{f_1f_1} &
 B^2 &
&I
&
 (BA)^2 \ar[r]_-{f_2f_2} &
 B^2 }
\end{equation}
 In fact, since the counit  $e$ of $B$  is dense, the first
and the third diagrams are equivalent to each other; and since the
comultiplication of $ B$ is dense, also the second and the last diagrams
are equivalent to each other;  see
\cite[Section~6.2]{BohmLack:cat_of_mbm}. 

\subsection{Regular multiplier bimonoids}
A {\em regular multiplier
bimonoid} \cite{BohmLack:braided_mba} in a braided monoidal category $\cc$ is
a tuple $(A,t_1,t_2,t_3,t_4,e)$ such that $(A,t_1,t_2,e)$ is a multiplier
bimonoid in $\cc$ and $(A,t_3,t_4,e)$ is a multiplier bimonoid in $\overline
\cc$ such that the following diagrams commute, where $m$ stands for the common
diagonal of the first diagram.   
\begin{equation}\label{eq:reg_mbm}
\xymatrix{
A^2 \ar[r]^-{t_1} \ar[d]^-c &
A^2 \ar[dd]_-{e1}
&
A^3 \ar[r]^-{1t_1} \ar[d]^-{c1} &
A^3 \ar[d]_-{c1}
&
A^3 \ar[r]^-{1t_1} \ar[dd]^-{t_41} &
A^3 \ar[dd]_-{t_41}
&
A^3 \ar[r]^-{t_21} \ar[d]^-{1c} &
A^3 \ar[d]_-{1c} 
&
A^3 \ar[r]^-{t_21} \ar[dd]^-{1t_3} &
A^3 \ar[dd]_-{1t_3} \\
A^2 \ar[d]^-{t_3} &
&
A^3 \ar[d]^-{t_31} &
A^3 \ar[d]_-{1m}
&
&&
A^3 \ar[d]^-{1t_4} &
A^3 \ar[d]_-{m1} \\
A^2 \ar[r]_-{e1} &
A
&
A^3 \ar[r]_-{1m} &
A^2
&
A^3 \ar[r]_-{1t_1} &
A^3
&
A^3 \ar[r]_-{m1} &
A^2
&
A^3 \ar[r]_-{t_21} &
A^3}
\end{equation}
Note that the first and the last diagrams, and the first and the third
diagrams, respectively, imply the commutativity of
\begin{equation}\label{eq:t_2-3_compatibility}
\xymatrix{
A^3 \ar[r]^-{t_21} \ar[d]_-{1t_3} &
A^3 \ar[r]^-{1c^{-1}} &
A^3 \ar[d]^-{1m} 
&
A^3 \ar[d]_-{t_41} \ar[r]^-{1t_1} &
A^3 \ar[r]^-{c^{-1}1} &
A^3 \ar[d]^-{m1} \\
A^3 \ar[rr]_-{m1} &&
A^2
&
A^3 \ar[rr]_-{1m} &&
A^2\ , }
\end{equation}
from which it follows that if the multiplication is non-degenerate, there can
be at most one choice of morphisms $t_3$ and $t_4$ making it into a regular
multiplier bimonoid; thus for a multiplier bimonoid with non-degenerate
multiplication, being regular is a property rather than further structure.  
Applying \eqref{eq:t_1_mod_map} to the multiplier bimonoid
$(A,t_3,t_4,e)$ in $\overline \cc$, we deduce the commutativity of
\begin{equation} \label{eq:t_3_mod_map}
\xymatrix{
A^3 \ar[r]^-{1c^{-1}} \ar[d]_-{t_31} &
A^3 \ar[r]^-{1m} &
A^2 \ar[d]^-{t_3} \\
A^3 \ar[r]_-{1c^{-1}} &
A^3 \ar[r]_-{1m} &
A^2 \ .}
\end{equation}

The multiplication $m$ of the multiplier bimonoid $(A,t_1,t_2,e)$ in $\cc$ is
non-degenerate if and only if the multiplication $m.c^{-1}$ of the multiplier
bimonoid $(A,t_3,t_4,e)$ in $\overline \cc$ is non-degenerate. In this case
many of the above axioms become redundant: see \cite[Proposition
3.11]{BohmLack:braided_mba}. 

The comultiplication of the multiplier bimonoid $(A,t_3,t_4,e)$ in $\overline
\cc$ is obtained from  the comultiplication $d$ of the multiplier bimonoid
$(A,t_1,t_2,e)$ in $\cc$  as $d\op$; hence it is dense if and only if $d$ is
dense, and hence one multiplier bimonoid will satisfy the conditions of 
Section~\ref{sect:mbm} if and only if the other does so. 

Various possible reformulations of the definition of morphism of multiplier
bimonoid are possible in the regular context; we shall need the following
result. 

\begin{proposition}\label{prop:multiplier-bimonoid-t3}
Let $(A,t_1,t_2,t_3,t_4,e)$ and  $(B,t_1,t_2,t_3,t_4,e)$  be
regular multiplier bimonoids having the properties listed in Section
\ref{sect:mbm}. 
If $f$ is a morphism of multiplier bimonoids from
$(A,t_1,t_2,e)$ to  $(B,t_1,t_2,e)$, then the following diagram
commutes.    
\begin{equation}\label{eq:multiplier-bimonoid-t3}
\xymatrix{
A^2 \ B^2 \ar[r]^{t_3t_1} \ar[d]_{c^{-1}11} & 
A^2  B^2 \ar[r]^{1c1} & 
(AB)^2 \ar[r]^{f_1f_1} & 
 B^2   \\
A^2 B^2  \ar[r]_{1f_11} & 
A  B^2  \ar[r]_{1 t_1} & 
A B^2   \ar[r]_{c1} & 
 B  A  B \ar[u]_{1f_1} 
}\end{equation}
\end{proposition}

\begin{proof}
We prove this using string diagrams, as in Section~\ref{sect:strings}. Then
the claim follows by the calculation  
\smallskip

{\color{white} .}\hspace{-.8cm}
\begin{tikzpicture}[scale=.935] 
\path
(0,3) node[empty, name=in1] {}
(0.5,3) node[empty, name=in2] {}
(1,3) node[empty, name=in3] {}
(1.5,3) node[empty, name=in4] {}
(2,3) node[empty, name=in5] {}
(0.5,0) node[empty, name=out1] {}
(1.5,0) node[empty, name=out2] {}
(1.2,2.5) node[arr, name=f1u] {$f_1$}
(1.5,1.5) node[arr, name=t1] { $t_1$}
(0.5,0.5) node[arr, name=f1l] {$f_1$}
(1.5,0.5) node[arr, name=f1r] {$f_1$};
\draw[braid] (in1) to[out=270,in=180] (f1l);
\draw[braid, name path=bin2] (in2) to[out=270,in=135] (f1u);
\path[braid, name path=bin3] (in3) to[out=225, in=135] (f1r);
\draw[braid] (in4) to[out=270,in=45] (f1u);
\draw[braid] (in5) to[out=270,in=45] (t1);
\draw[braid] (f1u) to[out=270,in=135] (t1);
\draw[braid] (t1) to[out=315,in=45] (f1r);
\draw[braid, name path=t1f1l] (t1) to[out=225,in=45] (f1l);
\fill[white,name intersections={of=bin3 and bin2}] (intersection-1) circle(0.1);
\fill[white,name intersections={of=bin3 and t1f1l}] (intersection-1) circle(0.1);
\draw[braid] (in3) to[out=225, in=135] (f1r);
\draw[braid] (f1l) to (out1) ;
\draw[braid] (f1r) to (out2);
\draw (2.5,1.5) node[empty] {$\stackrel{\eqref{eq:mbm_morphism}}=$};
  \end{tikzpicture}
  \begin{tikzpicture}[scale=.935]  
\path
(0,3) node[empty, name=in1] {}
(0.5,3) node[empty, name=in2] {}
(1,3) node[empty, name=in3] {}
(1.5,3) node[empty, name=in4] {}
(2,3) node[empty, name=in5] {}
(0.25,0) node[empty, name=out1] {}
(1.5,0) node[empty, name=out2] {}
(0.5,2) node[empty, name=x] {}
(1.5,0.5) node[arr, name=f1d] {$f_1$}
(0.25,2.5) node[arr, name=t2] {$t_2$}
(1.75,2.5) node[arr, name=t1] { $t_1$}
(0.25,1.25) node[arr, name=f1l] {$f_1$}
(1.75,1.25) node[arr, name=f1r] {$f_1$};
\draw[braid] (in1) to[out=270,in=135] (t2);
\draw[braid, name path=bin2] (in2) to[out=270,in=45] (t2);
\path[braid, name path=bin3] (in3) to[out=270, in=135] (f1d);
\draw[braid] (in4) to[out=270,in=135] (t1);
\draw[braid] (in5) to[out=270,in=45] (t1);
\draw[braid] (t2) to[out=225,in=135] (f1l);
\draw[braid] (t1) to[out=315,in=45] (f1r);
\draw[braid, name path=t1f1l] (t1) to[out=225,in=45] (x) to[out=225,in=45] (f1l);
\fill[white] (x) circle(0.1);
\draw[braid, name path=t2f1r] (t2) to[out=315,in=135] (x) to[out=315,in=135] (f1r);
\fill[white,name intersections={of=bin3 and t2f1r}] (intersection-1) circle(0.1);
\fill[white,name intersections={of=bin3 and t1f1l}] (intersection-1) circle(0.1);
\draw[braid] (f1r) to[out=270,in=45] (f1d);
\draw[braid] (in3) to[out=270, in=135] (f1d);
\draw[braid] (f1l) to[out=270,in=90] (out1) ;
\draw[braid] (f1d) to (out2);
\draw (2.5,1.5) node[empty] {$\stackrel{\eqref{eq:multiplicative}}=$};
  \end{tikzpicture}
  \begin{tikzpicture}[scale=.935]  
\path
(0,3) node[empty, name=in1] {}
(0.5,3) node[empty, name=in2] {}
(1,3) node[empty, name=in3] {}
(1.5,3) node[empty, name=in4] {}
(2,3) node[empty, name=in5] {}
(0.25,0) node[empty, name=out1] {}
(1.75,0) node[empty, name=out2] {}
(0.5,2) node[empty, name=x] {}
(1.25,1.5) node[empty, name=m] {}
(0.25,2.5) node[arr, name=t2] {$t_2$}
(1.75,2.5) node[arr, name=t1] { $t_1$}
(0.25,1.25) node[arr, name=f1l] {$f_1$}
(1.75,1.0) node[arr, name=f1r] {$f_1$};
\draw[braid] (in1) to[out=270,in=135] (t2);
\draw[braid, name path=bin2] (in2) to[out=270,in=45] (t2);
\draw[braid, name path=bin3] (in3) to[out=270,in=180] (m);
\draw[braid] (m) to[out=270,in=135] (f1r);
\draw[braid] (in4) to[out=270,in=135] (t1);
\draw[braid] (in5) to[out=270,in=45] (t1);
\draw[braid] (t2) to[out=225,in=135] (f1l);
\draw[braid] (t1) to[out=315,in=45] (f1r);
\draw[braid, name path=t1f1l] (t1) to[out=225,in=45] (x) to[out=225,in=45] (f1l);
\fill[white] (x) circle(0.1);
\draw[braid, name path=t2f1r] (t2) to[out=315,in=135] (x) to[out=315,in=0] (m);
\fill[white,name intersections={of=bin3 and t2f1r}] (intersection-1) circle(0.1);
\fill[white,name intersections={of=bin3 and t1f1l}] (intersection-1) circle(0.1);
\draw[braid] (in3) to[out=270,in=180] (m);
\draw[braid] (f1l) to[out=270,in=90] (out1) ;
\draw[braid] (f1r) to (out2);
\draw (2.7,1.5) node[empty] {$\stackrel{\eqref{eq:t_2-3_compatibility}}=$};
  \end{tikzpicture}
  \begin{tikzpicture}[scale=.935]  
\path
(0,3) node[empty, name=in1] {}
(0.5,3) node[empty, name=in2] {}
(1,3) node[empty, name=in3] {}
(1.5,3) node[empty, name=in4] {}
(2,3) node[empty, name=in5] {}
(0.5,0) node[empty, name=out1] {}
(1.5,0) node[empty, name=out2] {}
(0.5,2) node[empty, name=x] {}
(0.3,1.5) node[empty, name=m] {}
(0.75,2.5) node[arr, name=t3] {$t_3$}
(1.75,2.5) node[arr, name=t1] { $t_1$}
(0.5,1.0) node[arr, name=f1l] {$f_1$}
(1.5,1.0) node[arr, name=f1r] {$f_1$};
\draw[braid] (in1) to[out=270,in=180] (m);
\draw[braid] (in2) to[out=270,in=135] (t3);
\draw[braid, name path=bin3] (in3) to[out=270,in=45] (t3);
\draw[braid] (in4) to[out=270,in=135] (t1);
\draw[braid] (in5) to[out=270,in=45] (t1);
\draw[braid] (t3) to[out=225,in=0] (m);
\draw[braid] (t1) to[out=315,in=45] (f1r);
\draw[braid, name path=t1f1l] (t1) to[out=225,in=45] (f1l);
\path[braid, name path=t3f1r] (t3) to[out=315,in=135] (f1r);
\draw[braid] (m) to[out=270,in=135] (f1l);
\fill[white,name intersections={of=t1f1l and t3f1r}] (intersection-1) circle(0.1);
\draw[braid] (t3) to[out=315,in=135] (f1r);
\draw[braid] (f1l) to[out=270,in=90]  (out1) ;
\draw[braid] (f1r) to (out2);
\draw (2.5,1.5) node[empty] {$\stackrel{\eqref{eq:multiplicative}}=$};
  \end{tikzpicture}
  \begin{tikzpicture}[scale=.935]  
\path
(0,3) node[empty, name=in1] {}
(0.5,3) node[empty, name=in2] {}
(1,3) node[empty, name=in3] {}
(1.5,3) node[empty, name=in4] {}
(2,3) node[empty, name=in5] {}
(0.5,0) node[empty, name=out1] {}
(1.75,0) node[empty, name=out2] {}
(0.5,2) node[empty, name=x] {}
(0.5,1.5) node[empty, name=m] {}
(0.75,2.5) node[arr, name=t3] {$t_3$}
(1.75,2.5) node[arr, name=t1] { $t_1$}
(0.5,0.5) node[arr, name=f1d] {$f_1$}
(0.75,1.5) node[arr, name=f1l] {$f_1$}
(1.75,1.5) node[arr, name=f1r] {$f_1$};
\draw[braid] (in1) to[out=270,in=135] (f1d);
\draw[braid] (in2) to[out=270,in=135] (t3);
\draw[braid, name path=bin3] (in3) to[out=270,in=45] (t3);
\draw[braid] (in4) to[out=270,in=135] (t1);
\draw[braid] (in5) to[out=270,in=45] (t1);
\draw[braid] (t3) to[out=225,in=135] (f1l);
\draw[braid] (t1) to[out=315,in=45] (f1r);
\draw[braid, name path=t1f1l] (t1) to[out=225,in=45] (f1l);
\path[braid, name path=t3f1r] (t3) to[out=315,in=135] (f1r);
\draw[braid] (f1l) to[out=270,in=45] (f1d);
\fill[white,name intersections={of=t1f1l and t3f1r}] (intersection-1) circle(0.1);
\draw[braid] (t3) to[out=315,in=135] (f1r);
\draw[braid] (f1d) to[out=270,in=90]  (out1) ;
\draw[braid] (f1r) to (out2);
  \end{tikzpicture}

\noindent
and Remark~\ref{rem:f_non-degenerate}.
\end{proof}

\subsection{Multiplier Hopf monoids}

\begin{definition}
A {\em multiplier Hopf monoid} in a braided monoidal category $\cc$ is a
multiplier bimonoid $(A,t_1,t_2,e)$ for which  $t_1$ and $t_2$ are isomorphisms
in $\cc$.
\end{definition}

Van Daele's notion of multiplier Hopf algebra \cite{VanDaele:multiplier_Hopf}
is equivalent to a multiplier Hopf monoid in the category of vector spaces
over the complex numbers, whose multiplication is non-degenerate:
see \cite[Section 3]{VanDaele:multiplier_Hopf}.  

\begin{theorem} \label{thm:Hopf}
Let $\cc$ be a braided monoidal category satisfying the standing assumptions
of Section~\ref{sect:assumptions}. Let $(A,t_1,t_2,e)$ be a multiplier
bimonoid  in $\cc$ having the properties listed in Section \ref{sect:mbm}. 
Then $(A,t_1,t_2,e)$ is  a multiplier Hopf monoid  if and only if
there  is an \mM-morphism $s\colon A\nto A$ whose components render
commutative the diagrams  
$$
\xymatrix{
 {}\llap{(a)~~}  
A^3 \ar[r]^-{t_21}\ar[d]_-{d_1} &
A^3 \ar[d]^-{1s_1}
&&
 {}\llap{(b)~~}  
A^3 \ar[r]^-{1t_1} \ar[d]_-{d_2} &
A^3 \ar[d]^-{s_21} \\
A^2 &
A^2 \ar[l]^-{t_1} 
&&
A^2 &
A^2.  \ar[l]^-{t_2} }
$$
Such an \mM-morphism is unique,  dense, and is called the {\em antipode} of
the multiplier Hopf monoid.  
\end{theorem}

\begin{proof}
If $t_1$ and $t_2$ are invertible,  define the components of $s$ to be 
\begin{equation}\label{eq:s}
\xymatrix{
s_1:=A^2\ar[r]^-{t_1^{-1}}&
A^2\ar[r]^-{e1} &
A
\qquad \textrm{and}\qquad
s_2:=A^2\ar[r]^-{t_2^{-1}}&
A^2\ar[r]^-{1e} &
A.}
\end{equation}
Commutativity of the first diagram in 
\begin{equation}\label{eq:pr_1}
\xymatrix{
A^3 \ar[rrr]^-{1t_1} \ar[d]^-{t_2^{-1}1} \ar@/_1.3pc/@{=}[ddd] 
\ar@{}[rrd]|-{\eqref{eq:mbm_ax_compatibility}} &&&
A^3 \ar[dd]^-{s_21} \ar[ld]_-{t_2^{-1}1} \\
A^3 \ar[rr]^-{1t_1} \ar[dd]^-{t_21} \ar[rdd]^-{m1} \ar[rrrd]_-{1m} &&
A^3 \ar[rd]^-{1e1} \\
&&&
A^2 \ar[d]^-m \\
A^3 \ar[r]_-{1e1} &
A^2 \ar[rr]_-m \ar@{}[rru]|-{\textrm{(associativity)}} &&
A}
\xymatrix@R=45pt{
A^3 \ar[rrrr]^-{t_21} \ar[dd]_-{d_1} \ar[rrd]^-{1t_1^{-1}}&&
\ar@{}[rd]|-{\eqref{eq:mbm_ax_compatibility}} &&
A^3 \ar[dd]^-{1s_1} \ar[ld]_-{1t_1^{-1}} \\
&&
A^3 \ar[r]^-{t_21} \ar[rrd]_-{m1} \ar@{}[d]|-{\eqref{eq:short_fusion}}&
A^3 \ar[rd]^-{1e1} \\
A^2 &&&&
A^2 \ar[llll]^-{t_1}}
\end{equation}
together with the invertibility of $t_1$ implies
$m.1s_1=m.1e1.1t_1^{-1}=m.s_21$ so that $s_1$ and $s_2$ define an \mM-morphism
$s\colon A\nto A$. Then diagram (a) commutes by the commutativity of the
second diagram of \eqref{eq:pr_1} and a symmetrical reasoning applies to 
diagram (b). 

Suppose conversely that $s$ is an \mM-morphism as in the theorem. 
The first diagram of  
\begin{equation} \label{eq:pr_2}
\xymatrix{
A^3 \ar@{=}[rr] \ar[d]_-{t_21} \ar@{}[rrd]|-{\textrm{(a)}} &&
A^3 \ar[r]^-{1e1} \ar[d]^-{d_1} &
A^2 \ar[dd]^-m \\
A^3 \ar[r]^-{1s_1} \ar[d]_-{s_21} 
\ar@{}[rd]|-{~\eqref{eq:mM}} &
A^2 \ar[r]^-{t_1} \ar[d]^-m &
A^2 \ar[rd]^-{e1} \\
A^2 \ar[r]_-m &
A \ar@{=}[rr] &&
A}
\xymatrix{
A^3 \ar[r]_-{c1} \ar[dd]_-{1e1} \ar@/^1.3pc/[rrrr]^-{d_1}&
A^3 \ar[r]_-{1t_1} \ar[dd]_-{e11} &
A^3 \ar[r]_-{c^{-1}1} \ar[d]^-{e11} &
A^3 \ar[r]_-{m1} \ar[dd]^-{ee1} &
A^2 \ar[dd]^-{e1}\\
&&
A^2 \ar[d]^-{e1} &
\ar@{}[rd]|-{\eqref{eq:e_multiplicative}}\\
A^2 \ar@{=}[r] &
A^2 \ar[r]_-m \ar[ru]^-{t_1}&
A \ar@{=}[r] &
A \ar@{=}[r] &
A. }
\end{equation}
commutes; where the region on the right commutes by the commutativity of 
the second diagram of \eqref{eq:pr_2}. 
By the non-degeneracy of $m$ this implies that the second diagram of 
\begin{equation}\label{eq:s1t1e}
\xymatrix{
A^2 \ar[r]^-{t_1} \ar[rd]_-{e1} &
A^2 \ar[d]^-{s_1} 
&
A^2 \ar[r]^-{t_2} \ar[rd]_-{1e} &
A^2 \ar[d]^-{s_2} \\
&A
&
&A\ }
\end{equation}
commutes, and similarly the first diagram also commutes. 

Now $d_1$ is in \cq by assumption, hence by commutativity of (a), $t_1$ is
certainly an epimorphism preserved by monoidal product, and in fact will be a
regular epimorphism provided that $t_21$ and $1s_1$ are epimorphisms. But
$t_21$ is an epimorphism by commutativity of (b) and the fact that
$d_2\in\cq$, while $1s_1$ will be an epimorphism by commutativity of
\eqref{eq:s1t1e} 
since $e\in \cq$ implies that $1e1$ is an epimorphism. 
Thus $t_1$ is a regular epimorphism; if we can show that it is a monomorphism
then it will be invertible.  

For any object $X$, the function $\cc(X,A^2)\to \cc(AX,A^2)$ which sends $f$
to the equal paths of 
$$\xymatrix{
AX \ar[r]^-{1f} &
A^3 \ar[rr]^-{1t_1} \ar[d]^-{t_21} \ar@/_1.3pc/[dd]_-{m1} 
\ar@{}[rrd]|-{\eqref{eq:mbm_ax_compatibility}} &&
A^3 \ar[d]^-{t_21} \\
&
A^3 \ar[rr]^-{1t_1} \ar[d]^-{1e1} \ar@{}[rrd]|-{\eqref{eq:s1t1e}} &&
A^3 \ar[d]^-{1s_1} \\
& 
A^2 \ar@{=}[rr] &&
A^2}
$$
is injective by the non-degeneracy of $m$; thus the function sending $f$ to
$t_1.f$ is also injective, and so $t_1$ is a monomorphism.  

This proves that $t_1$ is invertible, while $t_2$ is invertible by a symmetric
reasoning.  
Finally the uniqueness and the density of $s$ clearly follow from
the commutativity of \eqref{eq:s1t1e},  since $e\in Q$ and both $t_1$ and
$t_2$ are invertible. 
\end{proof}

\begin{remark} \label{rem:who_is_reg_epi}
Let $\cc$ be a braided monoidal category satisfying our standing assumptions
of Section~\ref{sect:assumptions}. Let $(A,t_1,t_2,e)$ be a multiplier Hopf
monoid in $\cc$ and consider the following assertions.
\begin{itemize}
\item[{(a)}] The counit $e:A\to I$ belongs to \cq. 
\item[{(b)}] The multiplication $m:A^2\to A$  belongs to \cq. 
\item[{(c)}] The component $s_1:A^2\to A$ of the antipode belongs to
  \cq. 
\item[{(d)}] The component $s_2:A^2\to A$ of the antipode  belongs to
  \cq.  
\item[{(e)}] The morphism $d_1:A^3\to A^2$ of \eqref{eq:d_1-2}  belongs
  to \cq. 
\item[{(f)}] The morphism $d_2:A^3\to A^2$ of \eqref{eq:d_1-2}  belongs
  to \cq. 
\end{itemize}
Since $t_1$ is invertible and  $s_1.t_1=e1$, condition (c) is
equivalent to $e1$ belonging to \cq, which follows from (a). Dually (d) is
equivalent to $1e$  belonging to \cq,  and since $1e$ is obtained from $e1$ by
conjugating with the braiding, this is equivalent to (c). Since $m=e1.t_1$, it
belongs to \cq  if and only if  $e1$ does,  thus (b) is equivalent to (c).
Since $d_1=t_1.1s_1. t_2  1$, it  belongs to \cq 
if and only if $1s_1$ does so, thus if and only if $1e1$ does so; this
is in turn equivalent to $d_2$  belonging to \cq, and so (e) is equivalent to
(f), and these follow from (c).  
 
In summary, 
(a) $\Rightarrow$ (b) $\Leftrightarrow$ (c) $\Leftrightarrow$
(d) $\Rightarrow$ (e) $\Leftrightarrow$ (f).  Thus for a multiplier Hopf
monoid $(A,t_1,t_2,e)$ with non-degenerate multiplication, all of the
properties listed in Section \ref{sect:mbm} follow from the single assumption
(a). 
\end{remark}

From Lemma \ref{lem:unit} and Proposition \ref{prop:unit} we obtain the
following. 

\begin{theorem} 
Consider a braided monoidal category $\cc$ satisfying our standing assumptions
of Section~\ref{sect:assumptions}.
For a multiplier Hopf monoid $(A,t_1,t_2,e)$ in $\cc$, the following assertions
are equivalent. 
\begin{itemize}
\item[{(a)}] The multiplication $m:=e1.t_1=1e.t_2$ admits some unit
  $u$, and the counit $e$ is an epimorphism. 
\item[{(b)}] There is a Hopf monoid $A$ with some monoid structure $(m,u)$,
  some comultiplication $h$, and the given counit $e$, such that $t_1=1m.h1$
  and $t_2=m1.1h$.
\end{itemize}
\end{theorem}

\begin{proof}
(b) implies (a) by Proposition \ref{prop:unit}.

Conversely, if (a) holds then applying Lemma \ref{lem:unit}
(a)$\Rightarrow$(b) to the multiplicative $\mM$-morphism $(e,e):A \nto I$ we
see that $e$ is a split epimorphism. Then $1e1:A^3\to A^2$ is an
epimorphism as well and therefore so are $1s_1=1e1.1t_1^{-1}$ and
$d_1=t_1.1s_1.t_21$ differing from it by isomorphisms. Then also $d_2$ is an
epimorphism by Lemma \ref{lem:unit} (a)$\Rightarrow$(c) and thus we can apply
Proposition \ref{prop:unit} to conclude that $A$ is a bimonoid whose
multiplication $m$ and comultiplication $h$ obey $t_1=1m.h1$ and
$t_2=m1.1h$. Since $t_1$ and $t_2$ are isomorphisms by assumption, $A$ is a
Hopf monoid proving (b). 
\end{proof}

\begin{theorem} \label{thm:s_mbm_morphism}
Consider a braided monoidal category $\cc$ satisfying our standing assumptions
of Section~\ref{sect:assumptions}. Let $(A,t_1,t_2,e)$ be a multiplier Hopf
monoid in $\cc$ with non-degenerate multiplication and dense
counit. Then its antipode $s$ is a morphism of multiplier bimonoids
$(A,c.t_2.c^{-1},$ $c.t_1.c^{-1},e) \to (A,t_1,t_2,e)$ in the sense of
Paragraph~\ref{claim:morphism}.  
\end{theorem}

\begin{proof}
By Remark \ref{rem:who_is_reg_epi}, also the multiplication $m$ and the
components $d_1$ and $d_2$ of the comultiplication belong to \cq. 
In light of Paragraph~\ref{claim:twist}, it will suffice to show that $s$
defines a morphism of comonoids in \cm, from $A\op\cop$ to $A$. We already
know that $s$ is a dense $\mM$-morphism; it is multiplicative by commutativity
of the diagram   
$$\xymatrix{
A^3 \ar[r]^-{c^{-1}1} \ar[d]^-{1t_1^{-1}} \ar@/_1.3pc/[dd]_-{1s_1} 
\ar@{}[rrrd]|-{\eqref{eq:short_fusion}}&
A^3 \ar[rr]^-{m1} &&
A^2 \ar[d]_-{t_1^{-1}} \ar@/^1.3pc/[dd]^-{s_1} \\
A^3 \ar[r]^-{c^{-1}1} \ar[d]^-{1e1} &
A^3 \ar[r]^-{1t_1^{-1}} \ar[d]^-{e11} &
A^3 \ar[r]^-{m1} \ar[d]_-{e11} \ar@{}[rd]|-{\eqref{eq:e_multiplicative}}&
A^2 \ar[d]_-{e1} \\
A^2 \ar@{=}[r] &
A^2 \ar[r]^-{t_1^{-1}} \ar@/_1.2pc/[rr]_-{s_1} &
A^2 \ar[r]^-{e1} &
A}$$
and so is at least a morphism in \cm. Next we should show that it preserves
the comonoid structure. Preservation of the counit follows by commutativity of
the diagram 
$$
\xymatrix{
A^2 \ar[r]_-{t_1^{-1}} \ar[d]_-{1e} \ar@/^1.3 pc/[rr]^-{s_1} 
\ar@{}[rd]|-{\eqref{eq:mbm_ax_1}} &
A^2 \ar[r]_-{e1} \ar[d]^-{1e} &
A\ar[d]^-e \\
A \ar@{=}[r] &
A \ar[r]_-e &
I}
$$
and so it remains to show that $s$ preserves the comultiplication.

We do this using string diagrams, as in Section~\ref{sect:strings}. 
Introduce the shorthand notation $f:= d\bullet s:A\nto A^2$, so that
$f_1.1d_1=d_1.s_111$. Then 
$$
\begin{tikzpicture} 
\path
(0,3) node[empty,name=in1] {}
(0.5,3) node[empty,name=in2] {}
(1,3) node[empty,name=in3] {} 
(1.5,3) node[empty,name=in4] {}
(1,0) node[empty,name=out1] {}
(1.5,0) node[empty,name=out2] {}
(1.0,0.5) node[empty,name=m] {}
(1.5,2.5) node[empty,name=x] {};
\path 
(1.0,1.5) node[arr,name=ds1] {$f_1$}
(1.0,2.5) node[arr,name=t1] {$t_1$};
\draw[braid] (in1) to[out=270,in=135] (ds1);
\draw[braid] (in2) to[out=270,in=135] (t1);
\draw[braid] (t1) to[out=225,in=90] (ds1);
\path[braid, name path=bin4] (in4) to[out=270,in=45] (t1);
\path[braid, name path=t1ds1] (t1) to[out=315,in=45] (ds1);
\path[braid, name path=bout2] (ds1) to[out=315,in=90] (out2);
\draw[braid, name path=bin3] (in3) to[out=270,in=90] (x) to[out=270,in=0] (m);
\fill[white,name intersections={of=bin3 and bin4}] (intersection-1) circle(0.1);
\fill[white,name intersections={of=bin3 and bout2}] (intersection-1) circle(0.1);
\draw[braid] (in4) to[out=270,in=45] (t1);
\draw[braid] (t1) to[out=315,in=45] (ds1);
\draw[braid] (ds1) to[out=225,in=180] (m);
\draw[braid] (ds1) to[out=315,in=90] (out2);
\draw[braid] (m) to[out=270,in=90] (out1);
\draw (2,1.5) node[empty] {$\stackrel{\eqref{eq:lin_one_leg}}=$};
\end{tikzpicture}
\begin{tikzpicture} 
\path
(0,3) node[empty,name=in1] {}
(0.5,3) node[empty,name=in2] {}
(1,3) node[empty,name=in3] {} 
(1.5,3) node[empty,name=in4] {}
(0.5,0) node[empty,name=out1] {}
(1.5,0) node[empty,name=out2] {}
(1.0,1.5) node[empty,name=m] {}
(1.5,2.5) node[empty,name=x] {};
\path 
(1.0,1.0) node[arr,name=ds1] {$f_1$}
(1.0,2.5) node[arr,name=t1] {$t_1$};
\draw[braid] (in1) to[out=270,in=135] (ds1);
\draw[braid] (in2) to[out=270,in=135] (t1);
\draw[braid] (t1) to[out=225,in=180] (m);
\draw[braid] (m) to (ds1);
\path[braid, name path=bin4] (in4) to[out=270,in=45] (t1);
\path[braid, name path=t1ds1] (t1) to[out=315,in=45] (ds1);
\draw[braid, name path=bin3] (in3) to[out=270,in=90] (x) to[out=270,in=0] (m);
\fill[white,name intersections={of=bin3 and bin4}] (intersection-1) circle(0.1);
\fill[white,name intersections={of=bin3 and t1ds1}] (intersection-1) circle(0.1);
\draw[braid] (in4) to[out=270,in=45] (t1);
\draw[braid] (t1) to[out=315,in=45] (ds1);
\draw[braid] (ds1) to[out=225,in=90] (out1);
\draw[braid] (ds1) to[out=315,in=90] (out2);
\draw (2,1.5) node[empty] {$=$};
\end{tikzpicture}
\begin{tikzpicture} 
\path
(0,3) node[empty,name=in1] {}
(0.5,3) node[empty,name=in2] {}
(1,3) node[empty,name=in3] {} 
(1.5,3) node[empty,name=in4] {}
(0.5,0) node[empty,name=out1] {}
(1.5,0) node[empty,name=out2] {}
(0.5,0.5) node[empty,name=m] {}
(1.5,1) node[empty,name=x] {};
\path 
(0.25,2) node[arr,name=s1] {$s_1$}
(0.5,1) node[arr,name=t1] {$t_1$};
\draw[braid] (in1) to[out=270,in=135] (s1);
\draw[braid] (in2) to[out=270,in=45] (s1);
\draw[braid] (s1) to[out=270,in=135] (t1);
\draw[braid] (t1) to[out=225,in=180] (m);
\draw[braid] (m) to (out1);
\path[braid, name path=bin4] (in4) to[out=270,in=45] (t1);
\path[braid, name path=bout2] (t1) to[out=315,in=90] (out2);
\draw[braid, name path=bin3] (in3) to[out=270,in=90] (x) to[out=270,in=0] (m);
\fill[white,name intersections={of=bin3 and bin4}] (intersection-1) circle(0.1);
\fill[white,name intersections={of=bin3 and bout2}] (intersection-1) circle(0.1);
\draw[braid] (in4) to[out=270,in=45] (t1);
\draw[braid] (t1) to[out=315,in=90] (out2);
\end{tikzpicture} \quad .
$$
By non-degeneracy of the multiplication the resulting equality 
is in turn equivalent to the equality  
$$ 
\begin{tikzpicture} 
\path
(.25,2) node[empty,name=in1] {}
(0.75,2) node[empty,name=in2] {}
(1.25,2) node[empty,name=in4] {}
(0.75,0) node[empty,name=out1] {}
(1.25,0) node[empty,name=out2] {};
\path 
(1.0,0.5) node[arr,name=ds1] {$f_1$}
(1.0,1.5) node[arr,name=t1] {$t_1$};
\draw[braid] (in1) to[out=270,in=135] (ds1);
\draw[braid] (in2) to[out=270,in=135] (t1);
\draw[braid] (t1) to[out=225,in=90] (ds1);
\path[braid, name path=bin4] (in4) to[out=270,in=45] (t1);
\path[braid, name path=t1ds1] (t1) to[out=315,in=45] (ds1);
\path[braid, name path=bout2] (ds1) to[out=315,in=90] (out2);
\draw[braid] (in4) to[out=270,in=45] (t1);
\draw[braid] (t1) to[out=315,in=45] (ds1);
\draw[braid] (ds1) to[out=225,in=90] (out1);
\draw[braid] (ds1) to[out=315,in=90] (out2);
\draw (1.7,1.0) node[empty] {$=$};
\end{tikzpicture}
\begin{tikzpicture} 
\path
(0,2) node[empty,name=in1] {}
(0.5,2) node[empty,name=in2] {}
(1,2) node[empty,name=in4] {}
(0.25,0) node[empty,name=out1] {}
(.75,0) node[empty,name=out2] {};
\path 
(0.25,1.5) node[arr,name=s1] {$s_1$}
(0.5,0.5) node[arr,name=t1] {$t_1$};
\draw[braid] (in1) to[out=270,in=135] (s1);
\draw[braid] (in2) to[out=270,in=45] (s1);
\draw[braid] (s1) to[out=270,in=135] (t1);
\draw[braid] (t1) to[out=225,in=90] (out1);
\draw[braid] (in4) to[out=270,in=45] (t1);
\draw[braid] (t1) to[out=315,in=90] (out2);
\end{tikzpicture} 
\quad  .
$$
We may now use the inverse of $t_1$ and the formula $s_1=e1.t^{-1}_1$ to write
$f_1= (d\bullet s)_1$ as the expression on the left below 
$$
\begin{tikzpicture} 
\path 
(0,3) node[empty,name=in1] {}
(1,3) node[empty,name=in2] {} 
(2,3) node[empty,name=in3] {}
(1,0) node[empty,name=out1] {}
(2,0) node[empty,name=out2] {};
\path
(1.5,2.5) node[arr,name=tinv1] {${t^{\scriptscriptstyle-1}_1}$} 
(0.5,1.5) node[arr,name=tinv2]  {${t^{\scriptscriptstyle-1}_1}$}  
(1.5,0.5) node[arr,name=t1] {$\,t_1\,$}
(0,0.5) node[unit,name=e] {}; 
\draw[braid] (in1) to[out=270,in=135] (tinv2);
\draw[braid] (in2) to[out=270,in=135] (tinv1);
\draw[braid] (in3) to[out=270,in=45] (tinv1);
\draw[braid] (tinv1) to[out=225,in=45] (tinv2);
\draw[braid] (tinv1) to[out=315,in=45] (t1);
\draw[braid] (tinv2) to[out=225,in=45] (e);
\draw[braid] (tinv2) to[out=315,in=135] (t1);
\draw[braid] (t1) to[out=225,in=90] (out1);
\draw[braid] (t1) to[out=315,in=90] (out2);
\draw (2.5,1.5) node[empty] {$=$};
\end{tikzpicture}
\begin{tikzpicture} 
\path 
(0,3) node[empty,name=in1] {}
(1,3) node[empty,name=in2] {} 
(2,3) node[empty,name=in3] {}
(0.5,0) node[empty,name=out1] {}
(2,0) node[empty,name=out2] {};
\path
(0.5,2.5) node[arr,name=tinv1] {${t^{\scriptscriptstyle-1}_1}$}
(1.5,1.5) node[arr,name=tinv2] {${t^{\scriptscriptstyle-1}_1}$}
(1,0.5) node[unit,name=e] {};
\draw[braid] (in1) to[out=270,in=135] (tinv1);
\draw[braid] (in2) to[out=270,in=45] (tinv1);
\draw[braid] (in3) to[out=270,in=45] (tinv2);
\path[braid, name path=tinv1tinv2] (tinv1) to[out=225,in=135] (tinv2);
\draw[braid, name path=bout1] (tinv1) to[out=315,in=90] (out1);
\fill[white,name intersections={of=tinv1tinv2 and bout1}] (intersection-1) circle(0.1);
\draw[braid] (tinv1) to[out=225,in=135] (tinv2);
\draw[braid] (tinv2) to[out=225,in=45] (e);
\draw[braid] (tinv2) to[out=315,in=90] (out2);
\end{tikzpicture}
$$
which, by the fusion equation \eqref{eq:mbm_ax_1}, is equal to
the expression on the right. Using the formula $s_1=e1.t^{-1}_1$ once again,
we see that $(d\bullet s)_1$ is given by the composite  
$$\xymatrix{
A^3 \ar[r]^{t^{-1}_11} & A^3 \ar[r]^{c1} & A^3 \ar[r]^{1s_1} & A^2. }$$
We must show that this is equal to the first component of $ss\bullet c^\#\bullet
d\op$; in other words, we must show that the equality  
\begin{equation}\label{eq:coass}
\begin{tikzpicture} 
\path
(0,3) node[empty,name=in1] {}
(0.5,3) node[empty,name=in2] {}
(1.0,3) node[empty,name=in3] {}
(1.5,3) node[empty,name=in4] {}
(2.0,3) node[empty,name=in5] {}
(0.5,0) node[empty,name=out1] {}
(1.5,0) node[empty,name=out2] {};
\path 
(0.5,2) node[arr,name=d2] {$d_2$}
(0.5,1) node[arr,name=s1l]  {$\,s_1\,$}
(1.5,1) node[arr,name=s1r] {$\,s_1\,$};
\draw[braid, name path=bin1] (in1) to[out=315,in=45] (d2);
\path[braid, name path=bin2] (in2) to[out=225,in=135] (d2);
\path[braid, name path=bin3] (in3) to[out=225,in=90] (d2);
\fill[white,name intersections={of=bin1 and bin2}] (intersection-1) circle(0.1);
\fill[white,name intersections={of=bin1 and bin3}] (intersection-1) circle(0.1);
\draw[braid] (in2) to[out=225,in=135] (d2);
\draw[braid] (in3) to[out=225,in=90] (d2);
\path[braid, name path=d2s1r] (d2) to[out=225,in=135] (s1r);
\draw[braid, name path=d2s1l] (d2) to[out=315,in=135] (s1l);
\draw[braid, name path=bin4] (in4) to[out=270,in=45] (s1l);
\fill[white,name intersections={of=d2s1r and d2s1l}] (intersection-1) circle(0.1);
\fill[white,name intersections={of=d2s1r and bin4}] (intersection-1) circle(0.1);
\draw[braid] (in5) to[out=270,in=45] (s1r);
\draw[braid] (d2) to[out=225,in=135] (s1r);
\draw[braid] (s1l) to (out1);
\draw[braid] (s1r) to (out2);
\draw (2.5,1.5) node[empty] {$=$};
\end{tikzpicture}
\begin{tikzpicture} 
\path
(0,3) node[empty,name=in1] {}
(0.5,3) node[empty,name=in2] {}
(1.0,3) node[empty,name=in3] {}
(1.5,3) node[empty,name=in4] {}
(2.0,3) node[empty,name=in5] {}
(0.5,0) node[empty,name=out1] {}
(1.5,0) node[empty,name=out2] {};
\path 
(0.5,1.5) node[arr,name=tinv] {${t^{\scriptscriptstyle -1}_1}$}
(0.5,2.25) node[arr,name=s1l]  {$\,s_1\,$}
(1.5,2.25) node[arr,name=s1r] {$\,s_1\,$}
(1.5,0.5) node[arr,name=s1d] {$\,s_1\,$};
\draw[braid, name path=bin1] (in1) to[out=270,in=135] (tinv);
\path[braid, name path=bin2] (in2) to[out=270,in=135] (s1r);
\draw[braid, name path=bin3] (in3) to[out=270,in=135] (s1l);
\draw[braid, name path=bin4] (in4) to[out=270,in=45] (s1l);
\draw[braid] (in5) to[out=270,in=45] (s1r);
\fill[white,name intersections={of=bin4 and bin2}] (intersection-1) circle(0.1);
\fill[white,name intersections={of=bin2 and bin3}] (intersection-1) circle(0.1);
\draw[braid] (in2) to[out=270,in=135] (s1r);
\draw[braid] (s1l) to[out=270,in=45] (tinv);
\draw[braid] (s1r) to[out=270,in=45] (s1d);
\path[braid, name path=tinvs1d] (tinv) to[out=225,in=135] (s1d);
\draw[braid, name path=bout1] (tinv) to[out=315,in=90] (out1);
\fill[white,name intersections={of=bout1 and tinvs1d}] (intersection-1) circle(0.1);
\draw[braid] (tinv) to[out=225,in=135] (s1d);
\draw[braid] (s1d) to (out2);
\end{tikzpicture}
\end{equation}
holds. By the definition of $d_2$ and routine braid calculations the left hand
side of this is equal to the left hand side of the following chain of
calculations.  
$$
\begin{tikzpicture} 
\path
(.3,3) node[empty,name=in1] {}
(0.7,3) node[empty,name=in2] {}
(1.0,3) node[empty,name=in3] {}
(1.5,3) node[empty,name=in4] {}
(2.0,3) node[empty,name=in5] {}
(1,1.5) node[empty,name=m] {} 
(0.5,0) node[empty,name=out1] {}
(1.5,0) node[empty,name=out2] {};
\path 
(0.5,2.5) node[arr,name=t2] {$\,t_2\,$}
(0.5,0.5) node[arr,name=s1l]  {$\,s_1\,$}
(1.5,0.5) node[arr,name=s1r] {$\,s_1\,$};
\draw[braid, name path=bin1] (in1) to[out=315,in=45] (t2);
\path[braid, name path=bin2] (in2) to[out=225,in=135] (t2);
\fill[white,name intersections={of=bin1 and bin2}] (intersection-1) circle(0.1);
\draw[braid] (in2) to[out=225,in=135] (t2);
\path[braid, name path=bin3] (in3) to[out=270,in=180] (m);
\path[braid, name path=t2s1r] (t2) to[out=225,in=135] (s1r);
\draw[braid, name path=t2m] (t2) to[out=315,in=0] (m);
\fill[white,name intersections={of=t2m and bin3}] (intersection-1) circle(0.1);
\draw[braid] (in3) to[out=270,in=180] (m);
\draw[braid, name path=ms1l] (m) to[out=270,in=135] (s1l);
\draw[braid, name path=bin4] (in4) to[out=270,in=45] (s1l);
\fill[white,name intersections={of=t2s1r and ms1l}] (intersection-1) circle(0.1);
\fill[white,name intersections={of=t2s1r and bin4}] (intersection-1) circle(0.1);
\draw[braid] (t2) to[out=225,in=135] (s1r);
\draw[braid] (in5) to[out=270,in=45] (s1r);
\draw[braid] (s1l) to (out1);
\draw[braid] (s1r) to (out2);
\draw (2.5,1.5) node[empty] {$=$};
\end{tikzpicture}
\begin{tikzpicture} 
\path
(.3,3) node[empty,name=in1] {}
(0.7,3) node[empty,name=in2] {}
(1.0,3) node[empty,name=in3] {}
(1.5,3) node[empty,name=in4] {}
(2.0,3) node[empty,name=in5] {}
(0.5,0) node[empty,name=out1] {}
(1.5,0) node[empty,name=out2] {};
\path 
(0.5,2.5) node[arr,name=t2] {$\,t_2\,$}
(1.25,2.5) node[arr,name=s1u] {$\,s_1\,$}
(0.5,0.5) node[arr,name=s1l]  {$\,s_1\,$}
(1.5,0.5) node[arr,name=s1r] {$\,s_1\,$};
\draw[braid, name path=bin1] (in1) to[out=315,in=45] (t2);
\path[braid, name path=bin2] (in2) to[out=225,in=135] (t2);
\fill[white,name intersections={of=bin1 and bin2}] (intersection-1) circle(0.1);
\draw[braid] (in2) to[out=225,in=135] (t2);
\path[braid, name path=t2s1r] (t2) to[out=225,in=135] (s1r);
\draw[braid, name path=t2s1l] (t2) to[out=315,in=135] (s1l);
\draw[braid] (in3) to[out=270,in=135] (s1u);
\draw[braid, name path=s1us1l] (s1u) to[out=270,in=45] (s1l);
\draw[braid, name path=bin4] (in4) to[out=270,in=45] (s1u);
\fill[white,name intersections={of=t2s1r and s1us1l}] (intersection-1) circle(0.1);
\fill[white,name intersections={of=t2s1r and t2s1l}] (intersection-1) circle(0.1);
\draw[braid] (t2) to[out=225,in=135] (s1r);
\draw[braid] (in5) to[out=270,in=45] (s1r);
\draw[braid] (s1l) to (out1);
\draw[braid] (s1r) to (out2);
\draw (2.5,1.5) node[empty] {$=$};
\end{tikzpicture}
\begin{tikzpicture} 
\path
(.3,3) node[empty,name=in1] {}
(0.7,3) node[empty,name=in2] {}
(1.0,3) node[empty,name=in3] {}
(1.5,3) node[empty,name=in4] {}
(2.0,3) node[empty,name=in5] {}
(0.2,1.5) node[empty,name=x] {} 
(0.5,0) node[empty,name=out1] {}
(1.5,0) node[empty,name=out2] {};
\path 
(0.5,2.5) node[arr,name=t2] {$\,t_2\,$}
(1.25,2.5) node[arr,name=s1u] {$\,s_1\,$}
(1,1.8) node[arr,name=tinv]  {${t^{\scriptscriptstyle-1}_1}$}
(0.5,1.5) node[unit,name=e] {}
(1.5,0.5) node[arr,name=s1r] {$\,s_1\,$};
\draw[braid, name path=bin1] (in1) to[out=315,in=45] (t2);
\path[braid, name path=bin2] (in2) to[out=225,in=135] (t2);
\fill[white,name intersections={of=bin1 and bin2}] (intersection-1) circle(0.1);
\draw[braid] (in2) to[out=225,in=135] (t2);
\path[braid, name path=t2s1r] (t2) to[out=225,in=90] (x) to[out=270,in=135] (s1r);
\draw[braid, name path=t2tinv] (t2) to[out=315,in=135] (tinv);
\draw[braid] (in3) to[out=270,in=135] (s1u);
\draw[braid] (tinv) to (e);
\draw[braid, name path=bout1] (tinv) to[out=315,in=90] (out1);
\draw[braid, name path=s1utinv] (s1u) to[out=270,in=45] (tinv);
\draw[braid, name path=bin4] (in4) to[out=270,in=45] (s1u);
\fill[white,name intersections={of=t2s1r and bout1}] (intersection-1) circle(0.1);
\draw[braid] (t2) to[out=225,in=90] (x) to[out=270,in=135] (s1r);
\draw[braid] (in5) to[out=270,in=45] (s1r);
\draw[braid] (s1r) to (out2);
\draw (2.5,1.5) node[empty] {$=$};
\end{tikzpicture}
\begin{tikzpicture} 
\path
(0.3,3) node[empty,name=in1] {}
(0.8,3) node[empty,name=in2] {}
(1.0,3) node[empty,name=in3] {}
(1.5,3) node[empty,name=in4] {}
(1.8,3) node[empty,name=in5] {}
(0.4,1.0) node[empty,name=x] {} 
(0.5,0) node[empty,name=out1] {}
(1.5,0) node[empty,name=out2] {};
\path 
(0.5,1.3) node[arr,name=t2] {$\,t_2\,$}
(1.25,2.5) node[arr,name=s1u] {$\,s_1\,$}
(1,1.8) node[arr,name=tinv]  {${t^{\scriptscriptstyle-1}_1}$}
(0.75,1) node[unit,name=e] {}
(1.5,0.5) node[arr,name=s1r] {$\,s_1\,$};
\draw[braid, name path=bin1] (in1) to[out=315,in=135] (tinv);
\path[braid, name path=bin2] (in2) to[out=225,in=135] (t2);
\fill[white,name intersections={of=bin1 and bin2}] (intersection-1) circle(0.1);
\draw[braid] (in2) to[out=225,in=135] (t2);
\path[braid, name path=t2s1r] (t2) to[out=225,in=135] (s1r);
\draw[braid, name path=t2tinv] (t2) to[out=45,in=225] (tinv);
\draw[braid] (in3) to[out=270,in=135] (s1u);
\draw[braid] (t2) to (e);
\draw[braid, name path=bout1] (tinv) to[out=315,in=90] (out1);
\draw[braid, name path=s1utinv] (s1u) to[out=270,in=45] (tinv);
\draw[braid, name path=bin4] (in4) to[out=270,in=45] (s1u);
\fill[white,name intersections={of=t2s1r and bout1}] (intersection-1) circle(0.1);
\draw[braid] (t2) to[out=225,
in=135] (s1r);
\draw[braid] (in5) to[out=270,in=45] (s1r);
\draw[braid] (s1r) to (out2);
\draw (2.3,1.5) node[empty] {$=$};
\end{tikzpicture}
\begin{tikzpicture} 
\path
(0.3,3) node[empty,name=in1] {}
(0.8,3) node[empty,name=in2] {}
(1.0,3) node[empty,name=in3] {}
(1.5,3) node[empty,name=in4] {}
(1.8,3) node[empty,name=in5] {}
(0.2,1.0) node[empty,name=x] {} 
(0.4,1.3) node[empty,name=m] {}
(0.5,0) node[empty,name=out1] {}
(1.5,0) node[empty,name=out2] {};
\path 
(1.25,2.5) node[arr,name=s1u] {$\,s_1\,$}
(1,1.8) node[arr,name=tinv]  {${t^{\scriptscriptstyle-1}_1}$}
(1.5,0.5) node[arr,name=s1r] {$\,s_1\,$};
\draw[braid, name path=bin1] (in1) to[out=315,in=135] (tinv);
\path[braid, name path=bin2] (in2) to[out=225,in=180] (m);
\fill[white,name intersections={of=bin1 and bin2}] (intersection-1) circle(0.1);
\draw[braid] (in2) to[out=225,in=180] (m);
\path[braid, name path=ms1r] (t2) to[out=270,in=135] (s1r);
\draw[braid, name path=tinvm] (tinv) to[out=225,in=0] (m);
\draw[braid] (in3) to[out=270,in=135] (s1u);
\draw[braid, name path=bout1] (tinv) to[out=315,in=90] (out1);
\draw[braid, name path=s1utinv] (s1u) to[out=270,in=45] (tinv);
\draw[braid, name path=bin4] (in4) to[out=270,in=45] (s1u);
\fill[white,name intersections={of=ms1r and bout1}] (intersection-1) circle(0.1);
\draw[braid] (m) to[out=270,in=135] (s1r);
\draw[braid] (in5) to[out=270,in=45] (s1r);
\draw[braid] (s1r) to (out2);
\end{tikzpicture}
$$
Here the first equality holds by multiplicativity of $s$, the second by the
defining property of (the bottom left) $s_1$, the third by
\eqref{eq:mbm_ax_compatibility}, and the fourth by (one) definition of the
multiplication $m$; finally the right hand side is equal to the right hand
side of \eqref{eq:coass} by multiplicativity of $s$ once again. 
\end{proof} 

\begin{definition}
A {\em regular multiplier Hopf monoid} in a braided monoidal category $\cc$ is
a regular multiplier bimonoid $(A,t_1,t_2,t_3,t_4,e)$ such that
$(A,t_1,t_2,e)$ is a multiplier Hopf monoid in $\cc$ and $(A,t_3,t_4,e)$ is a
multiplier Hopf monoid in $\overline \cc$; that is, such that $t_1$, $t_2$,
$t_3$, and $t_4$ are all isomorphisms in $\cc$.
\end{definition}

Van Daele's notion of regular multiplier Hopf algebra
\cite{VanDaele:multiplier_Hopf} is equivalent to a regular multiplier Hopf
monoid in the category of vector spaces over the complex numbers,
whose multiplication is non-degenerate. 

A Hopf monoid --- regarded as a multiplier Hopf monoid --- is regular if and
only if the antipode is invertible.

\begin{corollary} \label{cor:s'}
Let  $\cc$ be a braided monoidal category satisfying our standing assumptions
of Section~\ref{sect:assumptions}.  
Let $(A,t_1,t_2,t_3,t_4,e)$ be a regular multiplier Hopf monoid in $\cc$ with
non-degenerate multiplication and dense counit. 
Applying Theorem \ref{thm:Hopf} to the multiplier Hopf monoid $(A,t_3,t_4,e)$
in $\overline \cc$, we conclude that the components 
\begin{equation}\label{eq:s'}
\xymatrix{
s'_1:= A^2 \ar[r]^-{t_3^{-1}} &
A^2 \ar[r]^-{e1} &
A
\quad \textrm{and}\quad
s'_2:= A^2 \ar[r]^-{t_4^{-1}} &
A^2 \ar[r]^-{1e} &
A}
\end{equation}
of the antipode $s'\colon A\nto A\op$ of $(A,t_3,t_4,e)$
render commutative the diagrams
$$
\xymatrix{
A^3 \ar[r]^-{c1} \ar[d]_-{1c} &
A^3 \ar[r]^-{s'_11} &
A^2 \ar[d]^-m 
&
A^3 \ar[r]^-{c_{A^2,A}} \ar[d]_-{d_2} &
A^3 \ar[r]^-{t_41} &
A^3 \ar[d]^-{1s'_1}
&
A^3 \ar[r]^-{c_{A,A^2}} \ar[d]_-{d_1} &
A^3 \ar[r]^-{1t_3} &
A^3 \ar[d]^-{s'_21} \\
A^3 \ar[r]_-{1s'_2} &
A^2 \ar[r]_-m &
A
&
A^2 &&
A^2 \ar[ll]^-{t_3}
&
A^2 &&
A^2 . \ar[ll]^-{t_4}}
$$
\end{corollary}
 
\begin{theorem}\label{thm:s-inverse}
Let $\cc$ be a braided monoidal category satisfying our standing
assumptions of Section~\ref{sect:assumptions}. Let $(A,t_1,t_2,t_3,t_4,e)$ be
a regular multiplier Hopf monoid in $\cc$ with non-degenerate
multiplication and dense counit. Then the $\mM$-morphisms
$s\colon A\op \nto A$ in Theorem \ref{thm:Hopf} and $s'\colon A\nto
A\op$ in Corollary \ref{cor:s'} are mutually inverse in \cm.   
\end{theorem}

\proof
We shall prove that $s\bullet s'$ is the identity \mM-morphism;
the other inverse law holds by a symmetric argument. Furthermore, by
non-degeneracy, it will suffice to show that the first components agree; in
other words, that $(s\bullet s')_1=m$, or equivalently that $s_1.s'_11=m.1s_1$. 

To do this we use the compatibility condition 
\begin{equation}\label{eq:t1t3}
  \begin{tikzpicture}
  \path 
(0.75,2) node[arr,name=t1u] {$t_1$} 
(0.75,1) node[arr,name=t1d]  {$t_1$} 
(0.25,0.5) node[arr,name=t3] {$t_3$}; 
\path 
(1,2.5) node[empty,name=in3] {} 
(0.5,2.5) node[empty,name=in2] {} 
(0,2.5) node[empty,name=in1] {} 
(0,0) node[empty,name=out1] {} 
(0.5,0) node[empty,name=out2] {} 
(1,0) node[empty,name=out3] {}; 
\draw[braid] (in3) to[out=270,in=45] (t1u);
\draw[braid] (in2) to[out=270,in=135] (t1u);
\draw[braid] (t1u) to[out=315,in=45] (t1d);
\draw[braid] (t1d) to[out=315,in=90] (out3);
\path[braid,name path=bin1] (in1) to[out=270,in=135] (t1d);
\draw[braid,name path=t1ut3] (t1u) to[out=225,in=135] (t3);
\fill[white, name intersections={of=bin1 and t1ut3}] (intersection-1) circle(0.1);
\draw[braid] (in1) to[out=270,in=135] (t1d);
\draw[braid] (t1d)  to[out=225,in=45] (t3);
\draw[braid] (t3) to[out=225,in=90] (out1);
\draw[braid] (t3) to[out=315,in=90] (out2);
\draw (2,1.35) node {$=$};
\end{tikzpicture} 
\begin{tikzpicture}
\path
(0.75, 0.5) node[arr,name=t1] {$t_1$}
(0.25,1.5) node[arr,name=t3] {$t_3$};
\path 
(1.0,2.5) node[empty,name=in3] {} 
(0.5,2.5) node[empty,name=in2] {} 
(0,2.5) node[empty,name=in1] {} 
(0,0) node[empty,name=out1] {} 
(0.5,0) node[empty,name=out2] {} 
(1,0) node[empty,name=out3] {}; 
\draw[braid] (in3) to[out=270,in=45] (t1);
\path[braid, name path=bin1] (in1) to[out=270,in=45](t3);
\draw[braid, name path=bin2] (in2) to[out=270,in=135](t3);
\fill[white, name intersections={of=bin1 and bin2}] (intersection-1) circle(0.1);
\draw[braid] (in1) to[out=270,in=45](t3);
\draw[braid] (t3) to[out=225,in=90](out1);
\draw[braid] (t3) to[out=315,in=135] (t1);
\draw[braid] (t1) to[out=225,in=90] (out2);
\draw[braid] (t1) to[out=315,in=90] (out3);
\end{tikzpicture}
\end{equation}
which appeared in \cite[Remark~3.10]{BohmLack:braided_mba}, in the following
calculation.
$$
\begin{tikzpicture} 
\path (1.25,.7) node[arr,name=s1] {$\,s_1\,$}
(.75,1.4) node[arr,name=s'1] {$s'_1$};
\draw[braid] (.5,2) to[out=270,in=135] (s'1);
\draw[braid] (1,2) to[out=270,in=45] (s'1);
\draw[braid] (1.5,2) to[out=270,in=45] (s1);
\draw[braid] (s'1) to[out=270,in=135] (s1);
\draw[braid] (s1) to[out=270,in=90] (1.25,0);
\draw (2,1) node {$=$};
\end{tikzpicture} 
\begin{tikzpicture} 
\path (1.25,.7) node[arr,name=t1i] {${t_1^{\scriptscriptstyle -1}}$}
(.75,1.4) node[arr,name=t3i] {${t_3^{\scriptscriptstyle -1}}$}
(0.5,.2) node[unit,name=e1] {} 
(1,.2) node[unit,name=e3] {} ;
\draw[braid] (.5,2) to[out=270,in=135] (t3i);
\draw[braid] (1,2) to[out=270,in=45] (t3i);
\draw[braid] (1.5,2) to[out=270,in=45] (t1i);
\draw[braid] (t3i) to[out=315,in=135] (t1i);
\draw[braid] (t1i) to[out=315,in=90] (1.5,0);
\draw[braid,name path=t3i>e] (t3i) to[out=225,in=90] (e1);
\draw[braid] (t1i) to[out=225,in=90] (e3);
\draw (2,1) node {$=$};
\end{tikzpicture} 
\begin{tikzpicture} 
\path (1.25,.7) node[arr,name=t1i] {${t_1^{\scriptscriptstyle-1}}$}
(.75,1.4) node[arr,name=t3i] {${t_3^{\scriptscriptstyle -1}}$}
(0.5,.2) node[unit,name=e1] {} 
(1,.2) node[unit,name=e3] {} ;
\draw[braid] (.5,2) to[out=270,in=135] (t3i);
\draw[braid] (1,2) to[out=270,in=45] (t3i);
\draw[braid] (1.5,2) to[out=270,in=45] (t1i);
\draw[braid] (t3i) to[out=315,in=135] (t1i);
\draw[braid] (t1i) to[out=315,in=90] (1.5,0);
\path[braid,name path=t1i>e] (t1i) to[out=225,in=45] (e1);
\draw[braid,name path=t3i>e] (t3i) to[out=225,in=135] (e3);
\fill[white, name intersections={of=t1i>e and t3i>e}] (intersection-1) circle(0.1);
\draw[braid] (t1i) to[out=225,in=45] (e1);
\draw (2,1) node {$=$};
\end{tikzpicture} 
\begin{tikzpicture} 
\path (1.25,1.7) node[arr,name=t1i] {${t_1^{\scriptscriptstyle -1}}$}
(.75,1.2) node[arr,name=t3i] {${t_3^{\scriptscriptstyle-1}}$}
(1.25,.4) node[arr,name=t1] {$\,t_1\,$}
(.9,0) node[unit,name=e1] {} 
(.5,.2) node[unit,name=e3] {} ;
\draw[braid] (.5,2) to[out=270,in=135] (t3i);
\draw[braid] (1,2) to[out=270,in=135] (t1i);
\draw[braid] (1.5,2) to[out=270,in=45] (t1i);
\draw[braid] (t1i) to[out=225,in=45] (t3i);
\draw[braid] (t1i) to[out=315,in=45] (t1);
\draw[braid] (t1) to[out=315,in=90] (1.5,0);
\draw[braid] (t1) to[out=225,in=45] (e1);
\draw[braid,name path=t3i>t1] (t3i) to[out=225,in=135] (t1);
\path[braid,name path=t3i>e] (t3i) to[out=315,in=45] (e3);
\fill[white, name intersections={of=t3i>t1 and t3i>e}] (intersection-1) circle(0.1);
\draw[braid] (t3i) to[out=315,in=45] (e3);
\draw (2,1) node {$=$};
\end{tikzpicture} 
\begin{tikzpicture} 
\path (1.25,1.5) node[arr,name=t1i] {${t_1^{\scriptscriptstyle-1}}$}
(1,.5) node[arr,name=t1] {$\,t_1\,$}
(0.7,.2) node[unit,name=e1] {} 
(.9,1) node[unit,name=ei1] {} ;
\draw[braid] (.5,2) to[out=270,in=135] (t1);
\draw[braid] (1,2) to[out=270,in=135] (t1i);
\draw[braid] (1.5,2) to[out=270,in=45] (t1i);
\draw[braid] (t1i) to[out=315,in=45] (t1);
\draw[braid] (t1) to[out=315,in=90] (1.3,0);
\draw[braid] (t1i) to[out=225,in=45] (ei1);
\draw[braid] (t1) to[out=225,in=45] (e1);
\draw (2,1) node {$=$};
\end{tikzpicture} 
\begin{tikzpicture} 
\path (1.25,1.5) node[arr,name=s1] {$\,s_1\,$};
\draw[braid] (.5,2) to[out=270,in=180] (.75,.7) to[out=0,in=270] (s1);
\draw[braid] (1,2) to[out=270,in=135] (s1);
\draw[braid] (1.5,2) to[out=270,in=45] (s1);
\draw[braid] (.75,.7) to[out=270,in=90] (.75,0);
\end{tikzpicture} 
$$
\endproof

\begin{theorem} \label{thm:reg_antipode}
Let  $\cc$ be a braided monoidal category satisfying our standing assumptions
of Section~\ref{sect:assumptions}. 
For a regular multiplier Hopf monoid $(A,t_1,t_2,t_3,t_4,e)$ whose 
multiplication is non-degenerate and 
whose counit is dense, the following assertions hold. 
\begin{enumerate}
\item There is a unique morphism $\overline s:A \to A$ such that 
  $s=\overline s^\#$. 
\item There is a unique morphism $\overline s':A \to A$ such that
$s'=\overline {s}^{\prime \#}$. 
\item The morphisms $\overline s$ in part (1) and $\overline s'$ in part (2)
  are mutually inverse isomorphisms in $\cc$.
\end{enumerate}
\end{theorem}

\begin{proof}
This follows immediately from Theorem~\ref{thm:s-inverse} and 
\cite[Proposition~3.10]{BohmLack:cat_of_mbm}.
\end{proof}

If \cc is the category of vector spaces over the field of complex
numbers, then Theorem \ref{thm:reg_antipode} reduces to \cite[Proposition
5.2]{VanDaele:multiplier_Hopf}. 

\begin{corollary}
Applying \cite[Example 6.3]{BohmLack:cat_of_mbm}, we conclude from 
Theorem \ref{thm:s_mbm_morphism} and Theorem \ref{thm:reg_antipode} that 
for a regular multiplier Hopf monoid $(A,t_1,t_2,t_3,t_4,e)$ whose 
multiplication is non-degenerate and whose counit is dense, the
morphism $\overline s:A \to A$ in Theorem \ref{thm:reg_antipode} obeys 
$e.\overline s=e$ and it renders commutative the diagram
$$
\xymatrix{
A^2 \ar[r]^-{c^{-1}} \ar[d]_-{\overline s\, \overline s} &
A^2 \ar[r]^-{t_2} &
A^2 \ar[r]^-c &
A^2 \ar[d]^-{\overline s\, \overline s} \\
A^2 \ar[rrr]_-{t_1} &&&
A^2}
$$
(so that $\overline s.m.c^{-1}=m.\overline s\,\overline s$).
\end{corollary}


\section{Comodules over multiplier Hopf monoids}

For a regular multiplier bimonoid in  a braided monoidal category \cc, the
category of comodules was shown in \cite{BohmLack:braided_mba} to have a
monoidal structure lifting that of \cc. The aim in this section is to show that
if the underlying multiplier bimonoid is Hopf, then a comodule possesses a dual
if and only if the underlying object of \cc does so.  

\subsection{Comodules}
\label{claim:comodule}
A {\em comodule} \cite{BohmLack:braided_mba} over a regular multiplier bimonoid 
$(A,t_1,t_2,t_3,t_4,e)$ is a tuple $(V,v^1,v^3)$, where $v^1:VA\to VA$ is a 
comodule over $t_1$ \cite{BohmLack:braided_mba}, in the sense that it renders 
commutative the diagrams
\begin{equation}\label{eq:t_1_comodule}
\xymatrix{
VA^2 \ar[r]^-{1t_1} \ar[d]_-{v^11} &
VA^2 \ar[r]^-{c1} &
AVA \ar[r]^-{1v^1} &
AVA \ar[r]^-{c^{-1}1} &
VA^2 \ar[d]^-{v^11} 
&
VA \ar[r]^-{v^1} \ar[rd]_-{1e} &
VA\ar[d]^-{1e} \\
VA^2 \ar[rrrr]_-{1t_1} &&&&
VA^2
&
& V\ ;}
\end{equation}
and $v^3:VA\to VA$  is a comodule over $t_3$ 
\cite{BohmLack:braided_mba}, in the sense that it renders commutative the
diagrams 
\begin{equation}\label{eq:t_3_comodule}
\xymatrix{
VA^2 \ar[r]^-{1t_3} \ar[d]_-{v^31} &
VA^2 \ar[r]^-{c^{-1}1} &
AVA \ar[r]^-{1v^3} &
AVA \ar[r]^-{c1} &
VA^2 \ar[d]^-{v^31} 
&
VA \ar[r]^-{v^3} \ar[rd]_-{1e} &
VA\ar[d]^-{1e} \\
VA^2 \ar[rrrr]_-{1t_3} &&&&
VA^2
&
& V\ ; }
\end{equation}
and where finally $v_1$ and $v_3$ satisfy the compatibility condition asserting 
the commutativity of 
\begin{equation}\label{eq:comodule_compatibility}
\xymatrix{
AVA \ar[r]^-{1v^1} \ar[d]_-{c1} &
AVA \ar[r]^-{c1} &
VA^2 \ar[d]^-{1m} \\
VA^2 \ar[r]_-{v^31} &
VA^2 \ar[r]_-{1m} &
VA\ .}
\end{equation}
A {\em morphism of comodules} $(V,v^1,v^3)\to (W,w^1,w^3)$ is a morphism
$f\colon V\to W$ in $\cc$ such that the following diagrams commute. 
\begin{equation} \label{eq:morphism_of_comodules}
\xymatrix{
VA \ar[r]^-{f1} \ar[d]_-{v^1} &
 WA \ar[d]^-{w^1}
&&
VA \ar[r]^-{f1} \ar[d]_-{v^3} &
 WA \ar[d]^-{w^3} \\
VA \ar[r]_-{f1} &
 WA
&&
VA \ar[r]_-{f1} &
WA}
\end{equation}

If the multiplication is non-degenerate, then it is an easy consequence of
\eqref{eq:comodule_compatibility} that the following diagrams commute.
\begin{equation}\label{eq:v^1-3_module_maps}
\xymatrix{
VA^2 \ar[r]^-{v^11} \ar[d]_-{1m} &
VA^2 \ar[d]^-{1m} \\
VA \ar[r]_-{v^1} &
VA} \quad
\xymatrix{
VA^2 \ar[r]^-{c^{-1}1} \ar[d]_-{1m} &
AVA \ar[r]^-{1v^3} &
AVA \ar[r]^-{c1} &
VA^2 \ar[d]^-{1m} \\
VA \ar[rrr]_-{v^3} &&&
VA}
\end{equation}
In the second of these, the top path is equal to $1c^{-1}.v^31.1c$.

\begin{lemma} \label{lem:comodule_nd}
Let $(A,t_1,t_2,t_3,t_4,e)$ be a regular multiplier bimonoid in a braided
monoidal category $\cc$. Assume that it satisfies the conditions listed in
Section \ref{sect:mbm}. For any morphisms $v^1:VA \to VA$ and $v^3:VA \to
VA$ rendering commutative \eqref{eq:comodule_compatibility}, the following
assertions are equivalent.  
\begin{itemize}
\item[{(a)}] $(V,v^1,v^3)$ is an $A$-comodule. 
\item[{(b)}] $(V,v^1)$ is a $t_1$-comodule (that is, it renders commutative
the diagrams of \eqref{eq:t_1_comodule}).
\item[{(c)}] $(V,v^3)$ is a $t_3$-comodule (that is, it renders commutative
the diagrams of \eqref{eq:t_3_comodule}).
\end{itemize}
\end{lemma}

\begin{proof}
Since (a) means that (b) and (c) simultaneously hold, it suffices to prove the
equivalence of (b) and (c). We show that (b) implies (c); the opposite
implication follows symmetrically.

Since $11e:VA^2 \to VA$ is an epimorphism by assumption, the second diagram of
\eqref{eq:t_3_comodule} commutes by the commutativity of 
$$
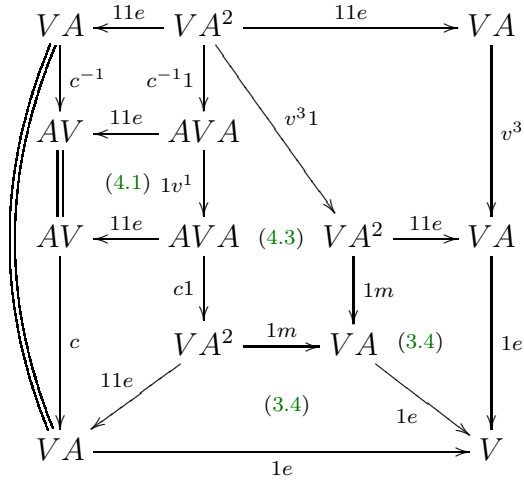
\xymatrix{
VA \ar[d]^-{c^{-1}} \ar@{=}@/_1.5pc/[dddd] &
VA^2 \ar[l]_-{11e} \ar[rr]^-{11e} \ar[d]_-{c^{-1}1} \ar[rdd]^-{v^31} &&
VA \ar[dd]^-{v^3} \\
AV \ar@{=}[d] \ar@{}[rd]|-{\eqref{eq:t_1_comodule}}&
AVA \ar[l]_-{11e} \ar[d]_-{1v^1} \\
AV \ar[dd]^-c &
AVA \ar[l]_-{11e} \ar[d]_-{c1} 
\ar@{}[r]|-{\eqref{eq:comodule_compatibility}} & 
VA^2 \ar[r]^-{11e} \ar[d]^-{1m} 
\ar@{}[rdd]|-{\eqref{eq:e_multiplicative}} &
VA \ar[dd]^-{1e} \\
&
VA^2 \ar[r]^-{1m} \ar[ld]_-{11e} \ar@{}[rd]|-{\eqref{eq:e_multiplicative}}&
VA \ar[rd]_-{1e} & \\
VA \ar[rrr]_-{1e} &&&
V}
$$
In order to see that the first diagram of \eqref{eq:t_3_comodule} commutes, we
use the following equalities 
$$
  
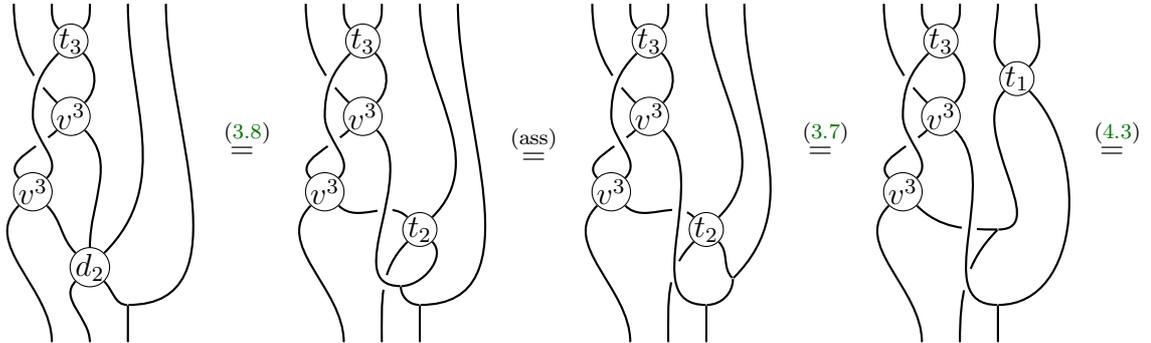
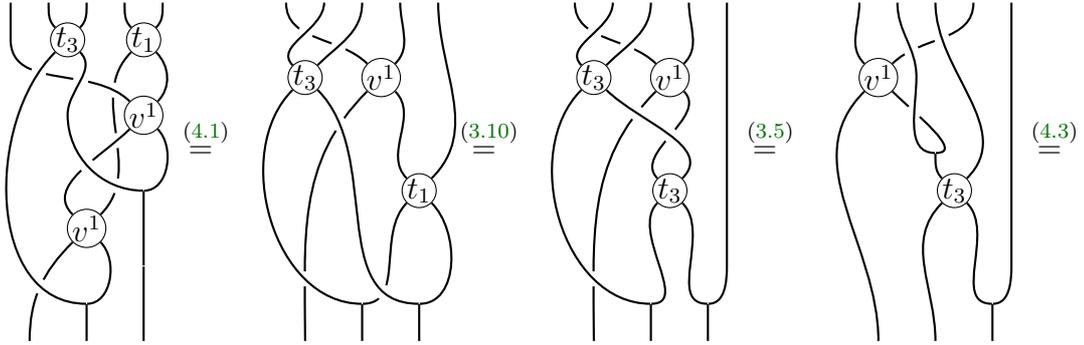
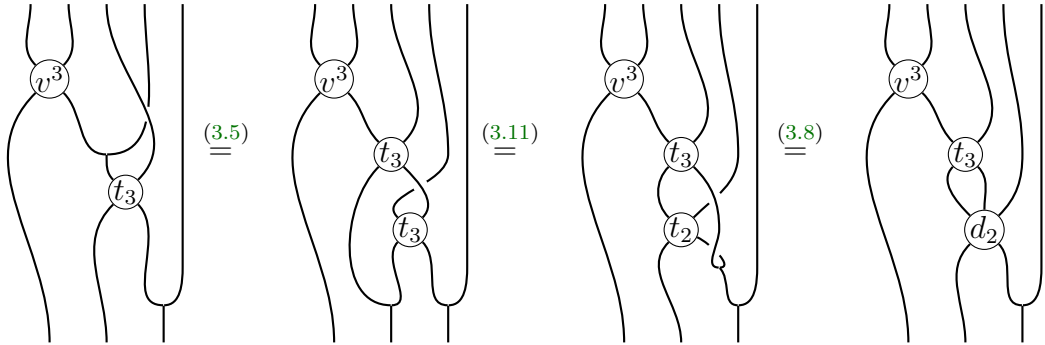
\begin{tikzpicture} 
    \path
(0,4.5) node[empty, name=in1] {}
(0.5,4.5) node[empty, name=in2] {}
(1,4.5) node[empty, name=in3] {}
(1.5,4.5) node[empty, name=in4] {}
(2,4.5) node[empty, name=in5] {}
(0.5,0) node[empty, name=out1] {}
(1.0,0) node[empty, name=out2] {}
(1.5,0) node[empty, name=out3] {}
(0.25,3) node[empty, name=x] {}
(1.5,0.5) node[empty, name=m] {};
\path
(0.75,4.0) node[arr, name=t3] {$t_3$}
(0.75,3.0) node[arr,name=v3u] {$v^3$}
(0.25,2.0) node[arr,name=v3d] {$v^3$}
(1,1) node[arr,name=d2] {$d_2$};
\draw[braid]  (in5) to[out=270,in=0] (m);
\draw[braid] (in4) to[out=270,in=45] (d2);
\draw[braid] (in3) to[out=270,in=45] (t3);
\draw[braid] (in2) to[out=270,in=135] (t3);
\draw[braid, name path=bin1] (in1) to[out=270,in=135] (v3u);
\draw[braid] (t3) to[out=315,in=45] (v3u);
\path[braid, name path=t3v3d] (t3) to[out=225, in=90] (x) to[out=270,in=45] (v3d);
\draw[braid] (v3u) to[out=315,in=90] (d2);
\draw[braid, name path=v3uv3d] (v3u) to[out=225,in=135] (v3d);
\fill[white, name intersections={of=bin1 and t3v3d}] (intersection-1) circle(0.1);
\fill[white, name intersections={of=v3uv3d and t3v3d}] (intersection-1) circle(0.1);
\draw[braid] (t3) to[out=225, in=90] (x) to[out=270,in=45] (v3d);
\draw[braid] (v3d) to[out=225,in=90] (out1);
\draw[braid] (v3d) to[out=315,in=135] (d2);
\draw[braid] (d2) to[out=315,in=180] (m);
\draw[braid] (d2) to[out=225,in=90] (out2);
\draw[braid] (m) to (out3);
\draw (3,2.7) node[empty] {$\stackrel{~\eqref{eq:d_1-2}}=$};
  \end{tikzpicture}
  \begin{tikzpicture} 
    \path
(0,4.5) node[empty, name=in1] {}
(0.5,4.5) node[empty, name=in2] {}
(1,4.5) node[empty, name=in3] {}
(1.5,4.5) node[empty, name=in4] {}
(2,4.5) node[empty, name=in5] {}
(0.5,0) node[empty, name=out1] {}
(1.0,0) node[empty, name=out2] {}
(1.5,0) node[empty, name=out3] {}
(0.25,3) node[empty, name=x] {}
(1.25,0.75) node[empty, name=m2] {}
(1.5,0.5) node[empty, name=m] {};
\path
(0.75,4.0) node[arr, name=t3] {$t_3$}
(0.75,3.0) node[arr,name=v3u] {$v^3$}
(0.25,2.0) node[arr,name=v3d] {$v^3$}
(1.5,1.5) node[arr,name=t2] {$t_2$};
\draw[braid]  (in5) to[out=270,in=0] (m);
\draw[braid] (in4) to[out=270,in=45] (t2);
\draw[braid] (in3) to[out=270,in=45] (t3);
\draw[braid] (in2) to[out=270,in=135] (t3);
\draw[braid, name path=bin1] (in1) to[out=270,in=135] (v3u);
\draw[braid] (t3) to[out=315,in=45] (v3u);
\path[braid, name path=t3v3d] (t3) to[out=225, in=90] (x) to[out=270,in=45] (v3d);
\path[braid, name path=v3um2] (v3u) to[out=315,in=180] (m2);
\draw[braid, name path=v3uv3d] (v3u) to[out=225,in=135] (v3d);
\fill[white, name intersections={of=bin1 and t3v3d}] (intersection-1) circle(0.1);
\fill[white, name intersections={of=v3uv3d and t3v3d}] (intersection-1) circle(0.1);
\draw[braid] (t3) to[out=225, in=90] (x) to[out=270,in=45] (v3d);
\draw[braid] (v3d) to[out=225,in=90] (out1);
\draw[braid, name path=v3dt2] (v3d) to[out=315,in=135] (t2);
\draw[braid] (t2) to[out=315,in=0] (m2);
\draw[braid, name path=bout2] (t2) to[out=225,in=90] (out2);
\fill[white, name intersections={of=v3dt2 and v3um2}] (intersection-1) circle(0.1);
\fill[white, name intersections={of=bout2 and v3um2}] (intersection-1) circle(0.1);
\draw[braid] (v3u) to[out=315,in=180] (m2);
\draw[braid] (m2) to[out=270,in=180]  (m);
\draw[braid] (m) to (out3);
\draw (3,2.63) node[empty] {$\stackrel{\mathrm{(ass)}}=$};
  \end{tikzpicture}
  \begin{tikzpicture} 
    \path
(0,4.5) node[empty, name=in1] {}
(0.5,4.5) node[empty, name=in2] {}
(1,4.5) node[empty, name=in3] {}
(1.5,4.5) node[empty, name=in4] {}
(2,4.5) node[empty, name=in5] {}
(0.5,0) node[empty, name=out1] {}
(1.0,0) node[empty, name=out2] {}
(1.5,0) node[empty, name=out3] {}
(0.25,3) node[empty, name=x] {}
(1.85,0.85) node[empty, name=m] {}
(1.5,0.5) node[empty, name=m2] {};
\path
(0.75,4.0) node[arr, name=t3] {$t_3$}
(0.75,3.0) node[arr,name=v3u] {$v^3$}
(0.25,2.0) node[arr,name=v3d] {$v^3$}
(1.5,1.5) node[arr,name=t2] {$t_2$};
\draw[braid]  (in5) to[out=270,in=45] (m);
\draw[braid] (in4) to[out=270,in=45] (t2);
\draw[braid] (in3) to[out=270,in=45] (t3);
\draw[braid] (in2) to[out=270,in=135] (t3);
\draw[braid, name path=bin1] (in1) to[out=270,in=135] (v3u);
\draw[braid] (t3) to[out=315,in=45] (v3u);
\path[braid, name path=t3v3d] (t3) to[out=225, in=90] (x) to[out=270,in=45] (v3d);
\path[braid, name path=v3um2] (v3u) to[out=315,in=180] (m2);
\draw[braid, name path=v3uv3d] (v3u) to[out=225,in=135] (v3d);
\fill[white, name intersections={of=bin1 and t3v3d}] (intersection-1) circle(0.1);
\fill[white, name intersections={of=v3uv3d and t3v3d}] (intersection-1) circle(0.1);
\draw[braid] (t3) to[out=225, in=90] (x) to[out=270,in=45] (v3d);
\draw[braid] (v3d) to[out=225,in=90] (out1);
\draw[braid, name path=v3dt2] (v3d) to[out=315,in=135] (t2);
\draw[braid] (t2) to[out=315,in=135] (m);
\draw[braid, name path=bout2] (t2) to[out=225,in=90] (out2);
\fill[white, name intersections={of=v3dt2 and v3um2}] (intersection-1) circle(0.1);
\fill[white, name intersections={of=bout2 and v3um2}] (intersection-1) circle(0.15);
\draw[braid] (v3u) to[out=315,in=180] (m2);
\draw[braid] (m) to[out=270,in=0]  (m2);
\draw[braid] (m2) to (out3);
\draw (3,2.7) node[empty] {$\stackrel{~\eqref{eq:short_compatibility}}=$};
  \end{tikzpicture}
  \begin{tikzpicture} 
    \path
(0,4.5) node[empty, name=in1] {}
(0.5,4.5) node[empty, name=in2] {}
(1,4.5) node[empty, name=in3] {}
(1.5,4.5) node[empty, name=in4] {}
(2,4.5) node[empty, name=in5] {}
(0.5,0) node[empty, name=out1] {}
(1.0,0) node[empty, name=out2] {}
(1.5,0) node[empty, name=out3] {}
(0.25,3) node[empty, name=x] {}
(1.5,1.5) node[empty, name=m] {}
(1.5,0.5) node[empty, name=m2] {};
\path
(0.75,4.0) node[arr, name=t3] {$t_3$}
(0.75,3.0) node[arr,name=v3u] {$v^3$}
(0.25,2.0) node[arr,name=v3d] {$v^3$}
(1.75,3.5) node[arr,name=t1] {$t_1$};
\draw[braid]  (in5) to[out=270,in=45] (t1);
\draw[braid] (in4) to[out=270,in=135] (t1);
\draw[braid] (in3) to[out=270,in=45] (t3);
\draw[braid] (in2) to[out=270,in=135] (t3);
\draw[braid, name path=bin1] (in1) to[out=270,in=135] (v3u);
\draw[braid] (t3) to[out=315,in=45] (v3u);
\path[braid, name path=t3v3d] (t3) to[out=225, in=90] (x) to[out=270,in=45] (v3d);
\path[braid, name path=v3um2] (v3u) to[out=315,in=180] (m2);
\draw[braid, name path=v3uv3d] (v3u) to[out=225,in=135] (v3d);
\fill[white, name intersections={of=bin1 and t3v3d}] (intersection-1) circle(0.1);
\fill[white, name intersections={of=v3uv3d and t3v3d}] (intersection-1) circle(0.1);
\draw[braid] (t3) to[out=225, in=90] (x) to[out=270,in=45] (v3d);
\draw[braid] (v3d) to[out=225,in=90] (out1);
\draw[braid, name path=v3dm] (v3d) to[out=315,in=180] (m);
\draw[braid] (t1) to[out=315,in=0] (m2);
\draw[braid] (t1) to[out=225,in=0] (m);
\draw[braid, name path=bout2] (m) to[out=225,in=90] (out2);
\fill[white, name intersections={of=v3dm and v3um2}] (intersection-1) circle(0.1);
\fill[white, name intersections={of=bout2 and v3um2}] (intersection-1) circle(0.15);
\draw[braid] (v3u) to[out=315,in=180] (m2);
\draw[braid] (m2) to (out3);
\draw (3,2.7) node[empty] {$\stackrel{~\eqref{eq:comodule_compatibility}}=$};
  \end{tikzpicture}
$$

$$
  \begin{tikzpicture} 
  \path
(0,4.5) node[empty, name=in1] {}
(0.5,4.5) node[empty, name=in2] {}
(1,4.5) node[empty, name=in3] {}
(1.5,4.5) node[empty, name=in4] {}
(2,4.5) node[empty, name=in5] {}
(0.25,0) node[empty, name=out1] {}
(1.0,0) node[empty, name=out2] {}
(1.75,0) node[empty, name=out3] {}
(0,4) node[empty, name=x] {}
(0.75,3) node[empty, name=y] {}
(1.75,1.0) node[empty, name=z] {}
(1.75,2) node[empty, name=m] {}
(1.0,0.5) node[empty, name=m2] {};
\path
(0.75,4.0) node[arr, name=t3] {$t_3$}
(1.75,3.0) node[arr,name=v1] {$v^1$}
(1,1.5) node[arr,name=v1d] {$v^1$}
(1.75,4.0) node[arr,name=t1] {$t_1$};
\draw[braid]  (in5) to[out=270,in=45] (t1);
\draw[braid] (in4) to[out=270,in=135] (t1);
\draw[braid] (in3) to[out=270,in=45] (t3);
\draw[braid] (in2) to[out=270,in=135] (t3);
\draw[braid, name path=t1v1d] (t1) to[out=225,in=45] (v1d);
\path[braid, name path=bin1] (in1) to[out=270,in=90] (x) to[out=270,in=135] (v1);
\fill[white, name intersections={of=bin1 and t1v1d}] (intersection-1) circle(0.1);
\draw[braid] (in1) to[out=270,in=90] (x) to[out=270,in=135] (v1);
\path[braid, name path=t3m] (t3) to[out=315,in=90] (y) to[out=270,in=180] (m);
\path[braid, name path=t3m2] (t3) to[out=225,in=180] (m2);
\path[braid, name path=v1v3] (v1) to[out=225,in=135] (v1d);
\fill[white, name intersections={of=bin1 and t3m}] (intersection-1) circle(0.1);
\fill[white, name intersections={of=bin1 and t3m2}] (intersection-1) circle(0.1);
\fill[white, name intersections={of=t1v1d and v1v3}] (intersection-1) circle(0.1);
\fill[white, name intersections={of=t1v1d and t3m}] (intersection-1) circle(0.1);
\draw[braid] (t1) to[out=315,in=45] (v1);
\draw[braid] (v1) to[out=315,in=0] (m);
\draw[braid] (v1) to[out=225,in=135] (v1d);
\fill[white, name intersections={of=t3m and v1v3}] (intersection-1) circle(0.1);
\fill[white, name intersections={of=t3m2 and v1v3}] (intersection-1) circle(0.1);
\draw[braid] (m) to (z) to[out=270,in=90]  (out3);
\draw[braid] (v1d) to[out=315,in=0] (m2);
\draw[braid, name path=bout1] (v1d) to[out=225,in=90] (out1);
\draw[braid] (t3) to[out=315,in=90] (y) to[out=270,in=180] (m);
\fill[white, name intersections={of=bout1 and t3m2}] (intersection-1) circle(0.1);
\draw[braid] (t3) to[out=225,in=180] (m2);
\draw[braid] (out2) to (m2);
\draw (2.5,2.7) node[empty] {$\stackrel{~\eqref{eq:t_1_comodule}}=$};
  \end{tikzpicture}
  \begin{tikzpicture} 
    \path
(0,4.5) node[empty, name=in1] {}
(0.5,4.5) node[empty, name=in2] {}
(1,4.5) node[empty, name=in3] {}
(1.5,4.5) node[empty, name=in4] {}
(2,4.5) node[empty, name=in5] {}
(0.25,0) node[empty, name=out1] {}
(1.0,0) node[empty, name=out2] {}
(1.75,0) node[empty, name=out3] {}
(1,0.5) node[empty, name=m] {}
(1.75,0.5) node[empty, name=m2] {};
\path
(0.25,3.5) node[arr, name=t3] {$t_3$}
(1.25,3.5) node[arr,name=v1] {$v^1$}
(1.75,2.0) node[arr,name=t1] {$t_1$};
\draw[braid]  (in5) to[out=270,in=45] (t1);
\draw[braid] (in4) to[out=270,in=45] (v1);
\path[braid, name path=bin3] (in3) to[out=270,in=45] (t3);
\path[braid, name path=bin2] (in2) to[out=270,in=135] (t3);
\draw[braid, name path=bin1] (in1) to[out=270,in=135] (v1);
\fill[white, name intersections={of=bin1 and bin2}] (intersection-1) circle(0.1);
\fill[white, name intersections={of=bin1 and bin3}] (intersection-1) circle(0.1);
\draw[braid] (in3) to[out=270,in=45] (t3);
\draw[braid] (in2) to[out=270,in=135] (t3);
\draw[braid] (v1) to[out=315,in=135] (t1);
\draw[braid, name path=bout1] (v1) to[out=225,in=90] (out1);
\draw[braid] (t1) to[out=315,in=0] (m2);
\draw[braid, name path=t1m] (t1) to[out=225,in=0] (m);
\path[braid, name path=t3m2] (t3) to[out=315,in=180] (m2);
\path[braid, name path=t3m] (t3) to[out=225,in=180] (m);
\fill[white, name intersections={of=t3m2 and t1m}] (intersection-1) circle(0.1);
\fill[white, name intersections={of=t3m and t1m}] (intersection-1) circle(0.1);
\fill[white, name intersections={of=t3m2 and bout1}] (intersection-1) circle(0.1);
\fill[white, name intersections={of=t3m and bout1}] (intersection-1) circle(0.1);
\draw[braid] (t3) to[out=315,in=180] (m2);
\draw[braid] (t3) to[out=225,in=180] (m);
\draw[braid] (m) to (out2);
\draw[braid] (m2) to (out3);
\draw (2.6,2.7) node[empty] {$\stackrel{~\eqref{eq:reg_mbm}}=$};
  \end{tikzpicture}
  \begin{tikzpicture} 
    \path
(0,4.5) node[empty, name=in1] {}
(0.5,4.5) node[empty, name=in2] {}
(1,4.5) node[empty, name=in3] {}
(1.5,4.5) node[empty, name=in4] {}
(2,4.5) node[empty, name=in5] {}
(0.25,0) node[empty, name=out1] {}
(1.0,0) node[empty, name=out2] {}
(1.75,0) node[empty, name=out3] {}
 (2,1) node[empty, name=x] {}
(1,0.5) node[empty, name=m] {}
(1.75,0.5) node[empty, name=m2] {};
\path
(0.25,3.5) node[arr, name=t3] {$t_3$}
(1.25,3.5) node[arr,name=v1] {$v^1$}
(1.25,2.0) node[arr,name=t3d] {$t_3$};
\draw[braid]  (in5) to[out=270,in=90] (x) to[out=270,in=0] (m2);
\draw[braid] (in4) to[out=270,in=45] (v1);
\path[braid, name path=bin3] (in3) to[out=270,in=45] (t3);
\path[braid, name path=bin2] (in2) to[out=270,in=135] (t3);
\draw[braid, name path=bin1] (in1) to[out=270,in=135] (v1);
\fill[white, name intersections={of=bin1 and bin2}] (intersection-1) circle(0.1);
\fill[white, name intersections={of=bin1 and bin3}] (intersection-1) circle(0.1);
\draw[braid] (in3) to[out=270,in=45] (t3);
\draw[braid] (in2) to[out=270,in=135] (t3);
\draw[braid, name path=v1t3d] (v1) to[out=315,in=135] (t3d);
\draw[braid, name path=bout1] (v1) to[out=225,in=90] (out1);
\draw[braid] (t3d) to[out=315,in=180] (m2);
\draw[braid] (t3d) to[out=225,in=0] (m);
\path[braid, name path=t3t3d] (t3) to[out=315,in=45] (t3d);
\path[braid, name path=t3m] (t3) to[out=225,in=180] (m);
\fill[white, name intersections={of=t3t3d and v1t3d}] (intersection-1) circle(0.1);
\fill[white, name intersections={of=t3t3d and bout1}] (intersection-1) circle(0.1);
\fill[white, name intersections={of=t3m and bout1}] (intersection-1) circle(0.1);
\draw[braid] (t3) to[out=315,in=45] (t3d);
\draw[braid] (t3) to[out=225,in=180] (m);
\draw[braid] (m) to (out2);
\draw[braid] (m2) to (out3);
\draw (2.5,2.7) node[empty] {$\stackrel{~\eqref{eq:short_fusion}}=$};
  \end{tikzpicture}
  \begin{tikzpicture} 
    \path
(0,4.5) node[empty, name=in1] {}
(0.5,4.5) node[empty, name=in2] {}
(1,4.5) node[empty, name=in3] {}
(1.5,4.5) node[empty, name=in4] {}
(2,4.5) node[empty, name=in5] {}
(0.25,0) node[empty, name=out1] {}
(1.0,0) node[empty, name=out2] {}
(1.75,0) node[empty, name=out3] {}
 (2,1) node[empty, name=x] {}
(0.75,3.5) node[empty, name=y] {}
(1,2.5) node[empty, name=m] {}
(1.75,0.5) node[empty, name=m2] {};
\path
(0.25,3.5) node[arr,name=v1] {$v^1$}
(1.25,2.0) node[arr,name=t3] {$t_3$};
\draw[braid]  (in5) to[out=270,in=90] (x) to[out=270,in=0] (m2);
\draw[braid, name path=bin4] (in4) to[out=270,in=45] (v1);
\path[braid, name path=bin3] (in3) to[out=270,in=45] (t3);
\path[braid, name path=bin2] (in2) to[out=270,in=90] (y) to[out=270,in=180] (m);
\draw[braid, name path=bin1] (in1) to[out=270,in=135] (v1);
\fill[white, name intersections={of=bin4 and bin2}] (intersection-1) circle(0.1);
\fill[white, name intersections={of=bin4 and bin3}] (intersection-1) circle(0.1);
\draw[braid, name path=v1m] (v1) to[out=315,in=0] (m);
\fill[white, name intersections={of=v1m and bin2}] (intersection-1) circle(0.1);
\draw[braid] (in3) to[out=270,in=45] (t3);
\draw[braid] (in2) to[out=270,in=90] (y) to[out=270,in=180] (m);
\draw[braid] (m) to[out=270,in=135] (t3);
\draw[braid] (t3) to[out=315,in=180] (m2);
\draw[braid, name path=bout1] (v1) to[out=225,in=90] (out1);
\draw[braid] (t3) to[out=225,in=90] (out2);
\draw[braid] (m2) to (out3);
\draw (2.5,2.7) node[empty] {$\stackrel{~\eqref{eq:comodule_compatibility}}=$};
  \end{tikzpicture}
$$

$$
  \begin{tikzpicture} 
    \path
(0,4.5) node[empty, name=in1] {}
(0.5,4.5) node[empty, name=in2] {}
(1,4.5) node[empty, name=in3] {}
(1.5,4.5) node[empty, name=in4] {}
(2,4.5) node[empty, name=in5] {}
(0.25,0) node[empty, name=out1] {}
(1.0,0) node[empty, name=out2] {}
(1.75,0) node[empty, name=out3] {}
 (2,1) node[empty, name=x] {}
(0.75,3.5) node[empty, name=y] {}
(1,2.5) node[empty, name=m] {}
(1.75,0.5) node[empty, name=m2] {};
\path
(0.25,3.5) node[arr,name=v3] {$v^3$}
(1.25,2.0) node[arr,name=t3] {$t_3$};
\draw[braid]  (in5) to[out=270,in=90] (x) to[out=270,in=0] (m2);
\path[braid, name path=bin3] (in3) to[out=270,in=45] (t3);
\draw[braid, name path=bin4] (in4) to[out=270,in=0] (m);
\draw[braid, name path=bin2] (in2) to[out=270,in=45] (v3);
\draw[braid, name path=bin1] (in1) to[out=270,in=135] (v1);
\fill[white, name intersections={of=bin4 and bin3}] (intersection-1) circle(0.1);
\draw[braid] (in3) to[out=270,in=45] (t3);
\draw[braid, name path=bout1] (v3) to[out=225,in=90] (out1);
\draw[braid] (v3) to[out=315,in=180] (m);
\draw[braid] (m) to[out=270,in=135] (t3);
\draw[braid] (t3) to[out=225,in=90] (out2);
\draw[braid] (t3) to[out=315,in=180] (m2);
\draw[braid] (m2) to (out3);
\draw (2.5,2.7) node[empty] {$\stackrel{~\eqref{eq:short_fusion}}=$};
  \end{tikzpicture}
  \begin{tikzpicture} 
    \path
(0,4.5) node[empty, name=in1] {}
(0.5,4.5) node[empty, name=in2] {}
(1,4.5) node[empty, name=in3] {}
(1.5,4.5) node[empty, name=in4] {}
(2,4.5) node[empty, name=in5] {}
(0.25,0) node[empty, name=out1] {}
(1.0,0) node[empty, name=out2] {}
(1.75,0) node[empty, name=out3] {}
 (2,1) node[empty, name=x] {}
(1.75,2.5) node[empty, name=y] {}
(1,0.5) node[empty, name=m] {}
(1.75,0.5) node[empty, name=m2] {};
\path
(0.25,3.5) node[arr,name=v3] {$v^3$}
(1.0,2.5) node[arr,name=t3] {$t_3$}
(1.25,1.5) node[arr,name=t3d] {$t_3$};
\draw[braid]  (in5) to[out=270,in=90] (x) to[out=270,in=0] (m2);
\draw[braid, name path=bin4] (in4) to[out=270,in=90] (y) to[out=270,in=135] (t3d);
\draw[braid, name path=bin3] (in3) to[out=270,in=45] (t3);
\draw[braid, name path=bin2] (in2) to[out=270,in=45] (v3);
\draw[braid, name path=bin1] (in1) to[out=270,in=135] (v1);
\path[braid, name path=t3t3d] (t3) to[out=315,in=45] (t3d);
\fill[white, name intersections={of=bin4 and t3t3d}] (intersection-1) circle(0.1);
\draw[braid] (t3) to[out=315,in=45] (t3d);
\draw[braid, name path=bout1] (v3) to[out=225,in=90] (out1);
\draw[braid] (v3) to[out=315,in=135] (t3);
\draw[braid] (t3) to[out=225,in=180] (m);
\draw[braid] (t3d) to[out=225,in=0] (m);
\draw[braid] (t3d) to[out=315,in=180] (m2);
\draw[braid] (m) to[out=270,in=90] (out2);
\draw[braid] (m2) to (out3);
\draw (2.5,2.7) node[empty] {$\stackrel{~\eqref{eq:t_2-3_compatibility}}=$};
\end{tikzpicture}
  \begin{tikzpicture} 
    \path
(0,4.5) node[empty, name=in1] {}
(0.5,4.5) node[empty, name=in2] {}
(1,4.5) node[empty, name=in3] {}
(1.5,4.5) node[empty, name=in4] {}
(2,4.5) node[empty, name=in5] {}
(0.25,0) node[empty, name=out1] {}
(1.0,0) node[empty, name=out2] {}
(1.75,0) node[empty, name=out3] {}
 (2,1) node[empty, name=x] {}
(1.75,2.5) node[empty, name=y] {}
 (1.5,1.5) node[empty, name=z] {}
(1.5,1) node[empty, name=m] {}
(1.75,0.5) node[empty, name=m2] {};
\path
(0.25,3.5) node[arr,name=v3] {$v^3$}
(1.0,2.5) node[arr,name=t3] {$t_3$}
(1.0,1.5) node[arr,name=t2] {$t_2$};
\draw[braid]  (in5) to[out=270,in=90] (x) to[out=270,in=0] (m2);
\draw[braid, name path=bin4] (in4) to[out=270,in=90] (y) to[out=270,in=45] (t2);
\draw[braid, name path=bin3] (in3) to[out=270,in=45] (t3);
\draw[braid, name path=bin2] (in2) to[out=270,in=45] (v3);
\draw[braid, name path=bin1] (in1) to[out=270,in=135] (v1);
\draw[braid] (t3) to[out=225,in=135] (t2);
\path[braid, name path=t3m] (t3) to[out=315,in=90] (z) to[out=270,in=180] (m);
\draw[braid, name path=t2m] (t2) to[out=335,in=0] (m);
\fill[white, name intersections={of=bin4 and t3m}] (intersection-1) circle(0.1);
\fill[white, name intersections={of=t2m and t3m}] (intersection-1) circle(0.1);
\draw[braid] (t3) to[out=315,in=90] (z) to[out=270,in=180] (m);
\draw[braid] (m) to[out=315,in=180] (m2);
\draw[braid, name path=bout1] (v3) to[out=225,in=90] (out1);
\draw[braid] (v3) to[out=315,in=135] (t3);
\draw[braid] (t2) to[out=225,in=90] (out2);
\draw[braid] (m2) to (out3);
\draw (2.5,2.7) node[empty] {$\stackrel{~\eqref{eq:d_1-2}}=$};
  \end{tikzpicture}
  \begin{tikzpicture} 
    \path
(0,4.5) node[empty, name=in1] {}
(0.5,4.5) node[empty, name=in2] {}
(1,4.5) node[empty, name=in3] {}
(1.5,4.5) node[empty, name=in4] {}
(2,4.5) node[empty, name=in5] {}
(0.25,0) node[empty, name=out1] {}
(1.0,0) node[empty, name=out2] {}
(1.75,0) node[empty, name=out3] {}
 (2,1) node[empty, name=x] {}
(1.75,2.5) node[empty, name=y] {}
 (1.5,1.5) node[empty, name=z] {}
(1.5,1) node[empty, name=m] {}
(1.75,0.5) node[empty, name=m2] {};
\path
(0.25,3.5) node[arr,name=v3] {$v^3$}
(1.0,2.5) node[arr,name=t3] {$t_3$}
(1.25,1.5) node[arr,name=d2] {$d_2$};
\draw[braid]  (in5) to[out=270,in=90] (x) to[out=270,in=0] (m2);
\draw[braid, name path=bin4] (in4) to[out=270,in=90] (y) to[out=270,in=45] (d2);
\draw[braid, name path=bin3] (in3) to[out=270,in=45] (t3);
\draw[braid, name path=bin2] (in2) to[out=270,in=45] (v3);
\draw[braid, name path=bin1] (in1) to[out=270,in=135] (v1);
\draw[braid] (t3) to[out=225,in=135] (d2);
\draw[braid] (t3) to[out=315,in=90] (d2);
\draw[braid] (d2) to[out=315,in=180] (m2);
\draw[braid, name path=bout1] (v3) to[out=225,in=90] (out1);
\draw[braid] (v3) to[out=315,in=135] (t3);
\draw[braid] (d2) to[out=225,in=90] (out2);
\draw[braid] (m2) to (out3);
  \end{tikzpicture}
$$
and then the claim follows from equality of the first and last of these, the
non-degeneracy of $m$, and Remark~\ref{rem:f_non-degenerate}. 
\end{proof}

\begin{lemma} \label{lem:comodule_morphism_nd} Let $(A,t_1,t_2,t_3,t_4,e)$ be a
regular multiplier bimonoid in a braided monoidal category $\cc$ having the
properties listed in Section \ref{sect:mbm}.  
For any $A$-comodules $(V,v^1,v^3)$, $(W,w^1,w^3)$ and any morphism
$f:V\to W$ in $\cc$, the following assertions are equivalent. 
\begin{itemize}
\item[{(a)}] $f$ is a morphism of comodules.
\item[{(b)}] $f$ renders commutative the first diagram of
  \eqref{eq:morphism_of_comodules}. 
\item[{(c)}] $f$ renders commutative the second diagram of
  \eqref{eq:morphism_of_comodules}. 
\end{itemize}
\end{lemma}

\begin{proof}
Since (a) means that (b) and (c) simultaneously hold, it suffices to prove the
equivalence of (b) and (c). This follows by noting that -- thanks to the
non-degeneracy of $m$ -- (b) is equivalent to the commutativity of the
exterior; and (c) is equivalent to the commutativity of the inner square, in
\[
\xy *!D\xybox{
\xymatrix{
AVA \ar[rr]^-{1v^1} \ar[ddd]_-{1f1} \ar[rd]^-{c1} &
\ar@{}[rrd]|-{\eqref{eq:comodule_compatibility}} &
AVA \ar[r]^-{c1} &
VA^2 \ar[r]^-{1m} &
VA \ar[ddd]^-{f1} \ar@{=}[ld] \\
&
VA^2 \ar[r]^-{v^31} \ar[d]_-{f11} &
VA^2 \ar[r]^-{1m} &
VA \ar[d]^-{f1} \\
& 
 W A^2 \ar[r]^-{ w^3 1} 
\ar@{}[rrd]|-{\eqref{eq:comodule_compatibility}} &
 WA^2 \ar[r]^-{1m} &
 WA \ar@{=}[rd] \\
 AWA \ar[rr]_-{ 1w^1} \ar[ru]^-{c1}&&
 AWA \ar[r]_-{c1} &
 WA^2 \ar[r]_-{1m} &
 WA\ .} }\endxy 
\]
\end{proof}

\subsection{Change of base for comodules}

A morphism of (counital) coalgebras induces a functor between the
corresponding categories of comodules. There is a corresponding construction
for comodules over a regular multiplier bimonoid.  

\begin{proposition}\label{prop:f_*}
Let  $A$ and $B$ be regular multiplier bimonoids having the properties listed
in Section \ref{sect:mbm}, and let $f\colon A\nto B$ be a dense multiplicative
\mM-morphism.   
\begin{enumerate}
\item If $v^1\colon VA\to VA$ and $v^3\colon VA\to VA$ satisfy
  \eqref{eq:comodule_compatibility}, then there is a unique pair of morphisms
  $w^1\colon VB\to VB$ and $w^3\colon VB\to VB$ which satisfy
  \eqref{eq:comodule_compatibility} and make the following diagrams commute.  
\begin{equation}\label{eq:induced-comodule}
\xymatrix{
VAB \ar[r]^{1f_1} \ar[d]_{v^11} & VB \ar[d]^{w^1} \\
VAB \ar[r]_{1f_1} & VB } \quad
\xymatrix{
VAB \ar[r]^{1c^{-1}} \ar[d]_{v^31} & VBA \ar[r]^{1f_2} & VB \ar[d]^{w^3} \\
VAB \ar[r]_{1c^{-1}} & VBA \ar[r]_{1f_2} & VB }
\end{equation}
\item If $f$ is a morphism of multiplier bimonoids and $(V,v^1,v^3)$ is an
  $A$-comodule, then $(V,w^1,w^3)$ is a $B$-comodule. 
\end{enumerate}
The construction in part (2) gives the object map of a functor $f_*$ 
acting on the morphisms as the identity map, from the category of
$A$-comodules to the category of $B$-comodules.  For morphisms
$A\stackrel f \nto B \stackrel g \nto C$ of multiplier bimonoids, the
composite functor $g_*f_*$ is naturally isomorphic to $(g\bullet f)_*$ and the
identity $\mM$-morphism $A\nto A$ induces the identity
functor. If we regard \cm as a bicategory with only identity 2-cells, then
this defines a pseudofunctor from \cm to the 2-category of categories,
functors and natural transformations. 
\end{proposition}

\begin{proof}
First we verify the existence of $w^1$ and $w^3$ satisfying the defining
relations \eqref{eq:induced-comodule}. For the case of $w^1$, we use the fact
that $f_1$ is the coequalizer of some pair of morphisms 
$x,x'\colon X\to AB$, and that this coequalizer is preserved by  taking
the monoidal product with $V$. 
In the diagram
$$
\xymatrix{
AVAB \ar@{=}[r] \ar[dd]_-{11f_1} &
AVAB \ar[r]^-{1v^11} \ar[d]_-{c11} 
\ar@{}[rdd]|-{\eqref{eq:comodule_compatibility}} &
AVAB \ar[r]^-{11f_1} \ar[d]^-{c11} &
AVB \ar[r]^-{1c^{-1}} \ar[d]^-{c1} &
ABV \ar[ddd]^-{f_11} \\
& VA^2B \ar[d]_-{v^311} &
VA^2B\ar[r]^-{11f_1} \ar[d]^-{1m1} \ar@{}[rdd]|-{\eqref{eq:multiplicative}} &
VAB \ar[dd]^-{1f_1} \\
AVB \ar[d]_-{c1} &
VA^2B \ar[r]^-{1m1} \ar[d]_-{11f_1} \ar@{}[rd]|-{\eqref{eq:multiplicative}} &
VAB \ar[d]^-{1f_1} \\
VAB \ar[r]_-{v^31} &
VAB \ar[r]_-{1f_1} &
VB \ar@{=}[r] &
VB \ar[r]_-{c^{-1}} &
BV}
$$
the left-bottom composite coequalizes $11x$ and $11x'$, hence so does the
top-right composite. Now by Remark~\ref{rem:f_non-degenerate} it follows that
$c^{-1}.1f_1.v^11$ coequalizes $1x$ and $1x'$, hence so too does
$1f_1.v^11$. The existence of $w^1$ follows.  
The case of $w^3$ is similar.  

By the calculation
$$
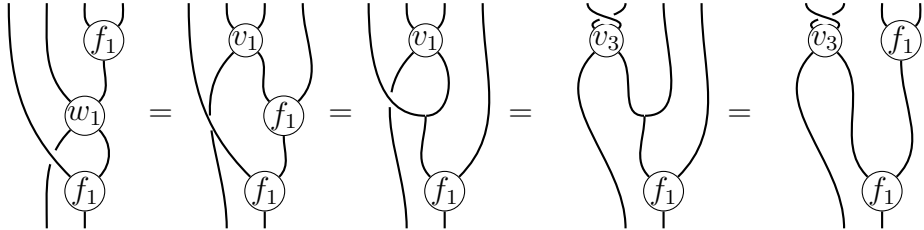
\begin{tikzpicture} 
\path
(0,3) node[empty, name=in1] {}
(0.5,3) node[empty, name=in2] {}
(1,3) node[empty, name=in3] {}
(1.5,3) node[empty, name=in4] {}
(0.5,0) node[empty, name=out1] {}
(1,0) node[empty, name=out2] {}
(1.25,2.5) node[arr, name=f1u] {$f_1$}
(1,1.5) node[arr, name=w1] {$w_1$}
(1,0.5) node[arr, name=f1d] {$f_1$};
\path[braid, name path=bin1] (in1) to[out=270,in=135] (f1d);
\draw[braid] (in2) to[out=270,in=135] (w1);
\draw[braid] (in3) to[out=270,in=135] (f1u);
\draw[braid] (in4) to[out=270, in=45] (f1u);
\draw[braid] (f1u) to[out=270,in=45] (w1);
\draw[braid] (w1) to[out=315,in=45] (f1d);
\draw[braid, name path=bout1] (w1) to[out=225, in=90] (out1);
\fill[white, name intersections={of=bin1 and bout1}] (intersection-1) circle(0.1);\draw[braid] (in1) to[out=270,in=135] (f1d);
\draw[braid] (f1d) to (out2);
\path (2,1.5) node {$=$};
\end{tikzpicture}
\begin{tikzpicture} 
\path
(0,3) node[empty, name=in1] {}
(0.5,3) node[empty, name=in2] {}
(1,3) node[empty, name=in3] {}
(1.5,3) node[empty, name=in4] {}
(0.5,0) node[empty, name=out1] {}
(1,0) node[empty, name=out2] {}
(1.25,1.5) node[arr, name=f1u] {$f_1$}
(0.75,2.5) node[arr, name=v1] {$v_1$}
(1,0.5) node[arr, name=f1d] {$f_1$};
\path[braid, name path=bin1] (in1) to[out=270,in=135] (f1d);
\draw[braid] (in2) to[out=270,in=135] (v1);
\draw[braid] (in3) to[out=270,in=45] (v1);
\draw[braid] (in4) to[out=270, in=45] (f1u);
\draw[braid] (v1) to[out=315,in=135] (f1u);
\draw[braid] (f1u) to[out=270,in=45] (f1d);
\draw[braid, name path=bout1] (v1) to[out=225, in=90] (out1);
\fill[white, name intersections={of=bin1 and bout1}] (intersection-1) circle(0.1);
\draw[braid] (in1) to[out=270,in=135] (f1d);
\draw[braid] (f1d) to (out2);
\path (2,1.5) node {$=$};
\end{tikzpicture}
\begin{tikzpicture} 
\path
(0,3) node[empty, name=in1] {}
(0.5,3) node[empty, name=in2] {}
(1,3) node[empty, name=in3] {}
(1.5,3) node[empty, name=in4] {}
(0.5,0) node[empty, name=out1] {}
(1,0) node[empty, name=out2] {}
(0.75,1.5) node[empty, name=m] {}
(0.75,2.5) node[arr, name=v1] {$v_1$}
(1,0.5) node[arr, name=f1] {$f_1$};
\path[braid, name path=bin1] (in1) to[out=270,in=180] (m);
\draw[braid] (in2) to[out=270,in=135] (v1);
\draw[braid] (in3) to[out=270,in=45] (v1);
\draw[braid] (in4) to[out=270, in=45] (f1);
\draw[braid] (v1) to[out=315,in=0] (m);
\draw[braid] (m) to[out=270,in=135] (f1);
\draw[braid, name path=bout1] (v1) to[out=225, in=90] (out1);
\fill[white, name intersections={of=bin1 and bout1}] (intersection-1) circle(0.1);
\draw[braid] (in1) to[out=270,in=180] (m);
\draw[braid] (f1) to (out2);
\path (2,1.5) node {$=$};
\end{tikzpicture}
\begin{tikzpicture} 
\path
(0,3) node[empty, name=in1] {}
(0.5,3) node[empty, name=in2] {}
(1,3) node[empty, name=in3] {}
(1.5,3) node[empty, name=in4] {}
(0.5,0) node[empty, name=out1] {}
(1,0) node[empty, name=out2] {}
(0.75,1.5) node[empty, name=m] {}
(0.25,2.5) node[arr, name=v3] {$v_3$}
(1,0.5) node[arr, name=f1] {$f_1$};
\path[braid, name path=bin1] (in1) to[out=270,in=45] (v3);
\draw[braid, name path=bin2] (in2) to[out=270,in=135] (v3);
\draw[braid] (in3) to[out=270,in=0] (m);
\draw[braid] (in4) to[out=270, in=45] (f1);
\draw[braid] (v3) to[out=315,in=180] (m);
\draw[braid] (m) to[out=270,in=135] (f1);
\draw[braid, name path=bout1] (v3) to[out=225, in=90] (out1);
\fill[white, name intersections={of=bin1 and bin2}] (intersection-1) circle(0.1);
\draw[braid] (in1) to[out=270,in=45] (v3);
\draw[braid] (f1) to (out2);
\path (2,1.5) node {$=$};
\end{tikzpicture}
\begin{tikzpicture} 
\path
(0,3) node[empty, name=in1] {}
(0.5,3) node[empty, name=in2] {}
(1,3) node[empty, name=in3] {}
(1.5,3) node[empty, name=in4] {}
(0.5,0) node[empty, name=out1] {}
(1,0) node[empty, name=out2] {}
(0.75,1.5) node[empty, name=m] {}
(0.25,2.5) node[arr, name=v3] {$v_3$}
(1.25,2.5) node[arr, name=f1u] {$f_1$}
(1,0.5) node[arr, name=f1] {$f_1$};
\path[braid, name path=bin1] (in1) to[out=270,in=45] (v3);
\draw[braid, name path=bin2] (in2) to[out=270,in=135] (v3);
\draw[braid] (in3) to[out=270,in=135] (f1u);
\draw[braid] (in4) to[out=270, in=45] (f1u);
\draw[braid] (v3) to[out=315,in=135] (f1);
\draw[braid] (f1u) to[out=270,in=45] (f1);
\draw[braid, name path=bout1] (v3) to[out=225, in=90] (out1);
\fill[white, name intersections={of=bin1 and bin2}] (intersection-1) circle(0.1);
\draw[braid] (in1) to[out=270,in=45] (v3);
\draw[braid] (f1) to (out2);
\end{tikzpicture}
$$
and density of $f$, we see that the diagram
\begin{equation}\label{eq:w1-nd}
\xymatrix{
AVB \ar[r]^{1w^1} \ar[d]_{c1} & AVB \ar[r]^{c1} & VAB \ar[d]^{1f_1} \\
VAB \ar[r]_{v^31} & VAB \ar[r]_{1f_1} & VB }
\end{equation}
commutes. (In light of Remark~\ref{rem:f_non-degenerate} this actually
characterizes $w^1$.) 

The compatibility condition \eqref{eq:comodule_compatibility} for $w^1$ and
$w^3$ follows from the calculation  
$$
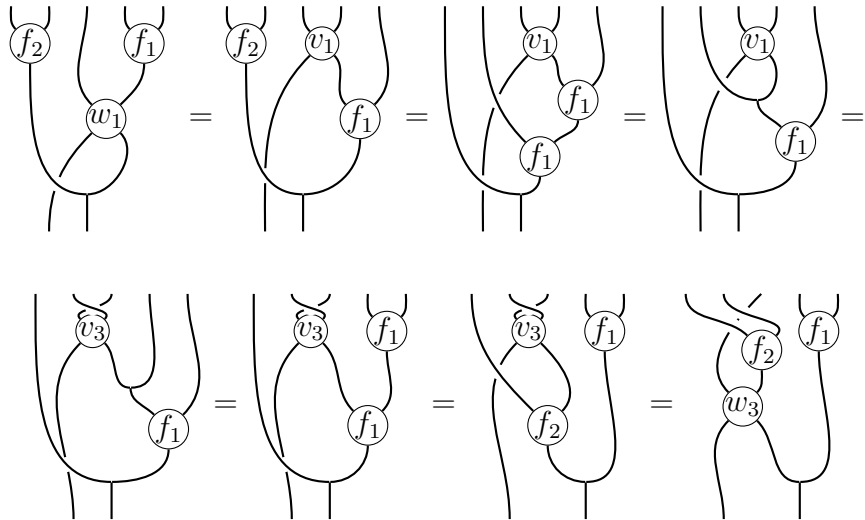
\begin{tikzpicture} 
\path 
(0,3) node[empty, name=in1] {}
(0.5,3) node[empty, name=in2] {}
(1,3) node[empty, name=in3] {}
(1.5,3) node[empty, name=in4] {}
(2,3) node[empty, name=in5] {}
(0.5,0) node[empty, name=out1] {}
(1,0) node[empty, name=out2] {}
(0.25,2.5) node[arr, name=f2] {$f_2$}
(1.75,2.5) node[arr, name=f1] {$f_1$}
(1.25,1.5) node[arr, name=w1] {$w_1$}
(1,0.5) node[empty, name=m] {};
\draw[braid] (in1) to[out=270,in=135] (f2);
\draw[braid] (in2) to[out=270, in=45] (f2);
\draw[braid] (in3) to[out=270,in=135] (w1);
\draw[braid] (in4) to[out=270,in=135] (f1);
\draw[braid] (in5) to[out=270, in=45] (f1);
\draw[braid] (f1) to[out=270,in=45] (w1);
\draw[braid, name path=bout1] (w1) to[out=225,in=90] (out1);
\draw[braid] (w1) to[out=315,in=0] (m);
\draw[braid] (m) to (out2);
\path[braid, name path=f2m] (f2) to[out=270,in=180] (m);
\fill[white, name intersections={of=f2m and bout1}] (intersection-1) circle(0.1);
\draw[braid] (f2) to[out=270,in=180] (m);
\path (2.5,1.5) node {$=$};
\end{tikzpicture}
\begin{tikzpicture} 
\path 
(0,3) node[empty, name=in1] {}
(0.5,3) node[empty, name=in2] {}
(1,3) node[empty, name=in3] {}
(1.5,3) node[empty, name=in4] {}
(2,3) node[empty, name=in5] {}
(0.5,0) node[empty, name=out1] {}
(1,0) node[empty, name=out2] {}
(0.25,2.5) node[arr, name=f2] {$f_2$}
(1.75,1.5) node[arr, name=f1] {$f_1$}
(1.25,2.5) node[arr, name=v1] {$v_1$}
(1,0.5) node[empty, name=m] {};
\draw[braid] (in1) to[out=270,in=135] (f2);
\draw[braid] (in2) to[out=270, in=45] (f2);
\draw[braid] (in3) to[out=270,in=135] (v1);
\draw[braid] (in4) to[out=270,in=45] (v1);
\draw[braid] (in5) to[out=270, in=45] (f1);
\draw[braid] (v1) to[out=315,in=135] (f1);
\draw[braid, name path=bout1] (v1) to[out=225,in=90] (out1);
\draw[braid] (f1) to[out=270,in=0] (m);
\draw[braid] (m) to (out2);
\path[braid, name path=f2m] (f2) to[out=270,in=180] (m);
\fill[white, name intersections={of=f2m and bout1}] (intersection-1) circle(0.1);
\draw[braid] (f2) to[out=270,in=180] (m);
\path (2.5,1.5) node {$=$};
\end{tikzpicture}
\begin{tikzpicture} 
\path 
(0,3) node[empty, name=in1] {}
(0.5,3) node[empty, name=in2] {}
(1,3) node[empty, name=in3] {}
(1.5,3) node[empty, name=in4] {}
(2,3) node[empty, name=in5] {}
(0.5,0) node[empty, name=out1] {}
(1,0) node[empty, name=out2] {}
(1.25,1.0) node[arr, name=f1d] {$f_1$}
(1.75,1.75) node[arr, name=f1] {$f_1$}
(1.25,2.5) node[arr, name=v1] {$v_1$}
(1,0.5) node[empty, name=m] {};
\path[braid, name path=bin1] (in1) to[out=270,in=180] (m);
\path[braid, name path=bin2] (in2) to[out=270, in=135] (f1d);
\draw[braid] (in3) to[out=270,in=135] (v1);
\draw[braid] (in4) to[out=270,in=45] (v1);
\draw[braid] (in5) to[out=270, in=45] (f1);
\draw[braid] (v1) to[out=315,in=135] (f1);
\draw[braid, name path=bout1] (v1) to[out=225,in=90] (out1);
\fill[white, name intersections={of=bin1 and bout1}] (intersection-1) circle(0.1);
\fill[white, name intersections={of=bin2 and bout1}] (intersection-1) circle(0.1);
\draw[braid] (in1) to[out=270,in=180] (m);
\draw[braid] (in2) to[out=270, in=135] (f1d);
\draw[braid] (f1) to[out=270,in=45] (f1d);
\draw[braid] (f1d) to[out=270,in=0] (m);
\draw[braid] (m) to (out2);
%
\path (2.5,1.5) node {$=$};
\end{tikzpicture}
\begin{tikzpicture} 
\path 
(0,3) node[empty, name=in1] {}
(0.5,3) node[empty, name=in2] {}
(1,3) node[empty, name=in3] {}
(1.5,3) node[empty, name=in4] {}
(2,3) node[empty, name=in5] {}
(0.5,0) node[empty, name=out1] {}
(1,0) node[empty, name=out2] {}
(1.75,1.2) node[arr, name=f1] {$f_1$}
(1.25,2.5) node[arr, name=v1] {$v_1$}
(1.25,1.75) node[empty, name=mu] {}
(1,0.5) node[empty, name=m] {};
\path[braid, name path=bin1] (in1) to[out=270,in=180] (m);
\path[braid, name path=bin2] (in2) to[out=270, in=180] (mu);
\draw[braid] (in3) to[out=270,in=135] (v1);
\draw[braid] (in4) to[out=270,in=45] (v1);
\draw[braid] (v1) to[out=315,in=0] (mu);
\draw[braid] (in5) to[out=270, in=45] (f1);
\draw[braid] (mu) to[out=270,in=135] (f1);
\draw[braid, name path=bout1] (v1) to[out=225,in=90] (out1);
\fill[white, name intersections={of=bin1 and bout1}] (intersection-1) circle(0.1);
\fill[white, name intersections={of=bin2 and bout1}] (intersection-1) circle(0.1);
\draw[braid] (in1) to[out=270,in=180] (m);
\draw[braid] (in2) to[out=270, in=180] (mu);
\draw[braid] (f1) to[out=270,in=0] (m);
\draw[braid] (m) to (out2);
%
\path (2.5,1.5) node {$=$};
\end{tikzpicture}
$$

$$
\begin{tikzpicture} 
\path 
(0,3) node[empty, name=in1] {}
(0.5,3) node[empty, name=in2] {}
(1,3) node[empty, name=in3] {}
(1.5,3) node[empty, name=in4] {}
(2,3) node[empty, name=in5] {}
(0.5,0) node[empty, name=out1] {}
(1,0) node[empty, name=out2] {}
(1.75,1.2) node[arr, name=f1] {$f_1$}
(0.75,2.5) node[arr, name=v3] {$v_3$}
(1.25,1.75) node[empty, name=mu] {}
(1,0.5) node[empty, name=m] {};
\path[braid, name path=bin1] (in1) to[out=270,in=180] (m);
\path[braid, name path=bin2] (in2) to[out=270, in=45] (v3);
\draw[braid, name path=bin3] (in3) to[out=270,in=135] (v3);
\draw[braid] (v3) to[out=315,in=180] (mu);
\draw[braid] (in4) to[out=270,in=0] (mu);
\draw[braid] (in5) to[out=270, in=45] (f1);
\draw[braid] (mu) to[out=270,in=135] (f1);
\draw[braid, name path=bout1] (v3) to[out=225,in=90] (out1);
\fill[white, name intersections={of=bin1 and bout1}] (intersection-1) circle(0.1);
\fill[white, name intersections={of=bin2 and bin3}] (intersection-1) circle(0.1);
\draw[braid] (in1) to[out=270,in=180] (m);
\draw[braid] (in2) to[out=270, in=45] (v3);
\draw[braid] (f1) to[out=270,in=0] (m);
\draw[braid] (m) to (out2);
\path (2.5,1.5) node {$=$};
\end{tikzpicture}
\begin{tikzpicture} 
\path 
(0,3) node[empty, name=in1] {}
(0.5,3) node[empty, name=in2] {}
(1,3) node[empty, name=in3] {}
(1.5,3) node[empty, name=in4] {}
(2,3) node[empty, name=in5] {}
(0.5,0) node[empty, name=out1] {}
(1,0) node[empty, name=out2] {}
(1.5,1.25) node[arr, name=f1d] {$f_1$}
(1.75,2.5) node[arr, name=f1] {$f_1$}
(0.75,2.5) node[arr, name=v3] {$v_3$}
(1,0.5) node[empty, name=m] {};
\path[braid, name path=bin1] (in1) to[out=270,in=180] (m);
\path[braid, name path=bin2] (in2) to[out=270, in=45] (v3);
\draw[braid, name path=bin3] (in3) to[out=270,in=135] (v3);
\draw[braid] (v3) to[out=315,in=135] (f1d);
\draw[braid] (in4) to[out=270,in=135] (f1);
\draw[braid] (in5) to[out=270, in=45] (f1);
\draw[braid] (f1) to[out=270,in=45] (f1d);
\draw[braid, name path=bout1] (v3) to[out=225,in=90] (out1);
\fill[white, name intersections={of=bin1 and bout1}] (intersection-1) circle(0.1);
\fill[white, name intersections={of=bin2 and bin3}] (intersection-1) circle(0.1);
\draw[braid] (in1) to[out=270,in=180] (m);
\draw[braid] (in2) to[out=270, in=45] (v3);
\draw[braid] (f1d) to[out=270,in=0] (m);
\draw[braid] (m) to (out2);
\path (2.5,1.5) node {$=$};
\end{tikzpicture}
\begin{tikzpicture} 
\path 
(0,3) node[empty, name=in1] {}
(0.5,3) node[empty, name=in2] {}
(1,3) node[empty, name=in3] {}
(1.5,3) node[empty, name=in4] {}
(2,3) node[empty, name=in5] {}
(0.5,0) node[empty, name=out1] {}
(1.5,0) node[empty, name=out2] {}
(1,1.25) node[arr, name=f2] {$f_2$}
(1.75,2.5) node[arr, name=f1] {$f_1$}
(0.75,2.5) node[arr, name=v3] {$v_3$}
(1.5,0.5) node[empty, name=m] {};
\path[braid, name path=bin1] (in1) to[out=270,in=135] (f2);
\path[braid, name path=bin2] (in2) to[out=270, in=45] (v3);
\draw[braid, name path=bin3] (in3) to[out=270,in=135] (v3);
\draw[braid] (v3) to[out=315,in=45] (f2);
\draw[braid] (in4) to[out=270,in=135] (f1);
\draw[braid] (in5) to[out=270, in=45] (f1);
\draw[braid] (f1) to[out=270,in=0] (m);
\draw[braid, name path=bout1] (v3) to[out=225,in=90] (out1);
\fill[white, name intersections={of=bin1 and bout1}] (intersection-1) circle(0.1);
\fill[white, name intersections={of=bin2 and bin3}] (intersection-1) circle(0.1);
\draw[braid] (in1) to[out=270,in=135] (f2);
\draw[braid] (in2) to[out=270, in=45] (v3);
\draw[braid] (f2) to[out=270,in=180] (m);
\draw[braid] (m) to (out2);
\path (2.5,1.5) node {$=$};
\end{tikzpicture}
\begin{tikzpicture} 
\path 
(0,3) node[empty, name=in1] {}
(0.5,3) node[empty, name=in2] {}
(1,3) node[empty, name=in3] {}
(1.5,3) node[empty, name=in4] {}
(2,3) node[empty, name=in5] {}
(0.5,0) node[empty, name=out1] {}
(1.5,0) node[empty, name=out2] {}
(1,2.25) node[arr, name=f2] {$f_2$}
(1.75,2.5) node[arr, name=f1] {$f_1$}
(0.75,1.5) node[arr, name=w3] {$w_3$}
(1.5,0.5) node[empty, name=m] {};
\path[braid, name path=bin1] (in1) to[out=270,in=135] (f2);
\path[braid, name path=bin2] (in2) to[out=270, in=45] (f2);
\draw[braid, name path=bin3] (in3) to[out=225,in=135] (w3);
\draw[braid] (f2) to[out=270,in=45] (w3);
\draw[braid] (in4) to[out=270,in=135] (f1);
\draw[braid] (in5) to[out=270, in=45] (f1);
\draw[braid] (f1) to[out=270,in=0] (m);
\draw[braid, name path=bout1] (w3) to[out=225,in=90] (out1);
\draw[braid] (w3) to[out=315,in=180] (m);
\fill[white, name intersections={of=bin1 and bin3}] (intersection-1) circle(0.1);
\fill[white, name intersections={of=bin2 and bin3}] (intersection-1) circle(0.1);
\draw[braid] (in1) to[out=270,in=135] (f2);
\draw[braid] (in2) to[out=270, in=45] (f2);
\draw[braid] (m) to (out2);
\end{tikzpicture}
$$
and the density of $f$. 

This proves (1); now we turn to (2). 
The counit condition follows easily by commutativity of the diagram
$$
\xymatrix{
VAB \ar@{=}[r] \ar[dd]_-{1f_1} &
VAB \ar@{=}[r] \ar[d]^-{v^11} &
VAB \ar[r]^-{1f_1} \ar[d]^-{1e1} &
VB \ar[d]^-{1e} \\
& VAB \ar[r]^-{1e1} \ar[d]^-{1f_1} &
VB \ar[r]^-{1e} &
V \ar@{=}[d] \\
VB \ar[r]_-{w^1} &
VB \ar[rr]_-{1e} &&
V}
$$
and density of $f$.

It remains only to prove the fusion law. For this, we use the fact that
$m1.1t_3$ is in \cq, since by \eqref{eq:t_2-3_compatibility} it is equal to
$d_2.1c^{-1}$. Now calculate as follows 

$$
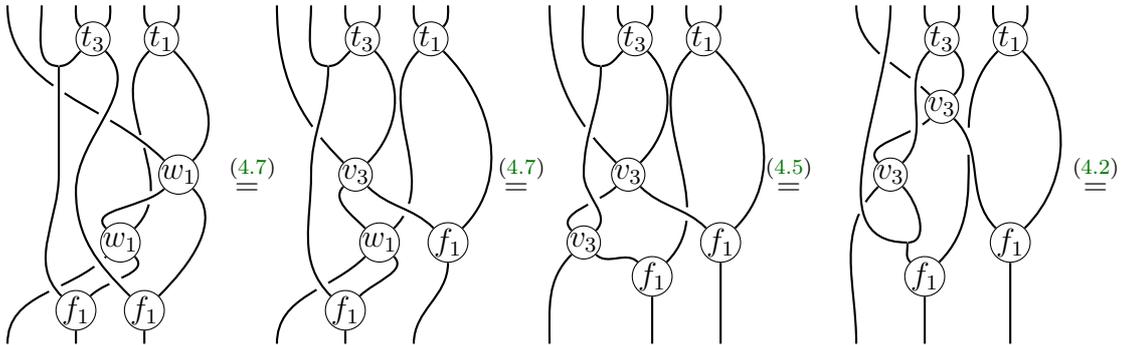
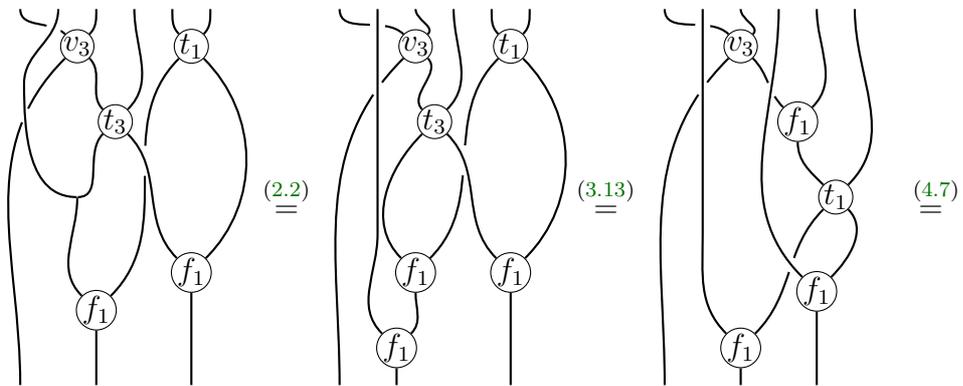
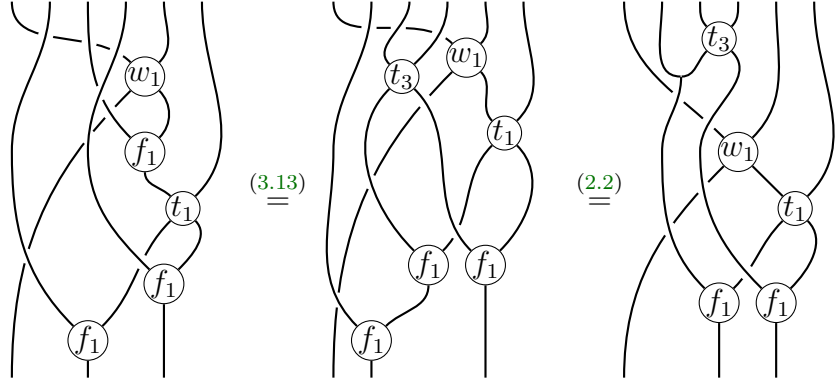
\begin{tikzpicture}[scale=.9] 
\path 
(0,5) node[empty, name=in1] {}  
(0.5,5) node[empty, name=in2] {}  
(1,5) node[empty, name=in3] {}  
(1.5,5) node[empty, name=in4] {}  
(2,5) node[empty, name=in5] {}  
(2.5,5) node[empty, name=in6] {}  
(0.0,0) node[empty, name=out1] {}  
(1.0,0) node[empty, name=out2] {}  
(2.0,0) node[empty, name=out3] {} 
(1.25,4.5) node[arr, name=t3] {$t_3$}
(2.25,4.5) node[arr, name=t1] {$t_1$}
(0.75,4.1) node[empty, name=m] {}
(0.75,2.5) node[empty, name=x] {}
(1.0,2.25) node[empty, name=y] {}
(2.5,2.5) node[arr, name=w1u] {$w_1$}
(1.65,1.5) node[arr, name=w1d] {$w_1$}
(1.0,0.5) node[arr, name=f1l] {$f_1$}
(2.0,0.5) node[arr, name=f1r] {$f_1$};
\path[braid, name path=bin1] (in1) to[out=270, in=135] (w1u);
\draw[braid] (in2) to[out=270,in=180] (m);
\draw[braid] (in3) to[out=270,in=135] (t3);
\draw[braid] (in4) to[out=270,in=45] (t3);
\draw[braid] (in5) to[out=270,in=135] (t1);
\draw[braid] (in6) to[out=270,in=45] (t1);
\draw[braid] (t3) to[out=225,in=0] (m);
\path[braid, name path=mf1l] (m) to[out=270,in=90] (x) to[out=270,in=135] (f1l); 
\path[braid, name path=t3f1r] (t3) to[out=315,in=90] (y) to[out=270,in=135] (f1r); 
\draw[braid] (t1) to[out=315,in=45] (w1u);
\draw[braid, name path=t1w1d] (t1) to[out=225,in=45] (w1d);
\fill[white,name intersections={of=bin1 and t1w1d}] (intersection-1) circle(0.1);
\draw[braid, name path=w1df1l] (w1d) to[out=315,in=45] (f1l);
\draw[braid, name path=bout1] (w1d) to[out=225,in=90] (out1);
\draw[braid] (in1) to[out=270, in=135] (w1u);
\fill[white,name intersections={of=bin1 and mf1l}] (intersection-1) circle(0.1);
\fill[white,name intersections={of=bout1 and mf1l}] (intersection-1) circle(0.1);
\fill[white,name intersections={of=bout1 and t3f1r}] (intersection-1) circle(0.1);
\fill[white,name intersections={of=w1df1l and t3f1r}] (intersection-1) circle(0.1);
\draw[braid] (m) to[out=270,in=90] (x) to[out=270,in=135] (f1l); 
\fill[white,name intersections={of=bin1 and t3f1r}] (intersection-1) circle(0.1);
\draw[braid] (t3) to[out=315,in=90] (y) to[out=270,in=135] (f1r); 
\draw[braid] (w1u) to[out=315,in=45] (f1r);
\path[braid, name path=w1uw1d] (w1u) to[out=225,in=135] (w1d);
\fill[white,name intersections={of=w1uw1d and t1w1d}] (intersection-1) circle(0.1);
\draw[braid] (w1u) to[out=225,in=135] (w1d);
\draw[braid] (f1l) to (out2);
\draw[braid] (f1r) to (out3);
\draw (3.5,2.5) node[empty] {$\stackrel{~\eqref{eq:w1-nd}}=$};
\end{tikzpicture}
\begin{tikzpicture}[scale=.9] 
\path 
(0,5) node[empty, name=in1] {}  
(0.5,5) node[empty, name=in2] {}  
(1,5) node[empty, name=in3] {}  
(1.5,5) node[empty, name=in4] {}  
(2,5) node[empty, name=in5] {}  
(2.5,5) node[empty, name=in6] {}  
(0.0,0) node[empty, name=out1] {}  
(1.0,0) node[empty, name=out2] {}  
(2.0,0) node[empty, name=out3] {} 
(1.25,4.5) node[arr, name=t3] {$t_3$}
(2.25,4.5) node[arr, name=t1] {$t_1$}
(0.75,4.1) node[empty, name=m] {}
(0.5,2.5) node[empty, name=x] {}
(1.0,2.25) node[empty, name=y] {}
(1.15,2.5) node[arr, name=v3u] {$v_3$}
(1.5,1.5) node[arr, name=w1d] {$w_1$}
(1.0,0.5) node[arr, name=f1l] {$f_1$}
(2.5,1.5) node[arr, name=f1r] {$f_1$};
\draw[braid, name path=bin1] (in1) to[out=270, in=135] (v3u);
\draw[braid] (in2) to[out=270,in=180] (m);
\draw[braid] (in3) to[out=270,in=135] (t3);
\draw[braid] (in4) to[out=270,in=45] (t3);
\draw[braid] (in5) to[out=270,in=135] (t1);
\draw[braid] (in6) to[out=270,in=45] (t1);
\draw[braid] (t3) to[out=225,in=0] (m);
\path[braid, name path=mf1l] (m) to[out=270,in=90] (x) to[out=270,in=135] (f1l); 
\draw[braid, name path=t3v3u] (t3) to[out=315,in=45] (v3u); 
\draw[braid] (t1) to[out=315,in=45] (f1r);
\draw[braid, name path=t1w1d] (t1) to[out=225,in=45] (w1d);
\fill[white,name intersections={of=bin1 and mf1l}] (intersection-1) circle(0.1);
\draw[braid, name path=w1df1l] (w1d) to[out=315,in=45] (f1l);
\draw[braid, name path=bout1] (w1d) to[out=225,in=90] (out1);
\fill[white,name intersections={of=bin1 and mf1l}] (intersection-1) circle(0.1);
\fill[white,name intersections={of=bout1 and mf1l}] (intersection-1) circle(0.1);
\path[braid, name path=v3uf1r] (v3u) to[out=315,in=135] (f1r);
\fill[white,name intersections={of=t1w1d and v3uf1r}] (intersection-1) circle(0.1);
\draw[braid] (m) to[out=270,in=90] (x) to[out=270,in=135] (f1l); 
\draw[braid] (v3u) to[out=315,in=135] (f1r);
\path[braid, name path=v3uw1d] (v3u) to[out=225,in=135] (w1d);
\draw[braid] (v3u) to[out=225,in=135] (w1d);
\draw[braid] (f1l) to (out2);
\draw[braid] (f1r) to[out=270,in=90]  (out3);
\draw (3.5,2.5) node[empty] {$\stackrel{~\eqref{eq:w1-nd}}=$};
\end{tikzpicture}
\begin{tikzpicture}[scale=.9] 
\path 
(0,5) node[empty, name=in1] {}  
(0.5,5) node[empty, name=in2] {}  
(1,5) node[empty, name=in3] {}  
(1.5,5) node[empty, name=in4] {}  
(2,5) node[empty, name=in5] {}  
(2.5,5) node[empty, name=in6] {}  
(0.0,0) node[empty, name=out1] {}  
(1.5,0) node[empty, name=out2] {}  
(2.5,0) node[empty, name=out3] {} 
(1.25,4.5) node[arr, name=t3] {$t_3$}
(2.25,4.5) node[arr, name=t1] {$t_1$}
(0.75,4.1) node[empty, name=m] {}
(0.5,2.5) node[empty, name=x] {}
(1.0,2.25) node[empty, name=y] {}
(1.15,2.5) node[arr, name=v3u] {$v_3$}
(0.5,1.5) node[arr, name=v3d] {$v_3$}
(1.5,1.0) node[arr, name=f1l] {$f_1$}
(2.5,1.5) node[arr, name=f1r] {$f_1$};
\draw[braid, name path=bin1] (in1) to[out=270, in=135] (v3u);
\draw[braid] (in2) to[out=270,in=180] (m);
\draw[braid] (in3) to[out=270,in=135] (t3);
\draw[braid] (in4) to[out=270,in=45] (t3);
\draw[braid] (in5) to[out=270,in=135] (t1);
\draw[braid] (in6) to[out=270,in=45] (t1);
\draw[braid] (t3) to[out=225,in=0] (m);
\path[braid, name path=mv3d] (m) to[out=270,in=90] (x) to[out=270,in=45] (v3d); 
\draw[braid, name path=t3v3u] (t3) to[out=315,in=45] (v3u); 
\draw[braid] (t1) to[out=315,in=45] (f1r);
\draw[braid, name path=t1f1l] (t1) to[out=225,in=45] (f1l);
\fill[white,name intersections={of=bin1 and mv3d}] (intersection-1) circle(0.1);
\draw[braid, name path=v3df1l] (v3d) to[out=315,in=135] (f1l);
\draw[braid, name path=bout1] (v3d) to[out=225,in=90] (out1);
\fill[white,name intersections={of=bin1 and mv3d}] (intersection-1) circle(0.1);
\fill[white,name intersections={of=bout1 and mv3d}] (intersection-1) circle(0.1);
\path[braid, name path=v3uf1r] (v3u) to[out=315,in=135] (f1r);
\fill[white,name intersections={of=t1f1l and v3uf1r}] (intersection-1) circle(0.1);
\draw[braid, name path=v3uv3d] (v3u) to[out=225,in=135] (v3d);
\draw[braid] (v3u) to[out=315,in=135] (f1r);
\fill[white,name intersections={of=v3uv3d and mv3d}] (intersection-1) circle(0.1);
\draw[braid] (m) to[out=270,in=90] (x) to[out=270,in=45] (v3d); 
\draw[braid] (f1l) to[out=270,in=90]  (out2);
\draw[braid] (f1r) to[out=270,in=90]  (out3);
\draw (3.5,2.5) node[empty] {$\stackrel{\eqref{eq:v^1-3_module_maps}}=$};
\end{tikzpicture}
\begin{tikzpicture}[scale=.9] 
\path 
(0,5) node[empty, name=in1] {}  
(0.5,5) node[empty, name=in2] {}  
(1,5) node[empty, name=in3] {}  
(1.5,5) node[empty, name=in4] {}  
(2,5) node[empty, name=in5] {}  
(2.5,5) node[empty, name=in6] {}  
(0.0,0) node[empty, name=out1] {}  
(1.0,0) node[empty, name=out2] {}  
(2.25,0) node[empty, name=out3] {} 
(1.25,4.5) node[arr, name=t3] {$t_3$}
(2.25,4.5) node[arr, name=t1] {$t_1$}
(0.75,1.5) node[empty, name=m] {}
(1.25,3.5) node[arr, name=v3u] {$v_3$}
(0.5,2.5) node[arr, name=v3d] {$v_3$}
(1.0,1.0) node[arr, name=f1l] {$f_1$}
(2.25,1.5) node[arr, name=f1r] {$f_1$};
\draw[braid, name path=bin1] (in1) to[out=270, in=135] (v3u);
\path[braid, name path=bin2] (in2) to[out=270,in=180] (m);
\draw[braid] (in3) to[out=270,in=135] (t3);
\draw[braid] (in4) to[out=270,in=45] (t3);
\draw[braid] (in5) to[out=270,in=135] (t1);
\draw[braid] (in6) to[out=270,in=45] (t1);
\path[braid, name path=t3uv3d] (t3) to[out=225,in=45] (v3d);
\path[braid, name path=mv3d] (m) to[out=270,in=90] (x) to[out=270,in=45] (v3d); 
\draw[braid, name path=t3v3u] (t3) to[out=315,in=45] (v3u); 
\draw[braid] (t1) to[out=315,in=45] (f1r);
\draw[braid, name path=t1f1l] (t1) to[out=225,in=45] (f1l);
\draw[braid, name path=v3uv3d] (v3u) to[out=225,in=135] (v3d);
\fill[white,name intersections={of=bin1 and t3uv3d}] (intersection-1) circle(0.1);
\fill[white,name intersections={of=v3uv3d and t3uv3d}] (intersection-1) circle(0.1);
\draw[braid] (t3) to[out=225,in=45] (v3d);
\path[braid, name path=v3uf1r] (v3u) to[out=315,in=135] (f1r);
\fill[white,name intersections={of=bin1 and bin2}] (intersection-1) circle(0.1);
\draw[braid, name path=v3dm] (v3d) to[out=315,in=0] (m);
\draw[braid, name path=bout1] (v3d) to[out=225,in=90] (out1);
\fill[white,name intersections={of=bin2 and bout1}] (intersection-1) circle(0.1);
\draw[braid] (in2) to[out=270,in=180] (m);
\fill[white,name intersections={of=t1f1l and v3uf1r}] (intersection-1) circle(0.2);
\draw[braid] (v3u) to[out=315,in=135] (f1r);
\draw[braid] (m) to[out=270,in=135] (f1l); 
\draw[braid] (f1l) to[out=270,in=90]  (out2);
\draw[braid] (f1r) to[out=270,in=90]  (out3);
\draw (3.5,2.5) node[empty] {$\stackrel{\eqref{eq:t_3_comodule}}=$};
\end{tikzpicture}
$$

$$
\begin{tikzpicture} 
\path 
(0,5) node[empty, name=in1] {}  
(0.5,5) node[empty, name=in2] {}  
(1,5) node[empty, name=in3] {}  
(1.5,5) node[empty, name=in4] {}  
(2,5) node[empty, name=in5] {}  
(2.5,5) node[empty, name=in6] {}  
(0.0,0) node[empty, name=out1] {}  
(1.0,0) node[empty, name=out2] {}  
(2.25,0) node[empty, name=out3] {} 
(1.25,3.5) node[arr, name=t3] {$t_3$}
(2.25,4.5) node[arr, name=t1] {$t_1$}
(0.75,2.5) node[empty, name=m] {}
(0.05,4.0) node[empty, name=x] {}
(0.75,4.5) node[arr, name=v3u] {$v_3$}
(1.0,1.0) node[arr, name=f1l] {$f_1$}
(2.25,1.5) node[arr, name=f1r] {$f_1$};
\draw[braid, name path=bin1] (in1) to[out=270, in=135] (v3u);
\path[braid, name path=bin2] (in2) to[out=270,in=90] (x) to[out=270,in=180] (m);
\draw[braid] (in3) to[out=270,in=45] (v3u);
\draw[braid] (in4) to[out=270,in=45] (t3);
\draw[braid] (in5) to[out=270,in=135] (t1);
\draw[braid] (in6) to[out=270,in=45] (t1);
\draw[braid] (t1) to[out=315,in=45] (f1r);
\draw[braid, name path=t1f1l] (t1) to[out=225,in=45] (f1l);
\draw[braid, name path=v3ut3] (v3u) to[out=315,in=135] (t3);
\fill[white,name intersections={of=bin1 and bin2}] (intersection-1) circle(0.1);
\path[braid, name path=t3f1r] (t3) to[out=315,in=135] (f1r);
\draw[braid] (t3) to[out=225,in=0] (m);
\draw[braid, name path=bout1] (v3u) to[out=225,in=90] (out1);
\fill[white,name intersections={of=bin2 and bout1}] (intersection-1) circle(0.1);
\draw[braid] (in2) to[out=270, in=90] (x) to[out=270,in=180] (m);
\fill[white,name intersections={of=t1f1l and t3f1r}] (intersection-1) circle(0.2);
\draw[braid] (t3) to[out=315,in=135] (f1r);
\draw[braid] (m) to[out=270,in=135] (f1l); 
\draw[braid] (f1l) to[out=270,in=90]  (out2);
\draw[braid] (f1r) to[out=270,in=90]  (out3);
\draw (3.5,2.5) node[empty] {$\stackrel{\eqref{eq:multiplicative}}=$};
\end{tikzpicture}
\begin{tikzpicture} 
\path 
(0,5) node[empty, name=in1] {}  
(0.5,5) node[empty, name=in2] {}  
(1,5) node[empty, name=in3] {}  
(1.5,5) node[empty, name=in4] {}  
(2,5) node[empty, name=in5] {}  
(2.5,5) node[empty, name=in6] {}  
(0.0,0) node[empty, name=out1] {}  
(.75,0) node[empty, name=out2] {}  
(2.25,0) node[empty, name=out3] {} 
(1.25,3.5) node[arr, name=t3] {$t_3$}
(2.25,4.5) node[arr, name=t1] {$t_1$}
(0.75,2.5) node[empty, name=m] {}
(0.5,2.0) node[empty, name=x] {}
(1,4.5) node[arr, name=v3u] {$v_3$}
(0.75,0.5) node[arr, name=f1d] {$f_1$}
(1.0,1.5) node[arr, name=f1l] {$f_1$}
(2.25,1.5) node[arr, name=f1r] {$f_1$};
\draw[braid, name path=bin1] (in1) to[out=270, in=135] (v3u);
\path[braid, name path=bin2] (in2) to[out=270, in=90] (x) to[out=270,in=135] (f1d);
\draw[braid] (in3) to[out=270,in=45] (v3u);
\draw[braid] (in4) to[out=270,in=45] (t3);
\draw[braid] (in5) to[out=270,in=135] (t1);
\draw[braid] (in6) to[out=270,in=45] (t1);
\draw[braid] (t1) to[out=315,in=45] (f1r);
\draw[braid, name path=t1f1l] (t1) to[out=225,in=45] (f1l);
\draw[braid, name path=v3ut3] (v3u) to[out=315,in=135] (t3);
\fill[white,name intersections={of=bin1 and bin2}] (intersection-1) circle(0.1);
\path[braid, name path=t3f1r] (t3) to[out=315,in=135] (f1r);
\draw[braid] (t3) to[out=225,in=135] (f1l);
\draw[braid, name path=bout1] (v3u) to[out=225,in=90] (out1);
\fill[white,name intersections={of=bin2 and bout1}] (intersection-1) circle(0.1);
\draw[braid] (in2) to[out=270, in=90] (x) to[out=270,in=135] (f1d);
\fill[white,name intersections={of=t1f1l and t3f1r}] (intersection-1) circle(0.2);
\draw[braid] (t3) to[out=315,in=135] (f1r);
\draw[braid] (f1l) to[out=270,in=45] (f1d); 
\draw[braid] (f1d) to[out=270,in=90]  (out2);
\draw[braid] (f1r) to[out=270,in=90]  (out3);
\draw (3.5,2.5) node[empty] {$\stackrel{\eqref{eq:multiplier-bimonoid-t3}}=$};
\end{tikzpicture}
\begin{tikzpicture} 
\path 
(0,5) node[empty, name=in1] {}  
(0.5,5) node[empty, name=in2] {}  
(1,5) node[empty, name=in3] {}  
(1.5,5) node[empty, name=in4] {}  
(2,5) node[empty, name=in5] {}  
(2.5,5) node[empty, name=in6] {}  
(0.0,0) node[empty, name=out1] {}  
(1.0,0) node[empty, name=out2] {}  
(2.0,0) node[empty, name=out3] {} 
(2.25,2.5) node[arr, name=t1] {$t_1$}
(0.5,2.0) node[empty, name=x] {}
(1,4.5) node[arr, name=v3u] {$v_3$}
(1,0.5) node[arr, name=f1d] {$f_1$}
(1.75,3.5) node[arr, name=f1l] {$f_1$}
(2.0,1.25) node[arr, name=f1r] {$f_1$};
\draw[braid, name path=bin1] (in1) to[out=270, in=135] (v3u);
\path[braid, name path=bin2] (in2) to[out=270, in=90] (x) to[out=270,in=135] (f1d);
\draw[braid] (in3) to[out=270,in=45] (v3u);
\path[braid, name path=bin4] (in4) to[out=270,in=135] (f1r);
\draw[braid] (in5) to[out=270,in=45] (f1l);
\draw[braid] (in6) to[out=270,in=45] (t1);
\draw[braid] (t1) to[out=315,in=45] (f1r);
\draw[braid, name path=t1f1d] (t1) to[out=225,in=45] (f1d);
\draw[braid, name path=v3uf1l] (v3u) to[out=315,in=135] (f1l);
\draw[braid] (f1l) to[out=270,in=135] (t1);
\fill[white,name intersections={of=bin1 and bin2}] (intersection-1) circle(0.1);
\draw[braid, name path=bout1] (v3u) to[out=225,in=90] (out1);
\fill[white,name intersections={of=bin2 and bout1}] (intersection-1) circle(0.1);
\draw[braid] (in2) to[out=270, in=90] (x) to[out=270,in=135] (f1d);
\fill[white,name intersections={of=t1f1d and bin4}] (intersection-1) circle(0.1);
\fill[white,name intersections={of=v3uf1l and bin4}] (intersection-1) circle(0.1);
\draw[braid] (in4) to[out=270,in=135] (f1r);
\draw[braid] (f1d) to[out=270,in=90]  (out2);
\draw[braid] (f1r) to[out=270,in=90]  (out3);
\draw (3.5,2.5) node[empty] {$\stackrel{~\eqref{eq:w1-nd}}=$};
\end{tikzpicture}
$$

$$
\begin{tikzpicture} 
\path 
(0,5) node[empty, name=in1] {}  
(0.5,5) node[empty, name=in2] {}  
(1,5) node[empty, name=in3] {}  
(1.5,5) node[empty, name=in4] {}  
(2,5) node[empty, name=in5] {}  
(2.5,5) node[empty, name=in6] {}  
(0.0,0) node[empty, name=out1] {}  
(1.0,0) node[empty, name=out2] {}  
(2.0,0) node[empty, name=out3] {} 
(2.25,2.25) node[arr, name=t1] {$t_1$}
(0.0,3.0) node[empty, name=x] {}
(1.0,3.0) node[empty, name=y] {}
(1.75,4.0) node[arr, name=w1] {$w_1$}
(1,0.5) node[arr, name=f1l] {$f_1$}
(1.75,3.0) node[arr, name=f1u] {$f_1$}
(2.0,1.25) node[arr, name=f1r] {$f_1$};
\draw[braid, name path=bin1] (in1) to[out=270, in=135] (w1);
\path[braid, name path=bin2] (in2) to[out=270, in=90] (x) to[out=270,in=135] (f1l);
\path[braid, name path=bin3] (in3) to[out=270,in=135] (f1u);
\path[braid, name path=bin4] (in4) to[out=270, in=90] (y) to[out=270, in=135] (f1r);
\draw[braid] (in5) to[out=270,in=45] (w1);
\draw[braid] (in6) to[out=270,in=45] (t1);
\draw[braid, name path=bout1] (w1) to[out=225,in=90] (out1);
\draw[braid] (t1) to[out=315,in=45] (f1r);
\draw[braid, name path=t1f1l] (t1) to[out=225,in=45] (f1l);
\draw[braid, name path=w1f1u] (w1) to[out=315,in=45] (f1u);
\draw[braid] (f1u) to[out=270,in=135] (t1);
\fill[white,name intersections={of=bin1 and bin2}] (intersection-1) circle(0.1);
\fill[white,name intersections={of=bin1 and bin3}] (intersection-1) circle(0.1);
\fill[white,name intersections={of=bin1 and bin4}] (intersection-1) circle(0.1);
\fill[white,name intersections={of=bin3 and bout1}] (intersection-1) circle(0.1);
\fill[white,name intersections={of=bin4 and bout1}] (intersection-1) circle(0.1);
\draw[braid] (in3) to[out=270,in=135] (f1u);
\fill[white,name intersections={of=bin3 and bin4}] (intersection-1) circle(0.1);
\fill[white,name intersections={of=bin2 and bout1}] (intersection-1) circle(0.1);
\draw[braid] (in2) to[out=270, in=90] (x) to[out=270,in=135] (f1l);
\fill[white,name intersections={of=t1f1l and bin4}] (intersection-1) circle(0.1);
\draw[braid] (in4) to[out=270,in=90] (y) to[out=270, in=135] (f1r);
\draw[braid] (f1d) to[out=270,in=90]  (out2);
\draw[braid] (f1r) to[out=270,in=90]  (out3);
\draw (3.5,2.5) node[empty] {$\stackrel{\eqref{eq:multiplier-bimonoid-t3}}=$};
\end{tikzpicture}
\begin{tikzpicture} 
\path 
(0,5) node[empty, name=in1] {}  
(0.5,5) node[empty, name=in2] {}  
(1,5) node[empty, name=in3] {}  
(1.5,5) node[empty, name=in4] {}  
(2,5) node[empty, name=in5] {}  
(2.5,5) node[empty, name=in6] {}  
(0.0,0) node[empty, name=out1] {}  
(.5,0) node[empty, name=out2] {}  
(2.0,0) node[empty, name=out3] {} 
(0.9,4.0) node[arr, name=t3] {$t_3$}
(2.25,3.25) node[arr, name=t1] {$t_1$}
(0.0,3.0) node[empty, name=x] {}
(1.0,3.0) node[empty, name=y] {}
(1.75,4.25) node[arr, name=w1] {$w_1$}
(0.5,0.5) node[arr, name=f1d] {$f_1$}
(1.25,1.5) node[arr, name=f1l] {$f_1$}
(2.0,1.5) node[arr, name=f1r] {$f_1$};
\draw[braid, name path=bin1] (in1) to[out=270, in=135] (w1);
\path[braid, name path=bin2] (in2) to[out=270, in=90] (x) to[out=270,in=135] (f1d);
\path[braid, name path=bin3] (in3) to[out=270,in=135] (t3);
\path[braid, name path=bin4] (in4) to[out=270,  in=45] (t3);
\draw[braid] (in5) to[out=270,in=45] (w1);
\draw[braid] (in6) to[out=270,in=45] (t1);
\draw[braid, name path=bout1] (w1) to[out=225,in=90] (out1);
\draw[braid] (t1) to[out=315,in=45] (f1r);
\draw[braid, name path=t1f1l] (t1) to[out=225,in=45] (f1l);
\draw[braid, name path=w1t1] (w1) to[out=315,in=135] (t1);
\fill[white,name intersections={of=bin1 and bin2}] (intersection-1) circle(0.1);
\fill[white,name intersections={of=bin1 and bin3}] (intersection-1) circle(0.1);
\fill[white,name intersections={of=bin1 and bin4}] (intersection-1) circle(0.1);
\draw[braid] (in3) to[out=270,in=135] (t3);
\path[braid, name path=t3f1l] (t3) to[out=225,in=135] (f1l);
\path[braid, name path=t3f1r] (t3) to[out=315,in=135] (f1r);
\fill[white,name intersections={of=t3f1l and bout1}] (intersection-1) circle(0.1);
\fill[white,name intersections={of=t3f1r and bout1}] (intersection-1) circle(0.1);
\fill[white,name intersections={of=t3f1r and t1f1l}] (intersection-1) circle(0.1);
\draw[braid] (t3) to[out=225,in=135] (f1l);
\draw[braid] (t3) to[out=315,in=135] (f1r);
\draw[braid] (f1l) to[out=270,in=45] (f1d);
\fill[white,name intersections={of=bin2 and bout1}] (intersection-1) circle(0.1);
\draw[braid] (in2) to[out=270, in=90] (x) to[out=270,in=135] (f1d);
\draw[braid] (in4) to[out=270,in=45] (t3);
\draw[braid] (f1d) to[out=270,in=90]  (out2);
\draw[braid] (f1r) to[out=270,in=90]  (out3);
\draw (3.5,2.5) node[empty] {$\stackrel{\eqref{eq:multiplicative}}=$};
\end{tikzpicture}
\begin{tikzpicture} 
\path 
(0,5) node[empty, name=in1] {}  
(0.5,5) node[empty, name=in2] {}  
(1,5) node[empty, name=in3] {}  
(1.5,5) node[empty, name=in4] {}  
(2,5) node[empty, name=in5] {}  
(2.5,5) node[empty, name=in6] {}  
(0.0,0) node[empty, name=out1] {}  
(1.25,0) node[empty, name=out2] {}  
(2.0,0) node[empty, name=out3] {} 
(0.75,4) node[empty, name=m] {}
(1.25,4.5) node[arr, name=t3] {$t_3$}
(2.25,2.25) node[arr, name=t1] {$t_1$}
 (0.5,3.0) node[empty, name=x] {}
(1.0,3.0) node[empty, name=y] {}
(1.5,3.0) node[arr, name=w1] {$w_1$}
(1.25,1.0) node[arr, name=f1l] {$f_1$}
(2.0,1.0) node[arr, name=f1r] {$f_1$};
\draw[braid, name path=bin1] (in1) to[out=270, in=135] (w1);
\draw[braid] (in2) to[out=270,in=180] (m);
\draw[braid] (in3) to[out=270,in=135] (t3);
\draw[braid] (in4) to[out=270,in=45] (t3);
\draw[braid] (in5) to[out=270,in=45] (w1);
\draw[braid] (in6) to[out=270,in=45] (t1);
\draw[braid] (t3) to[out=225,in=0] (m);
\path[braid, name path=mf1l] (m) to[out=270,in=90] (x) to[out=270, in=135] (f1l);
\path[braid, name path=t3f1r] (t3) to[out=315,in=90] (y) to[out=270, in=135] (f1r);
\draw[braid, name path=bout1] (w1) to[out=225,in=90] (out1);
\draw[braid] (t1) to[out=315,in=45] (f1r);
\draw[braid, name path=t1f1l] (t1) to[out=225,in=45] (f1l);
\draw[braid, name path=w1t1] (w1) to[out=315,in=135] (t1);
\fill[white,name intersections={of=bin1 and mf1l}] (intersection-1) circle(0.1);
\fill[white,name intersections={of=bin1 and t3f1r}] (intersection-1) circle(0.1);
\fill[white,name intersections={of=bout1 and mf1l}] (intersection-1) circle(0.1);
\fill[white,name intersections={of=bout1 and t3f1r}] (intersection-1) circle(0.1);
\fill[white,name intersections={of=t1f1l and t3f1r}] (intersection-1) circle(0.1);
\draw[braid] (m) to[out=270,in=90] (x) to[out=270, in=135] (f1l);
\draw[braid] (t3) to[out=315,in=90] (y) to[out=270, in=135] (f1r);
\draw[braid] (f1l) to[out=270,in=90]  (out2);
\draw[braid] (f1r) to[out=270,in=90]  (out3);
\end{tikzpicture}
$$
and now cancel $m1.1t_3$ and use Remark~\ref{rem:f_non-degenerate}.
 
The claim that a morphism of $A$-comodules is compatible also with the
$B$-coaction induced by $f$, as well as the composition law of the resulting
functors, are clear by the construction. The identity morphism induces the
identity functor in light of \eqref{eq:v^1-3_module_maps}.
\end{proof}

Proposition~\ref{prop:f_*}~(1) can be applied in particular to the comodule
$(A,t_1,t_3)$; then any dense and multiplicative \mM-morphism $f:A \nto B$
induces morphisms as in the right verticals of
$$
\xymatrix{
A^2B \ar[r]^{1f_1} \ar[d]_{t_11} & AB \ar[d]^{t_1^f} \\
A^2B \ar[r]_{1f_1} & AB } \quad
\xymatrix{
A^2B \ar[r]^{1c^{-1}} \ar[d]_{t_3 1} & ABA \ar[r]^{1f_2} & AB \ar[d]^{t_3^f} \\
A^2B \ar[r]_{1c^{-1}} & ABA \ar[r]_{1f_2} & AB. }
$$
If a pair $(v^1,v^3)$ obeys \eqref{eq:comodule_compatibility} for the
semigroup $A$ then $(v^3,v^1)$ obeys \eqref{eq:comodule_compatibility} for
$A\op$ in the braided monoidal category with inverse 
braiding. Thus by Proposition~\ref{prop:f_*} (1), any dense and multiplicative
\mM-morphism $g:A\op \nto B$ induces morphisms as in the right verticals of
\begin{equation} \label{eq:g_*}
\xymatrix{
VAB \ar[r]^-{1g_1} \ar[d]_-{v^31} &
VB \ar[d]^-{(g_*v)^1}
&&
VAB \ar[r]^-{1c} \ar[d]_-{v^11} &
VBA \ar[r]^-{1g_2} &
VB \ar[d]^-{(g_*v)^3} \\
VAB \ar[r]_-{1g_1} &
VB
&&
VAB \ar[r]_-{1c} &
VBA \ar[r]_-{1g_2} &
VB.}
\end{equation}
In particular, there are morphisms as in the right verticals of
$$
\xymatrix{
A^2B \ar[r]^-{1g_1} \ar[d]_-{t_31} &
AB \ar[d]^-{t_3^g}
&&
A^2B \ar[r]^-{1c} \ar[d]_-{t_11} &
ABA \ar[r]^-{1g_2} &
AB \ar[d]^-{t_1^g} \\
A^2B \ar[r]_-{1g_1} &
AB
&&
A^2B \ar[r]_-{1c} &
ABA \ar[r]_-{1g_2} &
AB.}
$$

If $A$ and $B$ are bimonoids, then $(f_*v)^1$ in Proposition~\ref{prop:f_*}
and $(g_*v)^1$ above are mutually inverse isomorphisms in \cc whenever $f$ and
$g$ are mutual inverses for the so-called convolution product. For multiplier
bimonoids, however, the convolution monoid is not available; instead, we have
the following generalization.  

\begin{lemma} \label{lem:conv_inv}
Let $(A,t_1,t_2,t_3,t_4,e)$ and $(B,t_1,t_2,t_3,t_4,e)$ be regular multiplier
bimonoids having the properties listed in Section \ref{sect:mbm}. 
Let $f\colon A\nto B$ and $g:A\op \nto B$ be dense multiplicative
\mM-morphisms. Using the construction and the notation of
Proposition~\ref{prop:f_*} and the previous paragraph, the following
assertions are equivalent to each other.
\begin{itemize}
\item[{(a)}] $(g_*v)^1.(f_*v)^1=1$, for any comodule $(v^1,v^3)$.   
\item[{(b)}] $g_1.t_1^f=e1$.
\end{itemize}
Symmetrically, also the following assertions are equivalent to each other.
\begin{itemize}
\item[{ (a')}] $(f_*v)^1.(g_*v)^1=1$, for any comodule $(v^1,v^3)$.   
\item[{ (b')}] $f_1.t_3^g=e1$.
\end{itemize}
\end{lemma}

\begin{proof}
We only prove the equivalence of  (a) and (b); the equivalence of
the primed assertions follows by symmetry. 

(a)$\Rightarrow$(b).  By the density of $g$, the computation 
$$
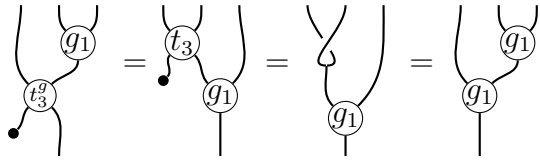
\begin{tikzpicture} 
\path[unit] (.4,.3) node [unit,name=e] {}
(.75,.8) node [arr,name=t3g] {${}_{t_3^g}$}
(1.25,1.5) node [arr,name=g1] {$g_1$};
\draw[braid] (.5,2) to[out=270,in=135] (t3g);
\draw[braid] (1,2) to[out=270,in=135] (g1);
\draw[braid] (1.5,2) to[out=270,in=45] (g1);
\draw[braid] (g1) to[out=270,in=45] (t3g);
\draw[braid] (t3g) to[out=225,in=45] (e);
\draw[braid] (t3g) to[out=315,in=90] (1,0);
\draw (2,1.2) node {$=$};
\end{tikzpicture}
\begin{tikzpicture} 
\path[unit] (.5,1) node [unit,name=e] {}
(.75,1.5) node [arr,name=t3] {$t_3$}
(1.25,.8) node [arr,name=g1] {$g_1$};
\draw[braid] (.5,2) to[out=270,in=135] (t3);
\draw[braid] (1,2) to[out=270,in=45] (t3);
\draw[braid] (1.5,2) to[out=270,in=45] (g1);
\draw[braid] (t3) to[out=315,in=135] (g1);
\draw[braid] (t3) to[out=225,in=45] (e);
\draw[braid] (g1) to[out=270,in=90] (1.25,0);
\draw (2,1.2) node {$=$};
\end{tikzpicture}
\begin{tikzpicture} 
\path (.75,1.2) node [empty,name=m] {}
(1,.5) node [arr,name=g1] {$g_1$};
\draw[braid,name path=d] (.5,2) to[out=270,in=0] (m);
\path[braid,name path=u] (1,2) to[out=270,in=180] (m);
\fill[white,name intersections={of=u and d}] (intersection-1) circle(0.1);
\draw[braid] (1,2) to[out=270,in=180] (m);
\draw[braid] (1.5,2) to[out=270,in=45] (g1);
\draw[braid] (m) to[out=270,in=135] (g1);
\draw[braid] (g1) to[out=270,in=90] (1,0);
\draw (2,1.2) node {$=$};
\end{tikzpicture}
\begin{tikzpicture} 
\path (.75,.8) node [arr,name=g1d] {$g_1$}
(1.25,1.5) node [arr,name=g1u] {$g_1$};
\draw[braid] (.5,2) to[out=270,in=135] (g1d);
\draw[braid] (1,2) to[out=270,in=135] (g1u);
\draw[braid] (1.5,2) to[out=270,in=45] (g1u);
\draw[braid] (g1u) to[out=270,in=45] (g1d);
\draw[braid] (g1d) to[out=270,in=90] (.75,0);
\end{tikzpicture}
$$
implies that the bottom region of
$$
\xymatrix{AB \ar[r]^-{t_1^f} \ar@{=}[d] &
AB \ar[d]^-{g_1} \ar[ld]_-{t_3^g} \\
AB \ar[r]_-{e1} &
B}
$$
commutes. The top region commutes by  (a) applied to $(t_1,t_3)$.

 (b)$\Rightarrow$(a). Using (b) in the first equality,
$$
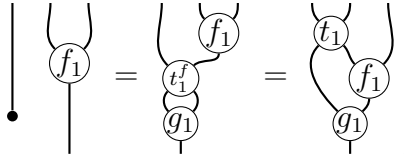
\begin{tikzpicture} 
\path[unit] (.5,.5) node [unit,name=e] {}
(1.25,1.2) node [arr,name=f1] {$f_1$};
\draw[braid] (.5,2) to[out=270,in=90] (e);
\draw[braid] (1,2) to[out=270,in=135] (f1);
\draw[braid] (1.5,2) to[out=270,in=45] (f1);
\draw[braid] (f1) to[out=270,in=90] (1.25,0);
\draw (2,1) node {$=$};
\end{tikzpicture}
\begin{tikzpicture} 
\path (1.25,1.6) node [arr,name=f1] {$f_1$}
(.75,1) node [arr,name=t1f] {${}_{t_1^f}$}
(.75,.4) node [arr,name=g1] {$g_1$};
\draw[braid] (.5,2) to[out=270,in=135] (t1f);
\draw[braid] (1,2) to[out=270,in=135] (f1);
\draw[braid] (1.5,2) to[out=270,in=45] (f1);
\draw[braid] (f1) to[out=270,in=45] (t1f);
\draw[braid] (t1f) to[out=225,in=135] (g1);
\draw[braid] (t1f) to[out=315,in=45] (g1);
\draw[braid] (g1) to[out=270,in=90] (.75,0);
\draw (2,1) node {$=$};
\end{tikzpicture}
\begin{tikzpicture} 
\path (1.25,1) node [arr,name=f1] {$f_1$}
(.75,1.6) node [arr,name=t1] {$t_1$}
(1,.4) node [arr,name=g1] {$g_1$};
\draw[braid] (.5,2) to[out=270,in=135] (t1);
\draw[braid] (1,2) to[out=270,in=45] (t1);
\draw[braid] (1.5,2) to[out=270,in=45] (f1);
\draw[braid] (t1) to[out=315,in=135] (f1);
\draw[braid] (f1) to[out=270,in=45] (g1);
\draw[braid] (t1) to[out=225,in=135] (g1);
\draw[braid] (g1) to[out=270,in=90] (1,0);
\end{tikzpicture} .
$$
Using this in the second and the penultimate equalities, and writing $w$ for
$f_*v$ and $\widetilde w$ for $g_*v$ for brevity, we see that 
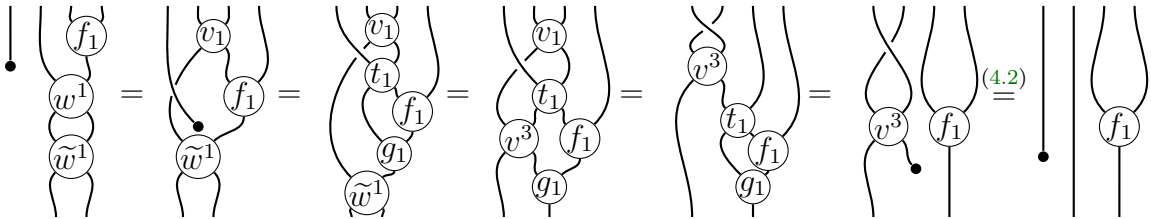
\begin{eqnarray*}
\begin{tikzpicture}[scale=.8] 
\path[unit] (.5,2.5) node [unit,name=e] {}
(1.75,3) node [arr,name=f1] {$f_1$}
(1.5,2) node [arr,name=w] {$w^1$}
(1.5,1) node [arr,name=wt] {$\widetilde w^1$};
\draw[braid] (.5,3.5) to[out=270,in=90] (e);
\draw[braid] (1,3.5) to[out=270,in=135] (w);
\draw[braid] (1.5,3.5) to[out=270,in=135] (f1);
\draw[braid] (2,3.5) to[out=270,in=45] (f1);
\draw[braid] (f1) to[out=270,in=45] (w);
\draw[braid] (w) to[out=225,in=135] (wt);
\draw[braid] (w) to[out=315,in=45] (wt);
\draw[braid] (wt) to[out=225,in=90] (1.25,0);
\draw[braid] (wt) to[out=315,in=90] (1.75,0);
\draw (2.5,2) node {$=$};
\end{tikzpicture}
\begin{tikzpicture}[scale=.8] 
\path[unit] (1,1.5) node [unit,name=e] {}
(1.75,2) node [arr,name=f1] {$f_1$}
(1.25,3) node [arr,name=v1] {$v_1$}
(1,1) node [arr,name=wt] {$\widetilde w^1$};
\path[braid,name path=u] (.5,3.5) to[out=270,in=135] (e);
\draw[braid] (1,3.5) to[out=270,in=135] (v1);
\draw[braid] (1.5,3.5) to[out=270,in=45] (v1);
\draw[braid] (2,3.5) to[out=270,in=45] (f1);
\draw[braid] (v1) to[out=315,in=135] (f1);
\draw[braid] (f1) to[out=270,in=45] (wt);
\draw[braid,name path=d] (v1) to[out=225,in=135] (wt);
\draw[braid] (wt) to[out=225,in=90] (.75,0);
\draw[braid] (wt) to[out=315,in=90] (1.25,0);
\fill[white,name intersections={of=u and d}] (intersection-1) circle(0.1);
\draw[braid] (.5,3.5) to[out=270,in=135] (e);
\draw (2.5,2) node {$=$};
\end{tikzpicture} 
\begin{tikzpicture}[scale=.8]   
\path (1.75,1.75) node [arr,name=f1] {$f_1$}
(1.25,3.1) node [arr,name=v1] {$v_1$}
(1.25,2.35) node [arr,name=t1] {$t_1$}
(1.45,1.05) node [arr,name=g1] {$g_1$}
(1,.4) node [arr,name=wt] {$\widetilde w^1$};
\path[braid,name path=u] (.5,3.5) to[out=270,in=135] (t1);
\draw[braid] (1,3.5) to[out=270,in=135] (v1);
\draw[braid] (1.5,3.5) to[out=270,in=45] (v1);
\draw[braid] (2,3.5) to[out=270,in=45] (f1);
\draw[braid] (v1) to[out=315,in=45] (t1);
\draw[braid,name path=d] (v1) to[out=225,in=135] (wt);
\draw[braid] (t1) to[out=315,in=135] (f1);
\draw[braid] (t1) to[out=225,in=135] (g1);
\draw[braid] (f1) to[out=270,in=45] (g1);
\draw[braid] (g1) to[out=270,in=45] (wt);
\draw[braid] (wt) to[out=225,in=90] (.75,0);
\draw[braid] (wt) to[out=315,in=90] (1.25,0);
\fill[white,name intersections={of=u and d}] (intersection-1) circle(0.1);
\draw[braid]  (.5,3.5) to[out=270,in=135] (t1);
\draw (2.5,2) node {$=$};
\end{tikzpicture} 
\begin{tikzpicture}[scale=.8] 
\path (1.75,1.3) node [arr,name=f1] {$f_1$}
(1.25,3) node [arr,name=v1] {$v_1$}
(1.25,2) node [arr,name=t1] {$t_1$}
(1.25,.5) node [arr,name=g1] {$g_1$}
(.75,1.3) node [arr,name=v3] {$v^3$};
\path[braid,name path=u] (.5,3.5) to[out=270,in=135] (t1);
\draw[braid] (1,3.5) to[out=270,in=135] (v1);
\draw[braid] (1.5,3.5) to[out=270,in=45] (v1);
\draw[braid] (2,3.5) to[out=270,in=45] (f1);
\draw[braid] (v1) to[out=315,in=45] (t1);
\draw[braid,name path=d] (v1) to[out=225,in=135] (v3);
\draw[braid] (t1) to[out=315,in=135] (f1);
\draw[braid] (t1) to[out=225,in=45] (v3);
\draw[braid] (f1) to[out=270,in=45] (g1);
\draw[braid] (v3) to[out=315,in=135] (g1);
\draw[braid] (g1) to[out=270,in=90] (1.25,0);
\draw[braid] (v3) to[out=225,in=90] (.5,0);
\fill[white,name intersections={of=u and d}] (intersection-1) circle(0.1);
\draw[braid]  (.5,3.5) to[out=270,in=135] (t1);
\draw (2.6,2) node {$=$};
\end{tikzpicture} 
\begin{tikzpicture}[scale=.8] 
\path (1.75,1.1) node [arr,name=f1] {$f_1$}
(1.25,1.6) node [arr,name=t1] {$t_1$}
(1.5,.5) node [arr,name=g1] {$g_1$}
(.75,2.5) node [arr,name=v3] {$v^3$};
\path[braid,name path=u] (.5,3.5) to[out=270,in=45] (v3);
\draw[braid,name path=d] (1,3.5) to[out=270,in=135] (v3);
\draw[braid] (1.5,3.5) to[out=270,in=45] (t1);
\draw[braid] (2,3.5) to[out=270,in=45] (f1);
\draw[braid] (v3) to[out=315,in=135] (t1);
\draw[braid] (t1) to[out=315,in=135] (f1);
\draw[braid] (t1) to[out=225,in=135] (g1);
\draw[braid] (f1) to[out=270,in=45] (g1);
\draw[braid] (g1) to[out=270,in=90] (1.5,0);
\draw[braid] (v3) to[out=225,in=90] (.5,0);
\fill[white,name intersections={of=u and d}] (intersection-1) circle(0.1);
\draw[braid]  (.5,3.5) to[out=270,in=45] (v3);
\draw (2.6,2) node {$=$};
\end{tikzpicture} 
\begin{tikzpicture}[scale=.8] 
\path[unit] (1.2,.8) node [unit,name=e] {}
(1.75,1.5) node [arr,name=f1] {$f_1$}
(.75,1.5) node [arr,name=v3] {$v^3$};
\path[braid,name path=u] (.5,3.5) to[out=270,in=45] (v3);
\draw[braid,name path=d] (1,3.5) to[out=270,in=135] (v3);
\draw[braid] (1.5,3.5) to[out=270,in=135] (f1);
\draw[braid] (2,3.5) to[out=270,in=45] (f1);
\draw[braid] (f1) to[out=270,in=90] (1.75,0);
\draw[braid] (v3) to[out=225,in=90] (.5,0);
\draw[braid] (v3) to[out=315,in=135] (e);
\fill[white,name intersections={of=u and d}] (intersection-1) circle(0.1);
\draw[braid] (.5,3.5) to[out=270,in=45] (v3);
\draw (2.65,2.2) node {$\stackrel{\eqref{eq:t_3_comodule}} = $};
\end{tikzpicture} 
\begin{tikzpicture}[scale=.8] 
\path[unit] (.5,1) node [unit,name=e] {}
(1.75,1.5) node [arr,name=f1] {$f_1$};
\draw[braid] (.5,3.5) to[out=270,in=90] (e);
\draw[braid] (1,3.5) to[out=270,in=90] (1,0);
\draw[braid] (1.5,3.5) to[out=270,in=135] (f1);
\draw[braid] (2,3.5) to[out=270,in=45] (f1);
\draw[braid] (f1) to[out=270,in=90] (1.75,0);
\end{tikzpicture} 
\end{eqnarray*}
where the fourth equality holds by the identity
$$
\begin{tikzpicture}
  \path 
(0.75,2) node[arr,name=t1u] {$v^1$} 
(0.75,1) node[arr,name=t1d]  {$t_1$} 
(0.25,0.5) node[arr,name=t3] {$v^3$}; 
\path 
(1,2.5) node[empty,name=in3] {} 
(0.5,2.5) node[empty,name=in2] {} 
(0,2.5) node[empty,name=in1] {} 
(0,0) node[empty,name=out1] {} 
(0.5,0) node[empty,name=out2] {} 
(1,0) node[empty,name=out3] {}; 
\draw[braid] (in3) to[out=270,in=45] (t1u);
\draw[braid] (in2) to[out=270,in=135] (t1u);
\draw[braid] (t1u) to[out=315,in=45] (t1d);
\draw[braid] (t1d) to[out=315,in=90] (out3);
\path[braid,name path=bin1] (in1) to[out=270,in=135] (t1d);
\draw[braid,name path=t1ut3] (t1u) to[out=225,in=135] (t3);
\fill[white, name intersections={of=bin1 and t1ut3}] (intersection-1) circle(0.1);
\draw[braid] (in1) to[out=270,in=135] (t1d);
\draw[braid] (t1d)  to[out=225,in=45] (t3);
\draw[braid] (t3) to[out=225,in=90] (out1);
\draw[braid] (t3) to[out=315,in=90] (out2);
\draw (2,1.35) node {$=$};
\end{tikzpicture} 
\begin{tikzpicture}
\path
(0.75, 0.5) node[arr,name=t1] {$t_1$}
(0.25,1.5) node[arr,name=t3] {$v^3$};
\path 
(1.0,2.5) node[empty,name=in3] {} 
(0.5,2.5) node[empty,name=in2] {} 
(0,2.5) node[empty,name=in1] {} 
(0,0) node[empty,name=out1] {} 
(0.5,0) node[empty,name=out2] {} 
(1,0) node[empty,name=out3] {}; 
\draw[braid] (in3) to[out=270,in=45] (t1);
\path[braid, name path=bin1] (in1) to[out=270,in=45](t3);
\draw[braid, name path=bin2] (in2) to[out=270,in=135](t3);
\fill[white, name intersections={of=bin1 and bin2}] (intersection-1) circle(0.1);
\draw[braid] (in1) to[out=270,in=45](t3);
\draw[braid] (t3) to[out=225,in=90](out1);
\draw[braid] (t3) to[out=315,in=135] (t1);
\draw[braid] (t1) to[out=225,in=90] (out2);
\draw[braid] (t1) to[out=315,in=90] (out3);
\end{tikzpicture}
$$
which is obtained analogously to \eqref{eq:t1t3} from \eqref{eq:t_3_comodule}. 
Cancelling the epimorphisms $e111$ and $1f_1$ we obtain  (a). 
\end{proof}

\begin{corollary} \label{cor:v-tilde}
Let $(A,t_1,t_2,t_3,t_4,e)$ be a regular multiplier bimonoid in $\cc$ such that 
$(A,t_1,t_2,e)$ is a multiplier Hopf monoid with non-degenerate
multiplication and dense counit. 
Then the antipode $s\colon A\op\nto A$ of Theorem~\ref{thm:s_mbm_morphism} and
the identity $i\colon A\nto A$ are dense multiplicative \mM-morphisms. 
It follows by \eqref{eq:s} that $s_1.t_1=e1$,
so that by Lemma \ref{lem:conv_inv}, $(s_*v)^1.v^1=1$ for any comodule
$(v^1,v^3)$. On the other hand, we have the following series of equalities.   
\begin{gather*}
\begin{tikzpicture} 
\path (1.75,2.1) node [arr,name=s1] {$s_1$}
(1.25,1.5) node [arr,name=t3s] {${t_3^s}$}
(1.25,1) node [empty,name=mu] {}
(1,.5) node [empty,name=md] {};
\draw[braid] (.5,2.5) to[out=270,in=180] (md) to[out=0,in=270] (mu);
\draw[braid] (1,2.5) to[out=270,in=135] (t3s);
\draw[braid] (1.5,2.5) to[out=270,in=135] (s1);
\draw[braid] (2,2.5) to[out=270,in=45] (s1);
\draw[braid] (s1) to[out=270,in=45] (t3s);
\draw[braid] (t3s) to[out=225,in=180] (mu) to[out=0,in=315] (t3s);
\draw[braid] (md) to[out=270,in=90] (1,0);
\draw (2.5,1.2) node {$=$};
\end{tikzpicture}
\begin{tikzpicture}
\path (1.75,1.5) node [arr,name=s1] {$s_1$}
(1.25,2) node [arr,name=t3] {$t_3$}
(1.5,1) node [empty,name=mu] {}
(1,.5) node [empty,name=md] {};
\draw[braid] (.5,2.5) to[out=270,in=180] (md) to[out=0,in=270] (mu);
\draw[braid] (1,2.5) to[out=270,in=135] (t3);
\draw[braid] (1.5,2.5) to[out=270,in=45] (t3);
\draw[braid] (2,2.5) to[out=270,in=45] (s1);
\draw[braid] (t3) to[out=315,in=135] (s1);
\draw[braid] (t3) to[out=225,in=180] (mu) to[out=0,in=270] (s1);
\draw[braid] (md) to[out=270,in=90] (1,0);
\draw (2.5,1.4) node {$\stackrel{\mathrm{(ass)}}=$};
\end{tikzpicture}
\begin{tikzpicture}
\path (1.75,1) node [arr,name=s1] {$s_1$}
(1.25,2) node [arr,name=t3] {$t_3$}
(.75,1.5) node [empty,name=mu] {}
(1.5,.5) node [empty,name=md] {};
\draw[braid] (.5,2.5) to[out=270,in=180] (mu) to[out=0,in=225] (t3);
\draw[braid] (1,2.5) to[out=270,in=135] (t3);
\draw[braid] (1.5,2.5) to[out=270,in=45] (t3);
\draw[braid] (2,2.5) to[out=270,in=45] (s1);
\draw[braid] (t3) to[out=315,in=135] (s1);
\draw[braid] (mu) to[out=270,in=180] (md) to[out=0,in=270] (s1);
\draw[braid] (md) to[out=270,in=90] (1.5,0);
\draw (2.75,1.4) node {$\stackrel{\eqref{eq:t_2-3_compatibility}} =$};
\end{tikzpicture}
\begin{tikzpicture}
\path (1.6,1) node [arr,name=s1] {$s_1$}
(.75,2.1) node [arr,name=t2] {$t_2$}
(1.25,1.4) node [empty,name=mu] {}
(1,.5) node [empty,name=md] {};
\draw[braid] (.5,2.5) to[out=270,in=135] (t2);
\draw[braid] (1,2.5) to[out=270,in=45] (t2);
\path[braid,name path=u] (1.5,2.5) to[out=270,in=180] (mu);
\draw[braid] (2,2.5) to[out=270,in=45] (s1);
\draw[braid,name path=d] (mu) to[out=0,in=315] (t2);
\draw[braid] (mu) to[out=270,in=135] (s1);
\draw[braid] (t2) to[out=225,in=180] (md) to[out=0,in=270] (s1);
\draw[braid] (md) to[out=270,in=90] (1,0);
\fill[white, name intersections={of=u and d}] (intersection-1) circle(0.1);
\draw[braid] (1.5,2.5) to[out=270,in=180] (mu);
\draw (2.5,1.4) node {$\stackrel{\eqref{eq:multiplicative}}=$};
\end{tikzpicture}
\\
\begin{tikzpicture}
\path (1.75,1.6) node [arr,name=s1u] {$s_1$}
(1.25,1) node [arr,name=s1d] {$s_1$}
(.75,2) node [arr,name=t2] {$t_2$}
(1,.5) node [empty,name=m] {};
\draw[braid] (.5,2.5) to[out=270,in=135] (t2);
\draw[braid] (1,2.5) to[out=270,in=45] (t2);
\draw[braid] (1.5,2.5) to[out=270,in=135] (s1u);
\draw[braid] (2,2.5) to[out=270,in=45] (s1u);
\draw[braid] (t2) to[out=315,in=135] (s1d);
\draw[braid] (s1u) to[out=270,in=45] (s1d);
\draw[braid] (t2) to[out=225,in=180] (m) to[out=0,in=270] (s1d);
\draw[braid] (m) to[out=270,in=90] (1,0);
\draw (2.5,1.4) node {$\stackrel{\eqref{eq:mM}}=$};
\end{tikzpicture}
\begin{tikzpicture}
\path (1.75,1.6) node [arr,name=s1] {$s_1$}
(.75,1) node [arr,name=s2] {$s_2$}
(.75,2) node [arr,name=t2] {$t_2$}
(1.25,.5) node [empty,name=m] {};
\draw[braid] (.5,2.5) to[out=270,in=135] (t2);
\draw[braid] (1,2.5) to[out=270,in=45] (t2);
\draw[braid] (1.5,2.5) to[out=270,in=135] (s1);
\draw[braid] (2,2.5) to[out=270,in=45] (s1);
\draw[braid] (t2) to[out=225,in=135] (s2);
\draw[braid] (t2) to[out=315,in=45] (s2);
\draw[braid] (s2) to[out=270,in=180] (m) to[out=0,in=270] (s1);
\draw[braid] (m) to[out=270,in=90] (1.25,0);
\draw (2.65,1.4) node {$\stackrel{\eqref{eq:s}}=$};
\end{tikzpicture}
\begin{tikzpicture}
\path [unit] (1,1.5) node [unit,name=e] {}
(1.75,2) node [arr,name=s1] {$s_1$}
(1,1) node [empty,name=m] {};
\draw[braid] (.5,2.5) to[out=270,,in=180] (m) to[out=0,in=270] (s1);
\draw[braid] (1,2.5) to[out=270,in=90] (e);
\draw[braid] (1.5,2.5) to[out=270,in=135] (s1);
\draw[braid] (2,2.5) to[out=270,in=45] (s1);
\draw[braid] (m) to[out=270,in=90] (1,0);
\end{tikzpicture} .
\end{gather*}
Using the non-degeneracy of the multiplication and the density of $s$, we
conclude that $m.t_3^s=e1$, so that by Lemma \ref{lem:conv_inv}, also
$v^1.(s_*v)^1=1$. That is to say, $v^1$ is an isomorphism in \cc, for any
comodule $(v^1,v^3)$. 
\end{corollary}

\subsection{Duals in the base category}

Consider a braided monoidal category $\cc$ and an object $V$ which possesses a
left  dual $\overline V$. We denote the unit and the counit of the duality by 
$\eta\colon I\to V\overline V$ and $\epsilon\colon \overline V V\to I$; 
since $\cc$ is braided, $\overline V$ is also the left dual of $V$.

The duality induces a comonoid structure on the object $Q:= \overline{V}V$
with  comultiplication $\gamma\colon Q\to Q^2$ and counit $\zeta\colon Q\to I$
given by  
$$\xymatrix{
\overline{V}V \ar[r]^-{1\eta1}  & \overline{V}V\overline{V}V && 
\overline{V}V  \ar[r]^{\epsilon} & I. }$$
Using the duality, to give a morphism $v^1\colon VA\to VA$ is
equivalently to give a  morphism  $q_1\colon QA=\overline{V}VA\to
A$; explicitly, $q_1$ is given by the composite on the left below 
\begin{equation} \label{eq:v1q1}
\xymatrix{
q_1: \overline{V}VA \ar[r]^{1v^1} & \overline{V}VA  \ar[r]^{\epsilon1} & A } \quad
\xymatrix{
v^1 : VA \ar[r]^-{\eta11} & V\overline{V}VA \ar[r]^-{1q_1} & VA}
\end{equation}
while one recovers $v^1$ from $q_1$ as the composite on the right.
Similarly to give a  morphism $v^3\colon VA\to VA$ is equivalently
to give a morphism $q_2\colon AQ=A\overline{V}V\to A$, namely 
$$\xymatrix{
A\overline{V}V \ar[r]^{c1} & \overline{V}AV \ar[r]^{1c} & 
\overline{V}VA \ar[r]^{1v^3} & \overline{V}VA \ar[r]^{\epsilon 1} & A. }$$

\begin{proposition}\label{prop:comodules-via-Q}
Let $(A,t_1,t_2,t_3,t_4,e)$ be a regular multiplier bimonoid having the
properties in Section \ref{sect:mbm}. 
Morphisms $v^1,v^3 \colon VA\to VA$ satisfy the
compatibility condition \eqref{eq:comodule_compatibility} if and only if
$q_1$ and $q_2$ satisfy the first condition of \eqref{eq:mM} and so are the
components of an $\mM$-morphism $q\colon Q\nto A$. In this case $v^1$
and $v^3$  satisfy \eqref{eq:t_1_comodule}, and so make $V$ into a
comodule, if and only if $q$ satisfies $e\bullet q=\zeta^\#$ and  $qq\circ
\gamma=d\bullet q$.
\end{proposition}

\begin{proof}
Condition \eqref{eq:comodule_compatibility} takes the form 
$$
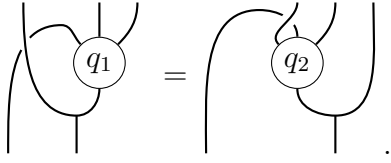
\begin{tikzpicture}
\path (1.5,1.2) node[arr,name=q1] {$\ q_1\ $};
\path[braid,name path=i>d] (.5,2) to[out=270,in=180] (1.2,.5) to[out=0,in=270] (q1);
\draw[braid] (1.5,2) to[out=270,in=90] (q1);
\draw[braid] (2,2) to[out=270,in=45] (q1);
\draw[braid,name path=q1>o] (q1) to[out=135,in=0] (1,1.7) to[out=180,in=90] (.3,0);
\draw[braid] (1.2,.5) to[out=270,in=90] (1.2,0);
\fill[white, name intersections={of=i>d and q1>o}] (intersection-1) circle(0.1);
\draw[braid] (.5,2) to[out=270,in=180] (1.2,.5) to[out=0,in=270] (q1);
\draw (2.5,1) node {$=$};
\end{tikzpicture}
\begin{tikzpicture}
\path (1.5,1.2) node[arr,name=q2] {$\ q_2\ $};
\path[braid,name path=i>l] (1.5,2) to[out=270,in=135] (q2);
\draw[braid] (2,2) to[out=270,in=45] (q2);
\draw[braid,name path=q2>o] (q2) to[out=90,in=0] (1,1.9) to[out=180,in=90] (.3,0);
\draw[braid] (q2) to[out=270,in=180] (2,.5) to [out=0,in=270] (2.5,2);
\draw[braid] (2,.5) to[out=270,in=90] (2,0);
\fill[white, name intersections={of=i>l and q2>o}] (intersection-1) circle(0.1);
\draw[braid]  (1.5,2) to[out=270,in=135] (q2);
\end{tikzpicture} .
$$
Applying the functor $\overline V(-)$ to it, and composing the result on the
respective sides by  
$$
\xymatrix{
\overline VVA \ar[r]^-{\epsilon 1} &
A
}\quad \textrm{and}\quad
\xymatrix{
A\overline VV A\ar[r]^-{c11} &
\overline VAVA,} 
$$
we obtain the equivalent form in \eqref{eq:mM}.

The counit condition, appearing as the second half of \eqref{eq:t_1_comodule},
says that the diagram on the left 
\begin{equation} \label{eq:q-e}
\xymatrix{
QA \ar[r]^{1e} \ar[d]_{q_1} & Q \ar[d]^{\zeta} \\ A \ar[r]_e & I } \quad
\xymatrix{ 
AQ \ar[r]^{e1} \ar[d]_{q_2} & Q \ar[d]^{\zeta} \\ A \ar[r]_{e} & I }
\end{equation}
commutes. This says, in the language of $\mM$-morphisms, that the first
components of $e\bullet q$ and $\zeta^\#:A\nto I$ are equal. Since the trivial
semigroup (monoid, in fact) $I$ is non-degenerate, this is equivalent to the
equality of their second components --- which appears on the right --- and
also to $e\bullet q=\zeta^\#$. 

We have an $\mM$-morphism $q\colon Q\nto A$ and a dense multiplicative
\mM-morphism $d\colon A\nto A^2$, and so can form the composite $d\bullet
q\colon Q\nto A^2$ which we  shall call  $f$; it is
defined by the commutativity of  
$$
\xymatrix{
QA^3 \ar[r]^-{1d_1} \ar[d]_{q_1 11} & 
QA^2 \ar[d]^-{f_1} \\ 
A^3 \ar[r]_-{d_1} & 
A^2. }
$$
On the other hand we have the \cc-morphism $\gamma\colon Q\to Q^2$ and the
\mM-morphism $qq\colon Q^2\nto A^2$, and they compose to give an \mM-morphism
$qq\circ \gamma \colon Q\nto A^2$. This
has first component  
$$
\xymatrix{
QA^2 \ar[r]^-{\gamma11} & 
Q^2A^2 \ar[r]^-{1c1} & 
(QA)^2 \ar[r]^-{q_1q_1} & 
A^2. }
$$
The first half of \eqref{eq:t_1_comodule}  asserts the equality of two 
morphisms $VA^2\to VA^2$. Under the duality this can be expressed as
the equality of two morphisms $\overline{V}VA^2\to A^2$;
specifically, the second and third morphisms appearing below.   
$$
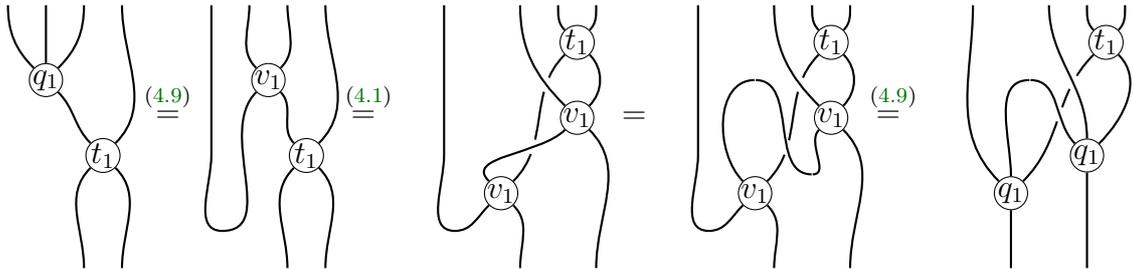
\begin{tikzpicture} 
  \path
(2,0) node[empty] {}
(0,3.5) node[empty, name=in1] {}
(0.5,3.5) node[empty, name=in2] {}
(1.0,3.5) node[empty, name=in3] {}
(1.5,3.5) node[empty, name=in4] {}
(1,0) node[empty, name=out1] {}
(1.5,0) node[empty, name=out2] {}
(0.5,2.5) node[arr, name=q1] {$q_1$}
(1.25,1.5) node[arr, name=t1] {$t_1$};
\draw[braid, name path=t1v1d] (q1) to[out=315, in=135] (t1);
\draw[braid] (in3) to[out=270, in=45] (q1);
\draw[braid] (in2) to[out=270, in=90] (q1);
\draw[braid] (in4) to[out=270, in=45] (t1);
\draw[braid] (t1) to[out=315, in=90] (out2);
\draw[braid] (t1) to[out=225, in=90] (out1);
\draw[braid] (in1) to[out=270, in=135] (q1);
\draw (2.1,2.2) node[empty] {$\stackrel{\eqref{eq:v1q1}}=$};
\end{tikzpicture} 
\begin{tikzpicture} 
  \path
(0,3.5) node[empty, name=in1] {}
(0.5,3.5) node[empty, name=in2] {}
(1.0,3.5) node[empty, name=in3] {}
(1.5,3.5) node[empty, name=in4] {}
(1,0) node[empty, name=out1] {}
(1.5,0) node[empty, name=out2] {}
(0,1.5) node[empty, name=u] {}
(0.15,0.5) node[empty, name=x] {}
(0.75,2.5) node[arr, name=v1] {$v_1$}
(1.25,1.5) node[arr, name=t1] {$t_1$};
\draw[braid, name path=t1v1d] (v1) to[out=315, in=135] (t1);
\draw[braid] (in3) to[out=270, in=45] (v1);
\draw[braid] (in2) to[out=270, in=135] (v1);
\draw[braid] (in4) to[out=270, in=45] (t1);
\draw[braid] (t1) to[out=315, in=90] (out2);
\draw[braid] (t1) to[out=225, in=90] (out1);
\draw[braid] (in1) to[out=270,in=90] (u) to[out=270,in=180] (x) to[out=0,in=225] (v1);
\draw (2.0,2.2) node[empty] {$\stackrel{~\eqref{eq:t_1_comodule}}=$};
\end{tikzpicture} \quad 
\begin{tikzpicture} 
  \path
(0,3.5) node[empty, name=in1] {}
(1,3.5) node[empty, name=in2] {}
(1.5,3.5) node[empty, name=in3] {}
(2.0,3.5) node[empty, name=in4] {}
(1,0) node[empty, name=out1] {}
(2,0) node[empty, name=out2] {}
(0,1.5) node[empty, name=u] {}
(0.15,0.5) node[empty, name=x] {}
(0.75,1) node[arr, name=v1d] {$v_1$}
(1.75,2) node[arr,name=v1u] {$v_1$}
(1.75,3) node[arr, name=t1] {$t_1$};
\draw[braid, name path=t1v1d] (t1) to[out=225, in=45] (v1d);
\path[braid, name path=bin2] (in2) to[out=270, in=135] (v1u);
\draw[braid] (in3) to[out=270, in=135] (t1);
\fill[white, name intersections={of=bin2 and t1v1d}] (intersection-1) circle(0.1);
\draw[braid] (in2) to[out=270, in=135] (v1u);
\draw[braid] (in4) to[out=270, in=45] (t1);
\draw[braid] (t1) to[out=315, in=45] (v1u);
\draw[braid] (v1u) to[out=315, in=90] (out2);
\path[braid, name path=v1uv1d] (v1u) to[out=225, in=135] (v1d);
\fill[white, name intersections={of=v1uv1d and t1v1d}] (intersection-1) circle(0.1);
\draw[braid] (v1u) to[out=225, in=135] (v1d);
\draw[braid] (in1) to[out=270, in=90] (u) to[out=270, in=180] (x) to[out=0,in=225] (v1d);
\draw[braid] (v1d) to[out=315, in=90] (out1);
\draw (2.5,2) node[empty] {$=$};
\end{tikzpicture} \quad 
\begin{tikzpicture} 
  \path
(0,3.5) node[empty, name=in1] {}
(1,3.5) node[empty, name=in2] {}
(1.5,3.5) node[empty, name=in3] {}
(2,3.5) node[empty, name=in4] {}
(1,0) node[empty, name=out1] {}
(2,0) node[empty, name=out2] {}
(0,1.5) node[empty, name=u] {}
(0.15,0.5) node[empty, name=x] {}
(1.5,1.25) node[empty, name=y] {}
(0.75,2.5) node[empty, name=z] {}
(0.75,1) node[arr, name=v1d] {$v_1$}
(1.75,2) node[arr,name=v1u] {$v_1$}
(1.75,3) node[arr, name=t1] {$t_1$};
\draw[braid, name path=t1v1d] (t1) to[out=225, in=45] (v1d);
\path[braid, name path=bin2] (in2) to[out=270, in=135] (v1u);
\draw[braid] (in3) to[out=270, in=135] (t1);
\fill[white, name intersections={of=bin2 and t1v1d}] (intersection-1) circle(0.1);
\draw[braid] (in2) to[out=270, in=135] (v1u);
\draw[braid] (in4) to[out=270, in=45] (t1);
\draw[braid] (t1) to[out=315, in=45] (v1u);
\draw[braid] (v1u) to[out=315, in=90] (out2);
\draw[braid,name path=v1uy] (v1u) to[out=225, in=0] (y);
\path[braid, name path=zy] (z) to[out=0, in=180] (y);
\fill[white, name intersections={of=zy and t1v1d}] (intersection-1) circle(0.1);
\draw[braid] (z) to[out=0, in=180] (y);
\draw[braid] (z) to[out=180, in=135] (v1d);
\draw[braid] (in1) to[out=270,in=90] (u) to[out=270,in=180] (x) to[out=0,in=225] (v1d);
\draw[braid] (v1d) to[out=315, in=90] (out1);
\draw (2.5,2.2) node[empty] {$\stackrel{~\eqref{eq:v1q1}}=$};
\end{tikzpicture} \quad 
\begin{tikzpicture} 
  \path
(0,3.5) node[empty, name=in1] {}
(1,3.5) node[empty, name=in2] {}
(1.5,3.5) node[empty, name=in3] {}
(2,3.5) node[empty, name=in4] {}
(.5,0) node[empty, name=out1] {}
(1.5,0) node[empty, name=out2] {}
(0,1.5) node[empty, name=u] {}
(0.15,0.5) node[empty, name=x] {}
(1.5,1.25) node[empty, name=y] {}
(0.75,2.5) node[empty, name=z] {}
(0.5,1) node[arr, name=q1l] {$q_1$}
(1.5,1.5) node[arr,name=q1r] {$q_1$}
(1.75,3) node[arr, name=t1] {$t_1$};
\draw[braid, name path=t1q1l] (t1) to[out=225, in=45] (q1l);
\path[braid, name path=bin2] (in2) to[out=270, in=90] (q1r);
\draw[braid] (in3) to[out=270, in=135] (t1);
\fill[white, name intersections={of=bin2 and t1q1l}] (intersection-1) circle(0.1);
\draw[braid] (in2) to[out=270, in=90] (q1r);
\draw[braid] (in4) to[out=270, in=45] (t1);
\draw[braid] (t1) to[out=315, in=45] (q1r);
\draw[braid] (q1r) to[out=270, in=90] (out2);
\path[braid, name path=zq1r] (z) to[out=0, in=135] (q1r);
\fill[white, name intersections={of=zq1r and t1q1l}] (intersection-1) circle(0.1);
\draw[braid] (z) to[out=0, in=135] (q1r);
\draw[braid] (z) to[out=180, in=90] (q1l);
\draw[braid] (in1) to[out=270, in=135] (q1l);
\draw[braid] (q1l) to[out=270, in=90] (out1);
\end{tikzpicture}
$$
Since the remaining equalities do hold, the first half of
\eqref{eq:t_1_comodule} is equivalent to the equality of the first and last
expressions; writing in terms of $Q$ rather than $\overline{V}V$ this says
that the diagram on the left 
\begin{equation} \label{eq:qtgamma}
\xymatrix{
QA^2 \ar[r]^{\gamma t_1} \ar[d]_{q_1 1} & 
Q^2A^2 \ar[r]^{1c1} & 
(QA)^2 \ar[d]^{q_1q_1} \\
A^2 \ar[rr]_{t_1} && 
A^2 }\quad
\xymatrix{
A^2Q \ar[r]^{t_2\gamma} \ar[d]_{1q_2} & 
A^2Q^2 \ar[r]^{1c1} & 
(AQ)^2 \ar[d]^{q_2q_2} \\
A^2 \ar[rr]_{t_2} && 
A^2 }
\end{equation}
commutes; we shall also need its (equivalent by Lemma
\ref{lem:comodule_nd})  counterpart for second components, which is the
diagram on the right. The right $A$-linearity of $v^1$, in the sense of
\eqref{eq:v^1-3_module_maps}, is equivalent to the right $A$-linearity of
$q_1$. By this and the non-degeneracy of $A$, commutativity of the first
diagram in \eqref{eq:qtgamma} amounts exactly to fact that $f_1$ is equal to
$q_1q_1.1c1.\gamma 11$. Thus the fusion equation of
\eqref{eq:t_1_comodule} 
amounts to the identity $ qq\circ \gamma = d\bullet q.$  
\end{proof}

If $(V,v^1,v^3)$ and $(W,w^1,w^3)$ are comodules whose underlying
objects $V$ and $W$  have left duals $\overline V$ and 
$\overline W$  corresponding to \mM-morphisms $q\colon \overline VV\nto
A$ and $ q\colon \overline W W \nto A$, then a morphism $f\colon
V\to W$ is a morphism of comodules exactly when the diagram  
\begin{equation}\label{eq:Q-morphism} 
\xymatrix{
 \overline W  VA \ar[r]^{1f1} \ar[d]_{\overline f11} & 
 \overline W W  A \ar[d]^{ q_1} \\
\overline VVA \ar[r]_{q_1} & 
A }
\end{equation}
commutes; here $\overline f\colon  \overline W \to \overline V$
is the transpose of $f$.  

For the same two comodules, we may form their monoidal product
$V W$, as in \cite{BohmLack:braided_mba}; in terms of \mM-morphisms, 
the resulting comodule structure has first component   
\begin{equation}\label{eq:Q-tensor}
 \xymatrix@C=35pt{
\overline W\,\overline V VWA \ar[r]^-{c^{-1}111} & 
\overline V\,\overline W VWA \ar[r]^-{1c^{-1}11} & 
\overline VV\overline W WA \ar[r]^-{11q_1} & 
\overline VVA \ar[r]^-{q_1} & 
A.} 
\end{equation}

\begin{theorem} \label{thm:dual_comodule}
Consider a symmetric monoidal category \cc which satisfies our standing
assumptions of Section~\ref{sect:assumptions}, and let $(A,t_1,t_2,t_3,t_4,e)$
be a regular multiplier bimonoid in \cc for which $(A,t_1,t_2,e)$ is a
multiplier Hopf monoid whose multiplication is non-degenerate and whose counit
is dense. Let $V$ be an object which possesses a left dual $\overline{V}$ in
\cc. Then any comodule with underlying object $V$ possesses a left dual in the
monoidal category of comodules, with underlying object $\overline{V}$. 
\end{theorem}

\begin{proof}
By Remark~\ref{rem:who_is_reg_epi}, the comultiplication will also be dense,
and so we may apply Proposition~\ref{prop:comodules-via-Q}; thus in order to
make $\overline{V}$ into a comodule, it will suffice to define a suitable
\mM-morphism $q'\colon Q'\nto A$, where $Q'=V\overline{V}$, which is a
comonoid in \cc with comultiplication $\gamma'$ and counit $\zeta'$ given by
the following composites. 
$$\xymatrix{
V\overline V \ar[r]^-{1\eta1} & 
VV\overline V\,\overline V \ar[r]^{1c1} & 
V\overline V V\overline V &&
V\overline V \ar[r]^{c^{-1}} & 
\overline V V \ar[r]^{\epsilon} & 
I }$$

The inverse braiding $c^{-1}\colon V\overline{V}\to \overline{V}V$ defines a
morphism $Q'\to Q$; in fact it is a comonoid morphism from the comonoid $Q'$
to the opposite comonoid $Q\cop$ of $Q$. The desired \mM-morphism $q'$ is now
given by $s\bullet q\op \circ c^{-1}$; we just need to check that this 
satisfies the conditions in Proposition~\ref{prop:comodules-via-Q}.
Compatibility with the comultiplication follows by the calculation  

\begin{align}
d\bullet s \bullet q\op \circ c^{-1}  & =  
ss\bullet c^\#\bullet d\op\bullet q\op \circ c^{-1} 
\tag{Theorem~\ref{thm:s_mbm_morphism}} \\
&=     ss\bullet c^\#\bullet (d\bullet q)\op \circ c^{-1} 
\tag{Proposition~\ref{prop:op}} \\
&=     ss\bullet c^\#\bullet (qq\circ \gamma)\op \circ c^{-1} 
\tag{Proposition~\ref{prop:comodules-via-Q}} \\
&=    ss\bullet c^\#     \bullet       (qq)\op\circ \gamma \circ c^{-1} 
 \tag{eq. \eqref{eq:op-mon}} \\
&=  ss\bullet q\op q\op \circ c \circ \gamma \circ c^{-1} 
\tag{eq. \eqref{eq:op-mon}}  \\
&=     ss\bullet q\op q\op\circ (c.\gamma.c^{-1}) 
\tag{Section \ref{sect:mM}} \\
&=     ss\bullet q\op q\op \circ (c^{-1}c^{-1}.\gamma') 
\notag\\
&=     ss\bullet q\op q\op \circ c^{-1}c^{-1} \circ \gamma' 
\tag{Section \ref{sect:mM}} \\
&=  (s\bullet q\op\circ c^{-1}) (s\bullet q\op\circ c^{-1}) \circ \gamma'
 \tag{Section \ref{sect:M-monoidal}}
\end{align}
while 
compatibility with  the counit is similar but easier: 
\begin{align}
e\bullet s \bullet q\op \circ c^{-1} &=
e\op \bullet q\op \circ c^{-1} \tag{Theorem~\ref{thm:s_mbm_morphism}} \\
&= (e\bullet q)\op \circ c^{-1} \tag{Proposition \ref{prop:op}} \\
&= \zeta^{\#\mathsf{op}}\circ c^{-1} 
\tag{Proposition~\ref{prop:comodules-via-Q}} \\
&= \zeta^\#\circ c^{-1} \notag\\
&= (\zeta.c^{-1})^\# \tag{Section \ref{sect:mM}} \\
&= \zeta^{\prime \#}. \notag
\end{align}
It remains to show that the unit and counit of the duality are
comodule morphisms. The first  component of the structure morphism
$(V\overline V)^2 \nto A$  of the monoidal product comodule $V\overline
V$ is given by \eqref{eq:Q-tensor}.  Now
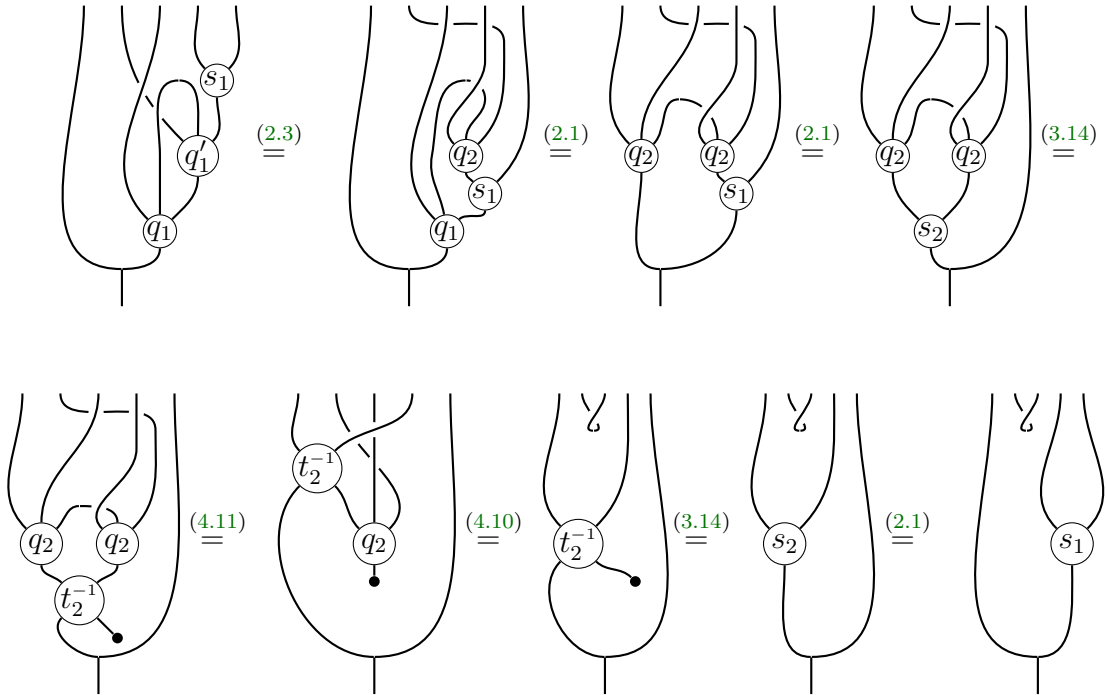
$$
\begin{tikzpicture} 
\path
(0,4) node[empty,name=in1] {}
(0.5,4) node[empty,name=in2] {}
(1.0,4) node[empty,name=in3] {}
(1.5,4) node[empty,name=in4] {}
(2,4) node[empty,name=in5] {}
(0.5,0) node[empty, name=out] {}
(0.5,0.5) node[empty, name=m] {}
(1,2) node[empty, name=y] {}
(1.25,3) node[empty, name=x] {};
\path
(1,1) node[arr, name=q1] {$q_1$}
(1.5,2) node[arr, name=qq1] {$q'_1$}
(1.75,3) node[arr, name=s1] {$s_1$};
\draw[braid] (in1) to[out=270,in=180] (m);
\draw[braid, name path=bin2] (in2) to[out=270,in=135] (qq1);
\path[braid, name path=bin3] (in3) to[out=270,in=135] (q1);
\fill[white,name intersections={of=bin3 and bin2}] (intersection-1) circle(0.1);
\draw[braid] (in3) to[out=270,in=135] (q1);
\draw[braid] (in4) to[out=270,in=135] (s1);
\draw[braid] (in5) to[out=270,in=45] (s1);
\draw[braid, name path=xqq1] (x) to[out=0, in=90] (qq1);
\path[braid, name path=xq1] (x) to[out=180, in=90] (y) to (q1);
\fill[white,name intersections={of=xq1 and bin2}] (intersection-1) circle(0.1);
\draw[braid] (x) to[out=180, in=90] (y) to (q1);
\draw[braid] (s1) to[out=270,in=45] (qq1);
\draw[braid] (qq1) to[out=270,in=45] (q1);
\draw[braid] (q1) to[out=270,in=0] (m);
\draw[braid] (m) to (out); 
\draw (2.5,2.2) node[empty]  {$\stackrel{~\eqref{eq:bullet-composition}}=$};
\end{tikzpicture}
\begin{tikzpicture} 
\path
(0,4) node[empty,name=in1] {}
(0.5,4) node[empty,name=in2] {}
(1.0,4) node[empty,name=in3] {}
(1.5,4) node[empty,name=in4] {}
(2,4) node[empty,name=in5] {}
 (0.5,0) node[empty, name=out] {}
(0.5,0.5) node[empty, name=m] {}
(0.8,2) node[empty, name=y] {}
(1.75,3.5) node[empty, name=z] {}
(1.5, 2.75) node[empty, name=v] {}
(1.5,3.25) node[empty, name=w] {}
(1.25,3.0) node[empty, name=x] {};
\path
(1,1) node[arr, name=q1] {$q_1$}
(1.25,2) node[arr, name=q2] {$q_2$}
(1.5,1.5) node[arr, name=s1] {$s_1$};
\draw[braid] (in1) to[out=270,in=180] (m);
\draw[braid, name path=bin2] (in2) to[out=270, in=90] (z) to[out=270, in=45] (q2);
\path[braid, name path=bin3] (in3) to[out=270,in=135] (q1);
\fill[white,name intersections={of=bin3 and bin2}] (intersection-1) circle(0.1);
\draw[braid] (in3) to[out=270,in=135] (q1);
\path[braid, name path=bin4] (in4) to[out=270, in=90] (w) to[out=270,
in=135] (q2);
\draw[braid] (in5) to[out=270,in=45] (s1);
\draw[braid, name path=xq2] (x) to[out=0,in=90] (v) to[out=270,in=90] (q2);
\draw[braid, name path=xq1] (x) to[out=180, in=80] (y) to[out=260,in=100] (q1);
\fill[white,name intersections={of=bin4 and bin2}] (intersection-1) circle(0.1);
\fill[white,name intersections={of=xq2 and bin4}] (intersection-1) circle(0.1);
\draw[braid] (in4) to[out=270, in=90] (w) to[out=270,in=135] (q2);
\draw[braid] (q2) to[out=270,in=135] (s1);
\draw[braid] (s1) to[out=270,in=45] (q1);
\draw[braid] (q1) to[out=270,in=0] (m);
\draw[braid] (m) to (out); 
\draw (2.5,2.2) node[empty] {$\stackrel{~\eqref{eq:mM}}=$};
\end{tikzpicture}
\begin{tikzpicture} 
\path
(0,4) node[empty,name=in1] {}
(0.5,4) node[empty,name=in2] {}
(1.0,4) node[empty,name=in3] {}
(1.5,4) node[empty,name=in4] {}
(2,4) node[empty,name=in5] {}
 (0.5,0) node[empty, name=out] {}
(0.5,0.5) node[empty, name=m] {}
(1,3) node[empty, name=y] {}
(1.75,3.5) node[empty, name=z] {}
(1.5, 2.75) node[empty, name=v] {}
(1.5,3.25) node[empty, name=w] {}
(0.75,2.75) node[empty, name=x] {};
\path
(0.25,2) node[arr, name=q2l] {$q_2$}
(1.25,2) node[arr, name=q2] {$q_2$}
(1.5,1.5) node[arr, name=s1] {$s_1$};
\draw[braid] (in1) to[out=270,in=135] (q2l);
\draw[braid, name path=bin2] (in2) to[out=270, in=90] (z) to[out=270, in=45] (q2);
\path[braid, name path=bin3] (in3) to[out=270,in=90] (q2l);
\fill[white,name intersections={of=bin3 and bin2}] (intersection-1) circle(0.1);
\draw[braid] (in3) to[out=270,in=90] (q2l);
\path[braid, name path=bin4] (in4) to[out=270, in=90] (w) to[out=270,in=135] (q2);
\draw[braid] (in5) to[out=270,in=45] (s1);
\path[braid, name path=xq2] (x) to[out=0,in=90] (q2);
\draw[braid] (x) to[out=0,in=90] (q2);
\path[braid, name path=xq2l] (x) to[out=180, in=45] (q2l);
\fill[white,name intersections={of=bin4 and bin2}] (intersection-1) circle(0.1);
\fill[white,name intersections={of=xq2 and bin4}] (intersection-1) circle(0.1);
\draw[braid] (x) to[out=180, in=45] (q2l);
\draw[braid] (in4) to[out=270, in=90] (w) to[out=270,in=135] (q2);
\draw[braid] (q2) to[out=270,in=135] (s1);
\draw[braid] (q2l) to[out=270,in=180] (m);
\draw[braid] (s1) to[out=270,in=0] (m);
\draw[braid] (m) to (out); 
\draw (2.5,2.2) node[empty]  {$\stackrel{~\eqref{eq:mM}}=$};
\end{tikzpicture}
\begin{tikzpicture} 
\path
(0,4) node[empty,name=in1] {}
(0.5,4) node[empty,name=in2] {}
(1.0,4) node[empty,name=in3] {}
(1.5,4) node[empty,name=in4] {}
(2,4) node[empty,name=in5] {}
 (1.0,0) node[empty, name=out] {}
(1.0,0.5) node[empty, name=m] {}
(1,3) node[empty, name=y] {}
(1.75,3.5) node[empty, name=z] {}
(1.5, 2.75) node[empty, name=v] {}
(1.5,3.25) node[empty, name=w] {}
(0.75,2.75) node[empty, name=x] {};
\path
(0.25,2) node[arr, name=q2l] {$q_2$}
(1.25,2) node[arr, name=q2] {$q_2$}
(0.75,1.0) node[arr, name=s2] {$s_2$};
\draw[braid] (in1) to[out=270,in=135] (q2l);
\draw[braid, name path=bin2] (in2) to[out=270, in=90] (z) to[out=270, in=45] (q2);
\path[braid, name path=bin3] (in3) to[out=270,in=90] (q2l);
\fill[white,name intersections={of=bin3 and bin2}] (intersection-1) circle(0.1);
\draw[braid] (in3) to[out=270,in=90] (q2l);
\path[braid, name path=bin4] (in4) to[out=270, in=90] (w) to[out=270,in=135] (q2);
\draw[braid] (in5) to[out=270,in=0] (m);
\path[braid, name path=xq2] (x) to[out=0,in=90] (q2);
\draw[braid] (x) to[out=0,in=90] (q2);
\path[braid, name path=xq2l] (x) to[out=180, in=45] (q2l);
\fill[white,name intersections={of=bin4 and bin2}] (intersection-1) circle(0.1);
\fill[white,name intersections={of=xq2 and bin4}] (intersection-1) circle(0.1);
\draw[braid] (x) to[out=180, in=45] (q2l);
\draw[braid] (in4) to[out=270, in=90] (w) to[out=270,in=135] (q2);
\draw[braid] (q2) to[out=270,in=45] (s2);
\draw[braid] (q2l) to[out=270,in=135] (s2);
\draw[braid] (s2) to[out=270,in=180] (m);
\draw[braid] (m) to (out); 
\draw (2.5,2.2) node[empty]  {$\stackrel{~\eqref{eq:s}}=$};
\end{tikzpicture}
$$

$$
\begin{tikzpicture} 
\path
(0,4) node[empty,name=in1] {}
(0.5,4) node[empty,name=in2] {}
(1.0,4) node[empty,name=in3] {}
(1.5,4) node[empty,name=in4] {}
(2,4) node[empty,name=in5] {}
 (1.0,0) node[empty, name=out] {}
(1.0,0.5) node[empty, name=m] {}
(1,3) node[empty, name=y] {}
(1.75,3.5) node[empty, name=z] {}
(1.5, 2.75) node[empty, name=v] {}
(1.5,3.25) node[empty, name=w] {}
(0.75,2.5) node[empty, name=x] {};
\path
(0.25,2) node[arr, name=q2l] {$\,q_2\,$}
(1.25,2) node[arr, name=q2] {$\,q_2\,$}
(0.75,1.25) node[arr, name=t2inv] {${t^{\scriptscriptstyle-1}_2}$}
(1.25,0.75) node[unit, name=e] {};
\draw[braid] (in1) to[out=270,in=135] (q2l);
\draw[braid, name path=bin2] (in2) to[out=270, in=90] (z) to[out=270, in=45] (q2);
\path[braid, name path=bin3] (in3) to[out=270,in=90] (q2l);
\fill[white,name intersections={of=bin3 and bin2}] (intersection-1) circle(0.1);
\draw[braid] (in3) to[out=270,in=90] (q2l);
\path[braid, name path=bin4] (in4) to[out=270, in=90] (w) to[out=270,in=135] (q2);
\draw[braid] (in5) to[out=270,in=0] (m);
\path[braid, name path=xq2] (x) to[out=0,in=90] (q2);
\draw[braid] (x) to[out=0,in=90] (q2);
\path[braid, name path=xq2l] (x) to[out=180, in=45] (q2l);
\fill[white,name intersections={of=bin4 and bin2}] (intersection-1) circle(0.1);
\fill[white,name intersections={of=xq2 and bin4}] (intersection-1) circle(0.1);
\draw[braid] (x) to[out=180, in=45] (q2l);
\draw[braid] (in4) to[out=270, in=90] (w) to[out=270,in=135] (q2);
\draw[braid] (q2) to[out=270,in=45] (t2inv);
\draw[braid] (q2l) to[out=270,in=135] (t2inv);
\draw[braid] (t2inv) to[out=225,in=180] (m);
\draw[braid] (t2inv) to (e);
\draw[braid] (m) to (out); 
\draw (2.5,2.2) node[empty]  {$\stackrel{~\eqref{eq:qtgamma}}=$};
\end{tikzpicture}
\begin{tikzpicture} 
\path
(0,4) node[empty,name=in1] {}
(0.5,4) node[empty,name=in2] {}
(1.0,4) node[empty,name=in3] {}
(1.5,4) node[empty,name=in4] {}
(2,4) node[empty,name=in5] {}
 (1.0,0) node[empty, name=out] {}
(1.0,0.5) node[empty, name=m] {}
(1,3) node[empty, name=y] {}
(1.75,3.5) node[empty, name=z] {}
(1.5, 2.75) node[empty, name=v] {}
(1.5,3.25) node[empty, name=w] {}
(0.75,2.75) node[empty, name=x] {};
\path
(1.0,2) node[arr, name=q2] {$\,q_2\,$}
(0.25,3) node[arr, name=t2inv] {${t^{\scriptscriptstyle-1}_2}$}
(1,1.5) node[unit, name=e] {};
\draw[braid] (in1) to[out=270,in=135] (t2inv);
\draw[braid, name path=bin2] (in2) to[out=270, in=45] (q2);
\path[braid, name path=bin3] (in3) to[out=270,in=90] (q2);
\path[braid, name path=bin4] (in4) to[out=270,in=45] (t2inv);
\fill[white,name intersections={of=bin3 and bin2}] (intersection-1) circle(0.1);
\draw[braid] (in3) to[out=270,in=90] (q2);
\draw[braid] (in5) to[out=270,in=0] (m);
\fill[white,name intersections={of=bin4 and bin2}] (intersection-1) circle(0.1);
\fill[white,name intersections={of=bin4 and bin3}] (intersection-1) circle(0.1);
\draw[braid] (in4) to[out=270,in=45] (t2inv);
\draw[braid] (q2) to[out=135,in=315] (t2inv);
\draw[braid] (t2inv) to[out=225,in=180] (m);
\draw[braid] (q2) to (e);
\draw[braid] (m) to (out); 
\draw (2.5,2.2) node[empty]  {$\stackrel{~\eqref{eq:q-e}}=$};
\end{tikzpicture}
\begin{tikzpicture} 
\path
(0,4) node[empty,name=in1] {}
(0.3,4) node[empty,name=in2] {}
(0.6,4) node[empty,name=in3] {}
(0.9,4) node[empty,name=in4] {}
(1.2,4) node[empty,name=in5] {}
 (0.6,0) node[empty, name=out] {}
(0.6,0.5) node[empty, name=m] {}
(0.45,3.5) node[empty, name=x] {};
\path
(0.25,2) node[arr, name=t2inv] {${t^{\scriptscriptstyle-1}_2}$}
(1,1.5) node[unit, name=e] {};
\draw[braid] (in1) to[out=270,in=135] (t2inv);
\draw[braid, name path=bin2] (in2) to[out=270, in=0] (x);
\path[braid, name path=bin3] (x) to[out=180,in=270] (in3);
\fill[white,name intersections={of=bin3 and bin2}] (intersection-1) circle(0.1);
\draw[braid] (x) to[out=180,in=270] (in3);
\draw[braid] (in4) to[out=270,in=45] (t2inv);
\draw[braid] (in5) to[out=270,in=0] (m);
\draw[braid] (t2inv) to[out=225,in=180] (m);
\draw[braid] (t2inv) to[out=315] (e);
\draw[braid] (m) to (out); 
\draw (1.8,2.2) node[empty]  {$\stackrel{~\eqref{eq:s}}=$};
\end{tikzpicture}
\begin{tikzpicture} 
\path
(0,4) node[empty,name=in1] {}
(0.3,4) node[empty,name=in2] {}
(0.6,4) node[empty,name=in3] {}
(0.9,4) node[empty,name=in4] {}
(1.2,4) node[empty,name=in5] {}
 (0.6,0) node[empty, name=out] {}
(0.6,0.5) node[empty, name=m] {}
(0.45,3.5) node[empty, name=x] {};
\path
(0.25,2) node[arr, name=s2] {$\,s_2\,$};
\draw[braid] (in1) to[out=270,in=135] (s2);
\draw[braid, name path=bin2] (in2) to[out=270, in=0] (x);
\path[braid, name path=bin3] (x) to[out=180,in=270] (in3);
\fill[white,name intersections={of=bin3 and bin2}] (intersection-1) circle(0.1);
\draw[braid] (x) to[out=180,in=270] (in3);
\draw[braid] (in4) to[out=270,in=45] (s2);
\draw[braid] (in5) to[out=270,in=0] (m);
\draw[braid] (s2) to[out=270,in=180] (m);
\draw[braid] (m) to (out); 
\draw (1.8,2.2) node[empty]  {$\stackrel{~\eqref{eq:mM}}=$};
\end{tikzpicture}
\begin{tikzpicture} 
\path
(0,4) node[empty,name=in1] {}
(0.3,4) node[empty,name=in2] {}
(0.6,4) node[empty,name=in3] {}
(0.9,4) node[empty,name=in4] {}
(1.2,4) node[empty,name=in5] {}
 (0.6,0) node[empty, name=out] {}
(0.6,0.5) node[empty, name=m] {}
(0.45,3.5) node[empty, name=x] {};
\path
(1.05,2) node[arr, name=s1] {$\,s_1\,$};
\draw[braid] (in1) to[out=270,in=180] (m);
\draw[braid, name path=bin2] (in2) to[out=270, in=0] (x);
\path[braid, name path=bin3] (x) to[out=180,in=270] (in3);
\fill[white,name intersections={of=bin3 and bin2}] (intersection-1) circle(0.1);
\draw[braid] (x) to[out=180,in=270] (in3);
\draw[braid] (in4) to[out=270,in=135] (s1);
\draw[braid] (in5) to[out=270,in=45] (s1);
\draw[braid] (s1) to[out=270,in=0] (m);
\draw[braid] (m) to (out); 
\end{tikzpicture}
$$
and now we may cancel $s_1$ and use non-degeneracy to deduce that the composite
$$\xymatrix{
V\overline{V}A \ar[r]^{c^{-1}1} & 
\overline{V}VA \ar[r]^-{11\eta1} & 
\overline{V}V^2\overline{V}A \ar[r]^{1c^{-1}11} & 
\overline{V}V^2\overline{V}A \ar[r]^{11q'_1} & 
\overline{V}VA \ar[r]^{q_1} & A }$$
is equal to $\epsilon1.c^{-1}1\colon V\overline{V}A\to A$; but 
$\epsilon.c^{-1}$  is the transpose $\overline\eta$ of $\eta$, and so
$\eta$ satisfies the condition \eqref{eq:Q-morphism} to be a morphism of
comodules.

Similarly the counit $\epsilon$ is a morphism of comodules by essentially the
same calculation, with $V$ and $\overline{V}$ interchanged. 
\end{proof}


\section{Modules over multiplier Hopf monoids}
 
Modules over a regular multiplier bimonoid in a braided monoidal category were
studied in \cite{BohmLack:braided_mba}. In this section we look at what
happens to this theory when the underlying multiplier bimonoid is Hopf. 

\subsection{Modules over non-degenerate semigroups} \label{claim:module}

If $(A,m)$ is a  semigroup, we define a (right) {\em $A$-module} to be an
object $V$ equipped with an associative action $v\colon VA\to V$ which also
lies in \cq. We shall sometimes refer to the property that $v\in\cq$ by saying
that the action is {\em surjective}.

\begin{remark} \label{rmk:action-epi}
This surjectivity condition is supposed to capture the idea of
unitality of an action in the absence of a unit for the multiplication;
that is, in the context of a semigroup. 
If the multiplication $m\colon A^2 \to A$ possesses a unit $u\colon I \to A$,
then the following properties of any associative action $v\colon VA \to V$
turn out to be equivalent: 
\begin{itemize}
\item[{(a)}] $v$ is unital; that is, $v.1u=1$. 
\item[{(b)}] $v$ is a split epimorphism.
\item[{(c)}] $v$ is a regular epimorphism.
\item[{(d)}] $v$ is an epimorphism.
\end{itemize}
Indeed, either (b) or (c) trivially implies (d), and (d) implies (a) by the
commutativity of
$$
\xymatrix{
&&
VA \ar[rrdd]^-v \\
&&
VA^2 \ar[u]^-{v1} \ar[rd]^-{1m} \\
V \ar[uurr]^-{1u} &
VA \ar[l]_-v \ar[ru]^-{11u} \ar@{=}[rr] &&
VA \ar[r]^-v &
V\ .}
$$
The implication (a)$\Rightarrow$(b) is obvious and (a) implies that
$$
\xymatrix@C=40pt{
VA^2 \ar@<2pt>[r]^-{v1} \ar@<-2pt>[r]_-{1m} &
VA \ar[r]^-v \ar@{-->}@/^1.3pc/[l]^-{11u} &
V \ar@{-->}@/^1.3pc/[l]^-{1u}
}
$$
is a split coequalizer; hence (c).
\end{remark}

We say that $(V,v)$ is {\em non-degenerate} if for any objects 
$X,Y$ of $\cc$, the following map is injective.
$$
\cc(X,YV) \to \cc(XA,YV),\qquad
f\mapsto 
\xymatrix{
XA \ar[r]^-{f1} &
YVA \ar[r]^-{1v} &
YV}
$$
Note that, here again as in Section \ref{sect:nondeg}, we made the 
strong assumption that this injectivity condition holds for {\em all} 
objects $X$ and $Y$. As discussed in \cite{BohmGom-TorLack:wmbm}, for 
many purposes it is enough, in fact, that it hold for 
all objects $X$, and all $Y$ in a distinguished class as in Section 
\ref{sect:ex_non-deg}. 

\begin{proposition}\label{prop:f*}
Let $f\colon A\nto B$ be a dense multiplicative \mM-morphism between
non-degenerate semigroups. There is an induced functor $f^*$ sending
non-degenerate  $B$-modules to non-degenerate $A$-modules. 
For multiplicative and dense \mM-morphisms $A\stackrel f \nto B
\stackrel g \nto C$, the composite functor $g^*f^*$ is naturally isomorphic to
$(g\bullet f)^*$ and the identity $\mM$-morphism $A\nto A$ induces the
identity functor. If we regard \cm as a bicategory with only identity
2-cells, then this defines a pseudofunctor from \cm to the 2-category of
categories, functors and natural transformations.
\end{proposition}

\begin{proof}
Let $v\colon VB\to V$ be a surjective non-degenerate associative 
action of $B$. Then there is a coequalizer   
$$\xymatrix@C=30pt{
X \ar@<0.2ex>[r]^-{x_1} \ar@<-0.5ex>[r]_-{x_2} & VB \ar[r]^{v} & V }$$
which is preserved by monoidal product. In the diagram
$$\xymatrix{
&& VB^2 \ar[r]^{v1} \ar[dr]^{1m} & 
VB \ar[dr]^{v} \\
XAB \ar[r]^{x_i11} \ar[dr]_{1f_1} & 
VBAB \ar[ur]^{1f_21} \ar[dr]_{11f_1} \ar@{}[rr]|-{\eqref{eq:mM}}&& 
VB \ar[r]^{v} & V \\
& XB \ar[r]_{x_i1} & 
VB^2 \ar[ur]^{1m} \ar[r]_{v1} & 
VB \ar[ur]_{v} }$$
the lower path depends only on $v.x_i$  not on $i$, thus the same
is true of the upper path. By non-degeneracy of $v$, it follows that
$v.1f_2.x_i1$ depends only on $v.x_i$, and so that there is a
unique morphism $f^*v$ making the first diagram of 
\begin{equation}\label{eq:f*v}
\xymatrix{
VBA \ar[r]^{v1} \ar[d]_{1f_2} & VA \ar[d]^{f^*v} \\
VB \ar[r]_{v} & V }
\xymatrix{
VAB \ar[r]^{1f_1} \ar[d]_{(f^*v)1} & VB \ar[d]^{v} \\
VB \ar[r]_{v} & V}
\end{equation}
commute. This $f^*v$ also makes the the second diagram of \eqref{eq:f*v} 
commute, by surjectivity of $v$ and commutativity of the
following diagrams.  
$$\xymatrix{
& VAB \ar[r]^{(f^*v)1} \ar@{}[d]|-{\eqref{eq:f*v}}& 
VB \ar[dr]^{v} \\
VBAB \ar[ur]^{v11} \ar[r]^{1f_21} \ar[dr]_{11f_1} & 
VB^2 \ar[ur]^{v1} \ar[dr]_{1m} \ar@{}[d]|-{\eqref{eq:mM}} && 
V \\
& VB^2 \ar[r]_{1m} & 
VB \ar[ur]_{v} }
\quad
\xymatrix{
& VAB \ar[dr]^{1f_1} \\
VBAB \ar[ur]^{v11} \ar[dr]_{11f_1} && 
VB \ar[r]^{v} & 
V \\
& VB^2 \ar[r]_{1m} \ar[ur]^{v1} & 
VB \ar[ur]_{v} }$$

Associativity and  surjectivity of $f^*v$ follow easily from the
corresponding facts about $v$, but non-degeneracy requires a little more work.  
We must show that a morphism $g\colon X\to YV$ may be recovered from 
$$\xymatrix@C=35pt{
XA \ar[r]^-{g1} & 
YVA \ar[r]^-{1(f^*v)} & 
YV.}$$
But from this last composite we may construct the common composite of the
diagram 
$$\xymatrix@C=35pt{
XAB \ar[r]^{g11} \ar[d]_{1f_1} & 
YVAB \ar[r]^{1(f^*v)1} \ar[d]_{11f_1} \ar@{}[rd]|-{\eqref{eq:f*v}
}& 
YVB \ar[d]^{1v} \\
XB \ar[r]_{g1} & 
YVB \ar[r]_{1v} & 
YV , }$$
and now by density of $f$ we may recover $1v.g1$, and so finally by
non-degeneracy of $v$ we may recover $g$.    
The claim that a morphism of $B$-modules is also compatible with the 
$A$-action induced by $f$, as well as the composition law of the resulting
functors, are clear by the construction. The identity morphism $B \nto B$
induces the identity functor in light of the associativity of the actions
$v:VB\to V$. 
\end{proof}

\subsection{Modules over counital fusion morphisms}

Let $t_2\colon A^2\to A^2$ and $e\colon A\to I$ satisfy equations
\eqref{eq:mbm_ax_2}. 
Typically, we shall be interested in the case where this structure is part of
a multiplier bimonoid $(A,t_1,t_2,e)$.   

An {\em action} of the fusion morphism $(A,t_2\colon A^2\to A^2)$ on an object
$V$ --- a $t_2$-action, for short --- is a morphism $v_2\colon VA\to VA$ which
renders commutative the diagram    
\begin{equation}\label{eq:t_2_module}
\xymatrix{
VA^2 \ar[r]^-{v_21} \ar[d]_-{1t_2} &
VA^2 \ar[r]^-{1c} &
VA^2 \ar[r]^-{v_21} &
VA^2 \ar[r]^-{1c^{-1}} &
VA^2 \ar[d]^-{1t_2} \\
VA^2 \ar[rrrr]_-{v_21} &&&&
VA^2 .}
\end{equation}
It follows that the morphism $v:= 1e.v_2\colon VA\to V$ is an associative
action of the semigroup $(A,m)$; if moreover it is an element of \cq, we say
that $(V,v_2)$ is a {\em module} over the counital fusion morphism
$(A,t_2,e)$. 

\begin{remark}\label{rem:action_epi2}
This generalizes the definition of module given in
\cite{BohmLack:braided_mba}, where $v$ was required to be a {\em split}
epimorphism, while here we only ask that it lie in the class \cq. Of course
this is no difference at all if \cq consists of the split epimorphisms. See
also Remark~\ref{rmk:action-epi} above.   
\end{remark}

Thus every module  over the counital fusion morphism $(A,t_2,e)$ has an
underlying module over the semigroup $(A,1e.t_2)$;  in the
non-degenerate case, these two notions turn out to be equivalent.    

\begin{proposition}\label{prop:action-module}
Let $(A,t_2,e)$ be a counital fusion morphism with non-degenerate
multiplication $m: =  1e.t_2$. Then the
forgetful functor from the category of $(A,t_2,e)$-modules to the category of
$(A,m)$-modules is an isomorphism.
\end{proposition}

\begin{proof}
 For any $(A,t_2,e)$-module $v_2:VA \to VA$,  take the diagram
\eqref{eq:t_2_module} and compose with $1e1\colon VA^2\to VA$, and use the
second half of \eqref{eq:mbm_ax_2} to deduce the commutativity of 
\begin{equation}\label{eq:v-v2}
\xymatrix{
VA^2 \ar[r]^{v1} \ar[d]_{1t_2} & 
VA \ar[d]^{v_2} \\
VA^2 \ar[r]_{v1} & 
VA. }
\end{equation}
Since $v1$ is a (regular) epimorphism we see that there can be at most one $t_2$-action 
$v_2$ for a given surjective semigroup 
action $v$. For the converse, we have a coequalizer
$$\xymatrix@C=30pt{
X \ar@<0.2ex>[r]^-{x_1} \ar@<-0.5ex>[r]_-{x_2} & VA \ar[r]^{v} & V }$$
which is preserved  by the monoidal product. In the commutative
diagram  
$$\xymatrix{
&& VA^3 \ar[r]^{v11} \ar[dr]^{11m} & VA^2 \ar[dr]^{1m} \\
XA^2 \ar[r]^{x_i11} \ar[dr]_{1t_1} & 
VA^3 \ar[ur]^{1t_21} \ar[dr]_{11t_1} 
\ar@{}[rr]|-{\eqref{eq:short_compatibility}}&& 
VA^2 \ar[r]^{v1} & VA \\
& XA^2 \ar[r]_{x_i11} & VA^3 \ar[ur]_{1m1} \ar[r]_{v11} & VA^2 \ar[ur]_{v1} }$$
the lower composite depends only on $v.x_i$, not on $i$, thus the
same is true of the upper composite. It follows by non-degeneracy of $m$
that $v1.1t_2.x_i1$ depends only on $v.x_i$ and so finally that
there is a (unique) $v_2$ making \eqref{eq:v-v2} commute. 
We re-obtain $v$ from this $v_2$ as $1e.v_2$ since the morphisms of its
top row are epimorphisms, and the following diagram commutes.
$$
\xymatrix{
VA \ar@/_1.2pc/[dd]_-{v_2} &
VA^2 \ar[l]_-{v1} \ar[ld]_-{1t_2} \ar[r]^-{v1} \ar[rd]^-{1m} &
VA \ar@/^1.2pc/[dd]^-{v} \\
VA^2\ar[rr]^-{11e} \ar[d]^-{v1} &&
VA \ar[d]_-v \\
VA \ar[rr]_-{1e} &&
V}
$$
The fact that this
$v_2$ makes \eqref{eq:t_2_module} commute follows immediately from the fusion
equation \eqref{eq:mbm_ax_2} and the fact that $v11\colon VA^3\to VA^2$ is an
epimorphism. 

This now proves that we have a bijection on objects. It is clear that any
morphism of $(A,t_2,e)$-modules respects the induced $(A,m)$-actions;
 the converse follows by \eqref{eq:v-v2} and surjectivity of the
$(A,m)$-action.    
\end{proof}

If $(V,v)$ and $(W,w)$ are modules over a non-degenerate semigroup
$(A,1e.t_2)$ associated to a counital fusion morphism $(A,t_2,e)$,
there is an  $(A,1e.t_2)$-action $\psi$ on the monoidal product
$VW$ given by the composite   
\begin{equation}\label{eq:psi-formula}
\xymatrix{
VWA \ar[r]^{1w_2} & VWA \ar[r]^{1c^{-1}} & VAW \ar[r]^{v1} & VW }
\end{equation}
where $w_2\colon WA\to WA$ is defined using $w$ as in
Proposition~\ref{prop:action-module}. By the defining property \eqref{eq:v-v2}
of $w_2$, it follows that $\psi$ makes the diagram  
\begin{equation} \label{eq:psi-w}
\xymatrix{
VWA^2 \ar[r]^{11t_2} \ar[d]_{1w1} & 
VWA^2 \ar[r]^{1w1} & 
VWA \ar[r]^{1c^{-1}} & 
VAW \ar[d]^{v1} \\
VWA \ar[rrr]_{\psi} &&& 
VW }
\end{equation}
commute. In general there is no reason why $\psi$ should lie in \cq, but it
clearly will do so if $t_2$ is in \cq. 
 In fact, if the multiplication $1e.t_2$ belongs to $\cq$ --- so that
it renders $A$ itself a module --- then $\psi$ lies in \cq for any modules $V$
and $W$ if and only if  the induced morphism $d_2$  of
\eqref{eq:d_1-2}  does so: see
\cite[Corollary~2.23]{BohmLack:braided_mba}. 
Similarly, the counit $e$ induces an  $(A,1e.t_2)$-action on the
monoidal unit $I$, and this will be a  module provided
that $e\in\cq$. When both $d_2$ and $e$ lie in \cq, then the category of modules
will be monoidal, as in \cite[Corollary~2.23]{BohmLack:braided_mba} once
again.   

\subsection{Modules over multiplier bimonoids}

A module over a multiplier bimonoid $(A,t_1,t_2,e)$ will just mean a module
over the underlying counital fusion morphism $(A,t_2,e)$. In this case,
however, there is an alternative characterization of the action $\psi\colon
VWA\to VW$ of \eqref{eq:psi-formula}: by the calculation
$$
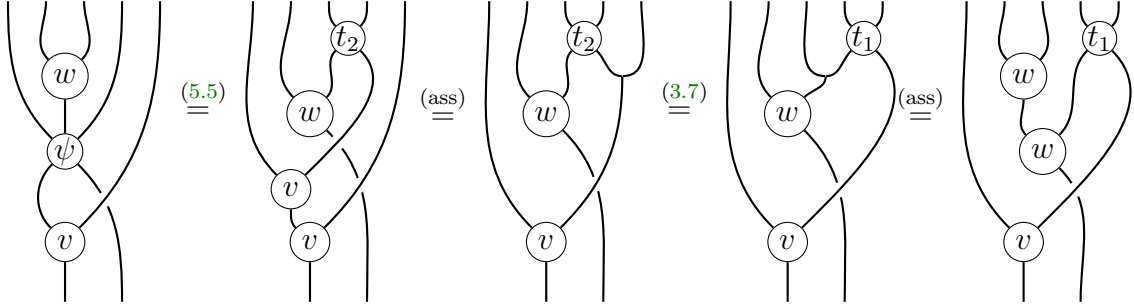
\begin{tikzpicture} 
 \path
(0,4) node[empty, name=in1] {}
(0.5,4) node[empty, name=in2] {}
(1,4) node[empty, name=in3] {}
(1.5,4) node[empty, name=in4] {}
(2,4) node[empty, name=in5] {}
(0.75,0) node[empty, name=out1] {}
(1.5,0) node[empty, name=out2] {}
(0.75,3) node[arr, name=w] {$\ w\ $}
(0.75,2) node[arr, name=psi] {$\psi$}
(0.75,.8) node[arr, name=v] {$\ v\ $};
\draw[braid] (in1) to[out=270,in=135] (psi);
\draw[braid] (in2) to[out=270,in=135] (w);
\draw[braid] (in3) to[out=270,in=45] (w);
\draw[braid] (in4) to[out=270,in=45] (psi);
\path[braid, name path=bin5] (in5) to[out=270,in=45] (v);
\draw[braid] (w) to[out=270,in=90] (psi);
\draw[braid, name path=bout2] (psi) to[out=315, in=90] (out2);
\fill[white,name intersections={of=bout2 and bin5}] (intersection-1) circle(0.1);
\draw[braid] (in5) to[out=270,in=45] (v);
\draw[braid] (psi) to[out=225,in=135] (v);
\draw[braid] (v) to (out1); 
\draw (2.5,2.7) node[empty] {$\stackrel{~\eqref{eq:psi-w}}=$};
\end{tikzpicture}
\begin{tikzpicture} 
 \path
(0,4) node[empty, name=in1] {}
(0.5,4) node[empty, name=in2] {}
(1,4) node[empty, name=in3] {}
(1.5,4) node[empty, name=in4] {}
(2,4) node[empty, name=in5] {}
(0.75,0) node[empty, name=out1] {}
(1.5,0) node[empty, name=out2] {}
(0.75,2.5) node[arr, name=w] {$\ w\ $}
(1.25,3.5) node[arr, name=t2] {$t_2$}
(0.5,1.5) node[arr, name=vu] {$\ v\ $}
(0.75,.8) node[arr, name=v] {$\ v\ $};
\draw[braid] (in1) to[out=270,in=135] (vu);
\draw[braid] (in2) to[out=270,in=135] (w);
\draw[braid] (in3) to[out=270,in=135] (t2);
\draw[braid] (in4) to[out=270,in=45] (t2);
\path[braid, name path=bin5] (in5) to[out=270,in=45] (v);
\draw[braid] (t2) to[out=225,in=45] (w);
\draw[braid, name path=bout2] (w) to[out=315, in=90] (out2);
\path[braid, name path=t2vu] (t2) to[out=315,in=45] (vu);
\fill[white,name intersections={of=bout2 and t2vu}] (intersection-1) circle(0.1);
\draw[braid] (t2) to[out=315,in=45] (vu);
\fill[white,name intersections={of=bout2 and bin5}] (intersection-1) circle(0.1);
\draw[braid] (in5) to[out=270,in=45] (v);
\draw[braid] (vu) to[out=270,in=135] (v);
\draw[braid] (v) to (out1); 
\draw (2.5,2.63) node[empty] {$\stackrel{\mathrm{(ass)}}=$};
\end{tikzpicture}
\begin{tikzpicture} 
 \path
(0,4) node[empty, name=in1] {}
(0.5,4) node[empty, name=in2] {}
(1,4) node[empty, name=in3] {}
(1.5,4) node[empty, name=in4] {}
(2,4) node[empty, name=in5] {}
(0.75,0) node[empty, name=out1] {}
(1.5,0) node[empty, name=out2] {}
(0.75,2.5) node[arr, name=w] {$\ w\ $}
(1.25,3.5) node[arr, name=t2] {$t_2$}
(1.75,3) node[empty, name=m] {}
(0.75,.8) node[arr, name=v] {$\ v\ $};
\draw[braid] (in1) to[out=270,in=135] (v);
\draw[braid] (in2) to[out=270,in=135] (w);
\draw[braid] (in3) to[out=270,in=135] (t2);
\draw[braid] (in4) to[out=270,in=45] (t2);
\draw[braid] (t2) to[out=225,in=45] (w);
\draw[braid, name path=bout2] (w) to[out=315, in=90] (out2);
\draw[braid] (t2) to[out=315,in=180] (m);
\path[braid, name path=mv] (m) to[out=270,in=45] (v);
\fill[white,name intersections={of=bout2 and mv}] (intersection-1) circle(0.1);
\draw[braid] (in5) to[out=270,in=0] (m);
\draw[braid] (m) to[out=270,in=45] (v);
\draw[braid] (v) to (out1); 
\draw (2.5,2.7) node[empty] {$\stackrel{~\eqref{eq:short_compatibility}}=$};
\end{tikzpicture}
\begin{tikzpicture} 
 \path
(0,4) node[empty, name=in1] {}
(0.5,4) node[empty, name=in2] {}
(1,4) node[empty, name=in3] {}
(1.5,4) node[empty, name=in4] {}
(2,4) node[empty, name=in5] {}
(0.75,0) node[empty, name=out1] {}
(1.5,0) node[empty, name=out2] {}
(0.75,2.5) node[arr, name=w] {$\ w\ $}
(1.75,3.5) node[arr, name=t1] {$t_1$}
(1.25,3) node[empty, name=m] {}
(0.75,.8) node[arr, name=v] {$\ v\ $};
\draw[braid] (in1) to[out=270,in=135] (v);
\draw[braid] (in2) to[out=270,in=135] (w);
\draw[braid] (in3) to[out=270,in=180] (m);
\draw[braid] (in4) to[out=270,in=135] (t1);
\draw[braid] (m) to[out=270,in=45] (w);
\draw[braid, name path=bout2] (w) to[out=315, in=90] (out2);
\path[braid, name path=mv] (m) to[out=270,in=45] (v);
\path[braid, name path=t1v] (t1) to[out=315,in=45] (v);
\fill[white,name intersections={of=bout2 and t1v}] (intersection-1) circle(0.1);
\draw[braid] (t1) to[out=315,in=45] (v);
\draw[braid] (in5) to[out=270,in=45] (t1);
\draw[braid] (t1) to[out=225,in=0] (m);
\draw[braid] (v) to (out1); 
\draw (2.5,2.63) node[empty] {$\stackrel{\mathrm{(ass)}}=$};
\end{tikzpicture}
\begin{tikzpicture} 
 \path
(0,4) node[empty, name=in1] {}
(0.5,4) node[empty, name=in2] {}
(1,4) node[empty, name=in3] {}
(1.5,4) node[empty, name=in4] {}
(2,4) node[empty, name=in5] {}
(0.75,0) node[empty, name=out1] {}
(1.5,0) node[empty, name=out2] {}
(1,2.0) node[arr, name=w] {$\ w\ $}
(0.75,3.0) node[arr, name=wu] {$\ w\ $}
(1.75,3.5) node[arr, name=t1] {$t_1$}
(1.25,3) node[empty, name=m] {}
(0.75,.8) node[arr, name=v] {$\ v\ $};
\draw[braid] (in1) to[out=270,in=135] (v);
\draw[braid] (in2) to[out=270,in=135] (wu);
\draw[braid] (in3) to[out=270,in=45] (wu);
\draw[braid] (wu) to[out=270,in=135] (w);
\draw[braid] (in4) to[out=270,in=135] (t1);
\draw[braid, name path=bout2] (w) to[out=315, in=90] (out2);
\path[braid, name path=mv] (m) to[out=270,in=45] (v);
\path[braid, name path=t1v] (t1) to[out=315,in=45] (v);
\fill[white,name intersections={of=bout2 and t1v}] (intersection-1) circle(0.1);
\draw[braid] (t1) to[out=315,in=45] (v);
\draw[braid] (in5) to[out=270,in=45] (t1);
\draw[braid] (t1) to[out=225,in=45] (w);
\draw[braid] (v) to (out1); 
\end{tikzpicture}
$$
and surjectivity of $w$, we may deduce that the diagram 
\begin{equation}\label{eq:psi-v}
\xymatrix{
VWA^2 \ar[r]^{11t_1} \ar[d]_{\psi1} & 
VWA^2 \ar[r]^{1w1} & 
VWA \ar[r]^{1c^{-1}} & 
VAW \ar[d]^{v1} \\
VWA \ar[r]_{1c^{-1}} & 
VAW \ar[rr]_{v1} && 
VW }
\end{equation}
commutes, and by non-degeneracy of $v$ it follows that this characterizes
$\psi$.  

\subsection{Modules over regular multiplier bimonoids} \label{sect:mod_for_reg}

Now consider a regular multiplier bimonoid  $(A,t_1,t_2,t_3,t_4,e)$. A module
\cite{BohmLack:braided_mba} over this multiplier bimonoid is a tuple
$(V,v_2,v_3)$ where $v_2:VA \to VA$ makes $V$ into  a module over $(A,t_2,e)$
and $v_3\colon AV\to AV$ renders commutative the analogue  
\begin{equation}\label{eq:t_3_module}
\xymatrix{
A^2V \ar[r]^-{1v_3} \ar[d]_-{t_31} &
A^2V \ar[r]^-{c^{-1}1} &
A^2V \ar[r]^-{1v_3} &
A^2V \ar[r]^-{c1} &
A^2V \ar[d]^-{t_31} \\
A^2V \ar[rrrr]_-{1v_3} &&&&
A^2V \ }
\end{equation}
of \eqref{eq:t_2_module},
subject to the compatibility conditions expressed by the commutativity of
\begin{equation}\label{eq:module_compatibility}
\xymatrix@R=10pt{
VA^2 \ar[r]^-{v_21} \ar[dd]_-{1t_3} &
VA^2 \ar[dd]^-{1t_3}
&
A^2V \ar[r]^-{1v_3} \ar[dd]_-{t_21} &
A^2V \ar[dd]^-{t_21}
&
VA \ar[r]^-{v_2} \ar[d]_-{c} 
&
VA \ar[dd]^-{1e} \\
&
&
&
& 
AV \ar[d]_-{v_3} \\
VA^2 \ar[r]_-{v_21} &
VA^2
&
A^2V \ar[r]_-{1v_3} &
A^2V
&
AV \ar[r]_-{e1} &
V .}
\end{equation}
(There is no need to impose separately a surjectivity condition
on $ e1.v_3$, since this follows from surjectivity for $1e.v_2$
 and the last compatibility condition.)  

A morphism of modules is just a morphism in the underlying category,
satisfying the evident compatibility conditions with $v_2$ and $v_3$.

We now have the following extension of Proposition~\ref{prop:action-module}.

\begin{proposition}\label{prop:action-module2}
The category of modules of a regular multiplier bimonoid
$(A,t_1,t_2,$ $t_3,t_4,e)$ with non-degenerate multiplication 
$m:=e1.t_1$  is isomorphic to the category of modules  of the underlying
semigroup $(A, m)$.   
\end{proposition}

\begin{proof}
We have already seen that there is a unique $t_2$-action $v_2$
corresponding to any surjective associative action $v:VA\to V$. Dually, there
is a unique $v_3$ making the diagram  
\begin{equation}\label{eq:v-v3}
\xymatrix{
A^2V \ar[r]^{1c^{-1}} \ar[d]_{t_31} & AVA \ar[r]^{1v} & AV \ar[d]^{v_3} \\
A^2V \ar[r]_{1c^{-1}} & AVA \ar[r]_{1v} & AV }
\end{equation}
commute, and furthermore this now implies that \eqref{eq:t_3_module} commutes. 

As for \eqref{eq:module_compatibility}, the first two diagrams commute by 
commutativity of the last diagram in \eqref{eq:reg_mbm} and surjectivity
of the semigroup action $v$, while the third commutes since both paths
yield $v$ by the construction of $v_2$ and $v_3$.   

This now proves that we have a bijection on objects. It is clear that any
morphism of modules over the regular multiplier bimonoid respects
the induced  semigroup actions; the converse follows by
\eqref{eq:v-v2}, \eqref{eq:v-v3}, and  surjectivity of the semigroup
action.
\end{proof}

\subsection{Duals in the base category}

\begin{proposition}\label{prop:rev-action-on-dual}
Let $A$ be a semigroup  in a monoidal category. Let $V$ be an object with
a left dual $\overline{V}$, and let $w\colon AV\to V$ be a surjective
 associative (left) action. Then the action $v\colon \overline{V}A\to
\overline{V}$ given by   
$$\xymatrix{
\overline{V}A \ar[r]^-{11\eta} & \overline{V}AV\overline{V} \ar[r]^{1w1} & 
\overline{V}V\overline{V} \ar[r]^-{\epsilon1} & \overline{V} }$$
is associative and non-degenerate. If moreover $w$ is non-degenerate, then $v$
is an epimorphism, preserved by the monoidal product.  
\end{proposition}

\begin{proof}
Associativity follows easily from associativity of $w$; this is probably
easiest done using string diagrams. As for non-degeneracy, we must show
that, for any objects $X$ and $Y$, the morphism $g\colon X\to Y\overline{V}$
can be recovered from
$$\xymatrix{
XA \ar[r]^{g1} & Y\overline{V}A \ar[r]^{1v} & Y\overline{V}.}$$
From $1v.g1$ we may form the common composite of
$$\xymatrix{
XAV \ar[r]^{g11} \ar[d]_{1w} & Y\overline{V}AV \ar[r]^{1v1} \ar[d]_{11w} & 
Y\overline{V}V \ar[d]^{1\epsilon} \\
XV \ar[r]_{g1} & Y\overline{V}V \ar[r]_{1\epsilon} & Y}$$
and since $1w$ is epimorphic this allows us to recover the composite
$1\epsilon.g1$ along the bottom, and this in turn allows us to recover $g$
using the duality.  
  
Now suppose that $w$ is non-degenerate, and let $x_1,x_2\colon
\overline{V}Y\to Z$ satisfy $x_1.v1=x_2.v1$. By commutativity of the
diagram  
$$
\xymatrix{
\overline V AY \ar[r]^-{11\eta 1} \ar[d]_-{v1} &
\overline V AV\overline V Y \ar[r]^-{111x_i} &
\overline V AVZ \ar[r]^-{1w1} \ar[d]^-{v11} &
\overline V VZ \ar[dd]^-{\epsilon 1} \\
\overline V Y \ar[r]^-{1\eta 1} \ar@{=}[d] &
\overline V V \overline V Y \ar[r]^-{11x_i} \ar[ld]_-{\epsilon 11} &
\overline V VZ \ar[rd]^-{\epsilon 1} \\
\overline V Y \ar[rrr]_-{x_i} &&&
Z}
$$
we see that the upper composite is independent of $i$; using the duality once
it follows that $w1.11x_i.1\eta1$ is independent of $i$, then by
non-degeneracy of $w$ it follows that $1x_i.\eta1$ is independent of $i$, and
finally using duality again it follows that $x_i$ is independent of $i$; in
other words, that $x_1=x_2$.  
\end{proof}

Take now a multiplier bimonoid $(A,t_1,t_2,e)$ with non-degenerate
multiplication and dense counit. Assume that it is Hopf, with antipode
$s\colon A\op\nto A$. Let $V$ be an object possessing a left dual 
$\overline{V}$, and a surjective associative non-degenerate
action $v\colon VA\to V$. Then the action $s^*v\colon VA\op\to V$ of
Proposition~\ref{prop:f*} may be regarded as a left action 
$\xymatrix@C=15pt{
AV \ar[r]^-{c^{-1}} &
VA \ar[r]^-{s^*v} &
V}$, 
and so by Proposition~\ref{prop:rev-action-on-dual}
we obtain  an action $\overline{v}\colon \overline{V}A\to \overline{V}$;
explicitly, this is given by  
\begin{equation}\label{eq:overline-v}
\xymatrix@C=30pt{
  \overline{V}A \ar[r]^-{11\eta} & \overline{V}AV\overline{V}
  \ar[r]^{1c^{-1}1} & \overline{V}VA\overline{V} \ar[r]^{1(s^*v)1} &
  \overline{V}V\overline{V}  \ar[r]^{\epsilon 1} & \overline{V}. } 
\end{equation}

\begin{theorem}\label{thm:dual_module}
Let \cc be a {\em symmetric} monoidal category satisfying our standing
assumptions of Section~\ref{sect:assumptions}. Let $(A,t_1,t_2,e)$ be a
multiplier Hopf monoid whose multiplication $m:=e1.t_1$ is non-degenerate and
whose counit is dense. Finally, let  $V$ be a  module over
the multiplier bimonoid $(A,t_1,t_2,e)$ whose underlying $(A,m)$-action
$v\colon VA\to V$ is non-degenerate. If $V$ has a left
dual $\overline{V}$ in \cc and the morphism $\overline{v}$ in
\eqref{eq:overline-v} is in \cq, then it makes $\overline{V}$ into an
$(A,t_1,t_2,e)$-module which is left dual to $V$.   
\end{theorem}

\begin{proof}
By Proposition~\ref{prop:rev-action-on-dual} and the assumption that
$\overline{v}$ is in \cq,  $\overline{v}$ makes $\overline{V}$ a module
over the semigroup $(A,m)$ and so by Proposition~\ref{prop:action-module} a
module over the multiplier bimonoid $(A,t_1,t_2,e)$. We need only show that
the unit and counit of the duality are morphisms of $(A,m)$-modules.  

In the case of the unit, the key  is the following calculation linking $v$ and
$\overline v$. 
$$
\begin{tikzpicture} 
\path 
(0.5,3) node [empty, name=in1] {}
(1.0,3) node[empty, name=in2] {}
(0,0) node[empty, name=out1] {}
(0.5,0) node[empty, name=out2] {}
(0,0.5) node[arr, name=v] {$\ v\ $}
(0.35,2) node[arr, inner sep=1pt, name=v'] {$\, \overline v\, $}
(0,2.5) node[empty, name=x] {};
\draw[braid] (in1) to[out=270,in=45] (v');
\draw[braid, name path=bin2] (in2) to[out=270,in=45] (v);
\draw[braid] (v') to[out=135,in=0] (x) to[out=180,in=135] (v);
\path[braid, name path=bout2] (v') to[out=270,in=90] (out2);
\fill[white,name intersections={of=bin2 and bout2}] (intersection-1) circle(0.1);
\draw[braid] (v') to[out=270,in=90] (out2);
\draw[braid] (v) to (out1);
\draw (1.5,1.7) node[empty] {$\stackrel{~\eqref{eq:overline-v}}=$};
\end{tikzpicture}
\begin{tikzpicture} 
\path 
(0.5,3) node [empty, name=in1] {}
(1.0,3) node[empty, name=in2] {}
(0,0) node[empty, name=out1] {}
(0.5,0) node[empty, name=out2] {}
(0,0.5) node[arr, name=v] {$\ v\ $}
(0.5,2) node[arr,  name=sv] {$s^*v$}
(0,2.5) node[empty, name=x] {}
(0.25,1.5) node[empty, name=y] {}
(0.75,2.8) node[empty, name=z] {};
\draw[braid, name path=bin1] (in1) to[out=270,in=45] (sv);
\draw[braid, name path=bin2] (in2) to[out=270,in=45] (v);
\path[braid, name path=bout2] (sv) to[out=135,in=180] (z) to[out=0,in=90] (out2);
\fill[white,name intersections={of=bin2 and bout2}] (intersection-1) circle(0.1);
\fill[white,name intersections={of=bin1 and bout2}] (intersection-1) circle(0.1);
\draw[braid] (sv) to[out=135,in=180] (z) to[out=0,in=90] (out2);
\draw[braid] (sv) to[out=270,in=0] (y) to[out=180,in=0] (x) to[out=180,in=135] (v);
\draw[braid] (v) to (out1);
\draw (1.5,1.5) node[empty] {$=$};
\end{tikzpicture}
\begin{tikzpicture} 
\path 
(0.5,3) node [empty, name=in1] {}
(1.0,3) node[empty, name=in2] {}
(0,0) node[empty, name=out1] {}
(0.5,0) node[empty, name=out2] {}
(0,0.5) node[arr, name=v] {$\ v\ $}
(0.25,2) node[arr, name=sv] {$s^*v$}
(0,2.5) node[empty, name=x] {}
(0.25,1.5) node[empty, name=y] {}
(0.75,2.8) node[empty, name=z] {};
\draw[braid, name path=bin1] (in1) to[out=270,in=45] (sv);
\draw[braid, name path=bin2] (in2) to[out=270,in=45] (v);
\path[braid, name path=bout2] (sv) to[out=135,in=180] (z) to[out=0,in=90] (out2);
\fill[white,name intersections={of=bin2 and bout2}] (intersection-1) circle(0.1);
\fill[white,name intersections={of=bin1 and bout2}] (intersection-1) circle(0.1);
\draw[braid] (sv) to[out=135,in=180] (z) to[out=0,in=90] (out2);
\draw[braid] (sv) to[out=270,in=135] (v);
\draw[braid] (v) to (out1);
\draw (1.5,1.7) node[empty] {$\stackrel{\eqref{eq:f*v}}=$};
\end{tikzpicture}
\begin{tikzpicture} 
\path 
(0.2,3) node [empty, name=in1] {}
(0.75,3) node[empty, name=in2] {}
(0,0) node[empty, name=out1] {}
(0.5,0) node[empty, name=out2] {}
(0,0.5) node[arr, name=v] {$\ v\ $}
(0.5,2) node[arr, name=s1] {$s_1$}
(0.75,2.5) node[empty, name=x] {};
\draw[braid, name path=bin1] (in1) to[out=270,in=135] (s1);
\draw[braid, name path=bin2] (in2) to[out=270,in=45] (s1);
\draw[braid] (s1) to[out=270, in=45] (v);
\path[braid, name path=bout2] (v) to[out=135,in=180] (x) to[out=0, in=90] (out2);
\fill[white,name intersections={of=bin2 and bout2}] (intersection-1) circle(0.1);
\fill[white,name intersections={of=bin1 and bout2}] (intersection-1) circle(0.1);
\draw[braid] (v) to[out=135,in=180] (x) to[out=0, in=90] (out2);
\draw[braid] (v) to (out1);
\draw (1.5,1.5) node[empty] {$=$};
\end{tikzpicture}
\begin{tikzpicture} 
\path 
(0.2,3) node [empty, name=in1] {}
(0.75,3) node[empty, name=in2] {}
(0,0) node[empty, name=out1] {}
(0.5,0) node[empty, name=out2] {}
(0,0.5) node[arr, name=v] {$\ v\ $}
(0.5,2) node[arr, name=s1] {$s_1$}
(0.5,1.5) node[empty, name=x] {};
\draw[braid, name path=bin1] (in1) to[out=270,in=135] (s1);
\draw[braid, name path=bin2] (in2) to[out=270,in=45] (s1);
\draw[braid, name path=s1v] (s1) to[out=270, in=45] (v);
\path[braid, name path=bout2] (v) to[out=135,in=180] (x) to[out=0, in=90] (out2);
\fill[white,name intersections={of=s1v and bout2}] (intersection-1) circle(0.1);
\draw[braid] (v) to[out=135,in=180] (x) to[out=0, in=90] (out2);
\draw[braid] (v) to (out1);
\end{tikzpicture}
$$
If  $\psi$ is the action for the monoidal product $V\overline V$, then, using
this last calculation and the fact that the braiding is a symmetry for the
second step of the next calculation, we have    
$$
\begin{tikzpicture}  
\path
(1,3) node[empty, name=in1] {}
(1.5,3) node[empty, name=in2] {}
(0.5,0) node[empty, name=out1] {}
(1,0) node[empty, name=out2] {}
(0.5,0.5) node[arr, name=v] {$\ v\ $}
(0.5,1.5) node[arr, name=psi] {$\psi$}
(0.25,2.5) node[empty, name=x] {};
\draw[braid] (in1) to[out=270,in=45] (psi);
\draw[braid] (psi) to[out=135,in=180] (x) to[out=0,in=90] (psi);
\draw[braid] (psi) to[out=225,in=135] (v);
\draw[braid] (v) to (out1);
\path[braid, name path=bin2] (in2) to[out=270, in=45] (v);
\draw[braid, name path=bout2] (psi) to[out=315,in=90] (out2);
\fill[white,name intersections={of=bin2 and bout2}] (intersection-1) circle(0.1);
\draw[braid] (in2) to[out=270, in=45] (v);
\draw (1.8,1.7) node[empty] {$\stackrel{~\eqref{eq:psi-v}}=$};
\end{tikzpicture}
\begin{tikzpicture}  
\path
(1,3) node[empty, name=in1] {}
(1.5,3) node[empty, name=in2] {}
(0.5,0) node[empty, name=out1] {}
(1,0) node[empty, name=out2] {}
(0.5,0.5) node[arr, name=v] {$\ v\ $}
(1.25,2.5) node[arr, name=t1] {$t_1$}
(0.75,1.5) node[arr, inner sep=1pt, name=v'] {$\, \overline v\, $}
(0.25,2.5) node[empty, name=x] {};
\draw[braid] (in1) to[out=270,in=135] (t1);
\draw[braid, name path=bin2] (in2) to[out=270, in=45] (t1);
\draw[braid] (v) to[out=135,in=180] (x) to[out=0,in=135] (v');
\draw[braid] (v) to (out1);
\draw[braid] (t1) to[out=225, in=45] (v');
\draw[braid, name path=bout2] (v') to[out=270,in=90] (out2);
\path[braid, name path=t1v] (t1) to[out=315, in=45] (v);
\fill[white,name intersections={of=t1v and bout2}] (intersection-1) circle(0.1);
\draw[braid] (t1) to[out=315, in=45] (v);
\draw (2,1.5) node[empty] {$=$};
\end{tikzpicture}
\begin{tikzpicture}  
\path
(1,3) node[empty, name=in1] {}
(1.5,3) node[empty, name=in2] {}
(0.5,0) node[empty, name=out1] {}
(1,0) node[empty, name=out2] {}
(0.5,0.5) node[arr, name=v] {$\ v\ $}
(1.25,2.5) node[arr, name=t1] {$t_1$}
(1.25,1.75) node[arr, name=s1] {$s_1$}
(0.5,1.5) node[empty, name=x] {};
\draw[braid] (in1) to[out=270,in=135] (t1);
\draw[braid, name path=bin2] (in2) to[out=270, in=45] (t1);
\draw[braid] (t1) to[out=225,in=135] (s1);
\draw[braid] (t1) to[out=315,in=45] (s1);
\path[braid, name path=s1v] (s1) to[out=270, in=45] (v);
\draw[braid, name path=bout2] (v) to[out=135,in=180] (x) to[out=0,in=90] (out2);
\fill[white,name intersections={of=s1v and bout2}] (intersection-1) circle(0.1);
\draw[braid] (s1) to[out=270, in=45] (v);
\draw[braid] (v) to (out1);
\draw (2,1.7) node[empty] {$\stackrel{~\eqref{eq:s}}=$};
\end{tikzpicture}
\begin{tikzpicture}  
\path
(1,3) node[empty, name=in1] {}
(1.5,3) node[empty, name=in2] {}
(0.5,0) node[empty, name=out1] {}
(1,0) node[empty, name=out2] {}
(0.5,0.5) node[arr, name=v] {$\ v\ $}
(1,2) node[unit, name=e] {}
(0.5,1.5) node[empty, name=x] {};
\draw[braid] (in1) to[out=270,in=90] (e);
\path[braid, name path=bin2] (in2) to[out=270, in=45] (v);
\draw[braid, name path=bout2] (v) to[out=135,in=180] (x) to[out=0,in=90] (out2);
\fill[white,name intersections={of=bin2 and bout2}] (intersection-1) circle(0.1);
\draw[braid] (in2) to[out=270, in=45] (v);
\draw[braid] (v) to (out1);
\end{tikzpicture}
$$
and so by, non-degeneracy of $v$, the diagram
$$\xymatrix{
A \ar[r]^-{\eta1} \ar[d]_{e} & V\overline VA \ar[d]^{\psi} \\
I \ar[r]_-{\eta} & V\overline V}$$
commutes, and so $\eta$ is a morphism of modules. 

In the case of the counit, once again we start with a calculation linking $v$
and $\overline v$. 
$$
\begin{tikzpicture} 
\path
(0,0) node[empty] {}
(0,3) node[empty, name=in1] {}
(0.5,3) node[empty, name=in2] {}
(1,3) node[empty, name=in3] {}
(1.5,3) node[empty, name=in4] {}
(0.25,2) node[arr, inner sep=1pt, name=v'] {$\, \overline v\, $}
(1.25,2) node[arr, name=v] {$\ v\ $};
\draw[braid] (in1) to[out=270,in=135] (v');
\draw[braid] (in2) to[out=270,in=45] (v');
\draw[braid] (in3) to[out=270,in=135] (v);
\draw[braid] (in4) to[out=270,in=45] (v);
\draw[braid] (v') to[out=270,in=270] (v);
\draw (1.8,1.7) node[empty] {$\stackrel{~\eqref{eq:overline-v}}=$};
\end{tikzpicture}
\begin{tikzpicture} 
\path
(0,0) node[empty] {}
(0,3) node[empty, name=in1] {}
(0.5,3) node[empty, name=in2] {}
(1,3) node[empty, name=in3] {}
(1.5,3) node[empty, name=in4] {}
(0.5,2.75) node[empty, name=x] {}
(0,1.5) node[empty, name=y] {}
(1,1) node[empty, name=z] {}
(0.25,2) node[arr, name=sv] {$s^*v$}
(1.25,2) node[arr, name=v] {$\ v\ $};
\draw[braid] (in1) to[out=270, in=180] (y) to[out=0,in=270] (sv);
\draw[braid, name path=bin2] (in2) to[out=270,in=45] (sv);
\draw[braid] (in3) to[out=270,in=135] (v);
\draw[braid] (in4) to[out=270,in=45] (v);
\path[braid, name path=svv] (sv) to[out=135,in=180] (x) to[out=0,in=180] (z) to[out=0,in=270] (v);
\fill[white,name intersections={of=bin2 and svv}] (intersection-1) circle(0.1);
\draw[braid] (sv) to[out=135,in=180] (x) to[out=0,in=180] (z) to[out=0,in=270] (v);
\draw (1.8,1.5) node[empty] {$=$};
\end{tikzpicture}
\begin{tikzpicture} 
\path
(0,0) node[empty] {}
(0,3) node[empty, name=in1] {}
(0.5,3) node[empty, name=in2] {}
(1,3) node[empty, name=in3] {}
(1.5,3) node[empty, name=in4] {}
(0.5,0.5) node[empty, name=y] {}
(1.0,1) node[arr, name=sv] {$s^*v$}
(0.5,2) node[arr, name=v] {$\ v\ $};
\draw[braid] (in1) to[out=270, in=180] (y) to[out=0,in=270] (sv);
\draw[braid, name path=bin2] (in2) to[out=270,in=45] (sv);
\path[braid, name path=bin3] (in3) to[out=270,in=135] (v);
\path[braid, name path=bin4] (in4) to[out=270,in=45] (v);
\fill[white,name intersections={of=bin2 and bin3}] (intersection-1) circle(0.1);
\fill[white,name intersections={of=bin2 and bin4}] (intersection-1) circle(0.1);
\draw[braid] (in3) to[out=270,in=135] (v);
\draw[braid] (in4) to[out=270,in=45] (v);
\draw[braid] (v) to[out=270,in=135] (sv);
\draw (1.8,1.7) node[empty] {$\stackrel{~\eqref{eq:f*v}}=$};
\end{tikzpicture}
\begin{tikzpicture} 
\path
(0,0) node[empty] {}
(0,3) node[empty, name=in1] {}
(0.5,3) node[empty, name=in2] {}
(1,3) node[empty, name=in3] {}
(1.5,3) node[empty, name=in4] {}
(0.25,2) node[empty, name=x] {}
(0.25,0.5) node[empty, name=y] {}
(1.0,2) node[arr, name=s2] {$s_2$}
(0.5,1) node[arr, name=v] {$\ v\ $};
\draw[braid] (in1) to[out=270, in=180] (y) to[out=0,in=270] (v);
\draw[braid, name path=bin2] (in2) to[out=270,in=45] (s2);
\path[braid, name path=bin3] (in3) to[out=270,in=85] (x) to[out=275,in=135] (v);
\path[braid, name path=bin4] (in4) to[out=270,in=135] (s2);
\fill[white,name intersections={of=bin2 and bin3}] (intersection-1) circle(0.1);
\fill[white,name intersections={of=bin2 and bin4}] (intersection-1) circle(0.1);
\draw[braid] (in3) to[out=270,in=85] (x) to[out=275, in=135] (v);
\draw[braid] (in4) to[out=270,in=135] (s2);
\draw[braid] (s2) to[out=270,in=45] (v);
\draw (1.8,1.5) node[empty] {$=$};
\end{tikzpicture}
\begin{tikzpicture} 
\path
(0,0) node[empty] {}
(0,3) node[empty, name=in1] {}
(0.5,3) node[empty, name=in2] {}
(1,3) node[empty, name=in3] {}
(1.5,3) node[empty, name=in4] {}
(0.25,2) node[empty, name=x] {}
(0.25,0.5) node[empty, name=y] {}
(1.0,2) node[arr, name=s2] {$s_2$}
(0.5,1) node[arr, name=v] {$\ v\ $};
\draw[braid] (in1) to[out=270, in=180] (y) to[out=0,in=270] (v);
\path[braid, name path=bin2] (in2) to[out=270,in=45] (s2);
\draw[braid, name path=bin3] (in3) to[out=270,in=85] (x) to[out=275, in=135] (v);
\draw[braid, name path=bin4] (in4) to[out=270,in=135] (s2);
\fill[white,name intersections={of=bin2 and bin3}] (intersection-1) circle(0.1);
\fill[white,name intersections={of=bin2 and bin4}] (intersection-1) circle(0.1);
\draw[braid] (in2) to[out=270,in=45] (s2);
\draw[braid] (s2) to[out=270,in=45] (v);
\end{tikzpicture}
$$
where the last step uses the fact that the braiding is a symmetry.

If $\phi$ is the action for the monoidal product 
$\overline VV$, then, using this last calculation for the third step, we have 
$$
\begin{tikzpicture} 
\path
(0,3) node[empty, name=in1] {}
(0.5,3) node[empty, name=in2] {}
(1.0,3) node[empty, name=in3] {}
(1.5,3) node[empty, name=in4] {}
(0.75,0.5) node[empty, name=x] {}
(0.75,2) node[arr, name=v] {$\ v\ $}
(0.75,1) node[arr, name=phi] {$\, \phi\, $};
\draw[braid] (in1) to[out=270,in=135] (phi);
\draw[braid] (in2) to[out=270,in=135] (v);
\draw[braid] (in3) to[out=270,in=45] (v);
\draw[braid] (in4) to[out=270,in=45] (phi);
\draw[braid] (v) to (phi);
\draw[braid] (phi) to[out=225,in=180] (x) to[out=0,in=315] (phi);
\draw (1.8,1.7) node[empty] {$\stackrel{~\eqref{eq:psi-w}}=$};
\end{tikzpicture}
\begin{tikzpicture} 
\path
(0,3) node[empty, name=in1] {}
(0.5,3) node[empty, name=in2] {}
(1.0,3) node[empty, name=in3] {}
(1.5,3) node[empty, name=in4] {}
(0.75,0.5) node[empty, name=x] {}
(0.75,2) node[arr, name=v] {$\ v\ $}
(1.25,2.5) node[arr, name=t2] {$t_2$}
(0.5,1) node[arr, inner sep=1pt, name=v'] {$\, \overline v\, $};
\draw[braid] (in1) to[out=270,in=135] (v');
\draw[braid] (in2) to[out=270,in=135] (v);
\draw[braid] (in3) to[out=270,in=135] (t2);
\draw[braid] (in4) to[out=270,in=45] (t2);
\draw[braid] (t2) to (v);
\draw[braid, name path=vv'] (v') to[out=270,in=180] (x) to[out=0,in=270] (v);
\path[braid, name path=t2v'] (t2) to[out=315,in=45] (v');
\fill[white,name intersections={of=t2v' and vv'}] (intersection-1) circle(0.1);
\draw[braid] (t2) to[out=315,in=45] (v');
\draw (1.8,1.5) node[empty] {$=$};
\end{tikzpicture}
\begin{tikzpicture} 
\path
(0,3) node[empty, name=in1] {}
(0.5,3) node[empty, name=in2] {}
(1.0,3) node[empty, name=in3] {}
(1.5,3) node[empty, name=in4] {}
(0.75,0.5) node[empty, name=x] {}
(1.25,1) node[arr, name=v] {$\ v\ $}
(1.25,2.5) node[arr, name=t2] {$t_2$}
(0.5,1) node[arr, inner sep=1pt, name=v'] {$\, \overline v\, $};
\draw[braid] (in1) to[out=270,in=135] (v');
\draw[braid, name path=bin2] (in2) to[out=270,in=135] (v);
\draw[braid] (in3) to[out=270,in=135] (t2);
\draw[braid] (in4) to[out=270,in=45] (t2);
\draw[braid, name path=t2v] (t2) to[out=225,in=45]  (v);
\draw[braid, name path=vv'] (v') to[out=270,in=180] (x) to[out=0,in=270] (v);
\path[braid, name path=t2v'] (t2) to[out=315,in=45] (v');
\fill[white,name intersections={of=t2v' and bin2}] (intersection-1) circle(0.1);
\fill[white,name intersections={of=t2v' and t2v}] (intersection-1) circle(0.1);
\draw[braid] (t2) to[out=315,in=45] (v');
\draw (1.8,1.5) node[empty] {$=$};
\end{tikzpicture}
\begin{tikzpicture} 
\path
(0,3) node[empty, name=in1] {}
(0.5,3) node[empty, name=in2] {}
(1.0,3) node[empty, name=in3] {}
(1.5,3) node[empty, name=in4] {}
(0.75,0.5) node[empty, name=x] {}
(1.0,1) node[arr, name=v] {$\ v\ $}
(1.25,2.5) node[arr, name=t2] {$t_2$}
(1.25,1.5) node[arr, name=s2] {$s_2$};
\draw[braid] (in1) to[out=270,in=180] (x) to[out=0, in=270] (v);
\draw[braid, name path=bin2] (in2) to[out=270,in=135] (v);
\draw[braid] (in3) to[out=270,in=135] (t2);
\draw[braid] (in4) to[out=270,in=45] (t2);
\draw[braid] (t2) to[out=225,in=135] (s2);
\draw[braid] (t2) to[out=315,in=45] (s2);
\draw[braid] (s2) to[out=270,in=45] (v);
\draw (2.5,1.7) node[empty] {$\stackrel{~\eqref{eq:s}}=$};
\end{tikzpicture}
\begin{tikzpicture} 
\path
(0,3) node[empty, name=in1] {}
(0.5,3) node[empty, name=in2] {}
(1.0,3) node[empty, name=in3] {}
(1.5,3) node[empty, name=in4] {}
(0.75,0.5) node[empty, name=x] {}
(1.0,1) node[arr, name=v] {$\ v\ $}
(1.5,2) node[unit, name=e] {};
\draw[braid] (in1) to[out=270,in=180] (x) to[out=0, in=270] (v);
\draw[braid, name path=bin2] (in2) to[out=270,in=135] (v);
\draw[braid] (in3) to[out=270,in=45] (v);
\draw[braid] (in4) to[out=270,in=90] (e);
\end{tikzpicture}
$$
and so by the surjectivity of $v$ the diagram 
$$\xymatrix{
\overline VVA \ar[d]_{\phi} \ar[r]^-{\epsilon 1} & A \ar[d]^{e} \\
\overline VV \ar[r]_-{\epsilon} & I }$$
commutes, and so $\epsilon$ is a morphism of modules. 
\end{proof}

\begin{remark}
In Theorem~\ref{thm:dual_module} we have had to assume that the
action $\overline{v}$ lies in \cq; otherwise, the hypotheses only guarantee
that it is an epimorphism. If the base category is abelian, and so in
particular if it is the category of modules over a (commutative) ring, then
every epimorphism is a regular epimorphism; thus if moreover \cq consists of
all regular epimorphisms, then the action $\overline{v}$ will always be
surjective.  
\end{remark}


\section{The fundamental theorem of Hopf modules}

For a regular multiplier bimonoid $A$ in a braided monoidal category
$\cc$ having the properties in Section \ref{sect:mbm}, a Hopf module 
is defined as an object $V$ equipped with $A$-module and $A$-comodule
structures, subject to a compatibility condition requiring that the action $VA
\to V$ is a morphism of comodules. Taking the monoidal product with $A$
induces a functor from $\cc$ to the category of Hopf modules. When it is an
equivalence, $A$ is said to satisfy the {\em fundamental theorem of Hopf
modules}. In this section we investigate when this happens.  

We shall say that a monomorphism is {\em $A$-pure} if it is preserved by 
taking the monoidal product with $A$. 

\begin{definition} \label{def:Hopf_module}
For a regular multiplier bimonoid $(A,t_1,t_2,t_3,t_4,e)$ in a braided
monoidal category $\cc$ having the properties in Section \ref{sect:mbm},
a {\em Hopf module} is a tuple $(V,v^1,v^3,v_2,v_3)$ where
$(V,v^1,v^3)$ is a comodule as in Section~\ref{claim:comodule}, $(V,v_2,v_3)$
is a module as in Section~\ref{sect:mod_for_reg}, and either (hence by
Lemma \ref{lem:comodule_morphism_nd} both) of the so-called Hopf
compatibility conditions -- saying that $v :=1e.v_2 \colon VA\to V$ is a
morphism of comodules -- hold: 
\begin{equation}\label{eq:Hopf_module}
\xymatrix{
VA^2 \ar[r]^-{1t_1} \ar[d]_-{v1} &
VA^2 \ar[r]^-{c1} &
AVA \ar[r]^-{1v^1} &
AVA \ar[r]^-{c^{-1}1} &
VA^2 \ar[d]^-{v1}\\
VA \ar[rrrr]_-{v^1} &&&&
VA\\
VA^2 \ar[r]^-{1c^{-1}} \ar[d]_-{v1} &
VA^2 \ar[r]^-{v^31} &
VA^2 \ar[r]^-{1c} &
VA^2 \ar[r]^-{1t_3} &
VA^2 \ar[d]^-{v1}\\
VA \ar[rrrr]_-{v^3} &&&&
VA\ .}
\end{equation}
A {\em morphism of Hopf modules} $(V,v^1,v^3,v_2,v_3)\to (W,w^1,w^3,w_2,w_3)$
is a morphism $V \to W$ in $\cc$ which is both a morphism of modules and a
morphism of comodules.  
\end{definition}

\begin{proposition}\label{prop:comparison_functor}
Let $\cc$ be a braided monoidal category satisfying our standing
assumptions of Section~\ref{sect:assumptions}, and
$(A,t_1,t_2,t_3,t_4,e)$ be a regular multiplier bimonoid in $\cc$ having the
properties in Section \ref{sect:mbm}. For any object $X$ of $\cc$, there is a
Hopf module $XA$ with comodule structure   
$$
\xymatrix{(XA^2 \ar[r]^-{1t_1}&XA^2,\ 
X A^2 \ar[r]^-{1t_3}&XA^2)}
$$
and module structure
$$
\xymatrix{(XA^2 \ar[r]^-{1t_2}&XA^2,\ 
AXA \ar[r]^-{c^{-1}1} &
XA^2 \ar[r]^-{1t_3} &
XA^2 \ar[r]^-{c1} &
AXA)\ .}
$$
This is the object part of a functor $(-)A$ from $\cc$ to the category of
$A$-Hopf modules. 
\end{proposition}

\begin{proof}
By the fusion and counit equations for $t_1$, the morphism $1t_1\colon XA^2\to
XA^2$ equips $XA$ with the structure of a comodule over $t_1$, in the sense
that the diagrams of \eqref{eq:t_1_comodule} commute. Similarly, $1t_3$ makes
$XA$ into a comodule over $t_3$. 
The compatibility condition \eqref{eq:comodule_compatibility} holds by the
second diagram of \eqref{eq:reg_mbm}. 

The morphism $1t_2\colon XA^2\to XA^2$ renders commutative the diagram
of \eqref{eq:t_2_module} by the fusion equation on $t_2$; and
$c1.1t_3.c^{-1}1\colon AXA \to AXA$ renders commutative \eqref{eq:t_3_module}
by the fusion equation on $t_3$. The first and the second diagrams of
\eqref{eq:module_compatibility} commute by the last diagram of
\eqref{eq:reg_mbm}. By the first diagram of \eqref{eq:reg_mbm}, both paths
around the last diagram of \eqref{eq:module_compatibility} are equal to
$1m:XA^2\to XA$ in \cq.  

The Hopf module compatibility conditions in \eqref{eq:Hopf_module} hold
by the short fusion equation \eqref{eq:short_fusion} on $t_1$ and $t_3$,
respectively. 
\end{proof}

Recall that the monoidal category \cc is said to be {\em left closed} if for
any object $X$, the endofunctor $X(-)$ possesses a right adjoint, to be
denoted by $[X,-]$. Whenever \cc is braided, then $[X,-]$ is also right
adjoint to the functor $(-)X$, meaning that \cc is also {\em right closed}; so
without any confusion we can just call it {\em closed}. The unit and the
counit of the adjunction $X(-)\dashv [X,-]$ will be denoted by $\eta$ and
$\epsilon$, respectively. 

Recall from Corollary~\ref{cor:v-tilde} that for a comodule $(V,v^1,v^3)$ over
a multiplier Hopf monoid with non-degenerate multiplication and dense counit,
the morphism $v^1$ is invertible. In particular, we may apply this to (the
underlying comodule of) a Hopf module. 

\begin{lemma}\label{lem:theta-tilde}
Let $\cc$ be a closed braided monoidal category satisfying our standing
assumptions of Section~\ref{sect:assumptions}. Let
$(A,t_1,t_2,t_3,t_4,e)$ be a regular multiplier bimonoid in $\cc$ with
non-degenerate multiplication, dense counit, and such that $(A,t_1,t_2,e)$ is
a multiplier Hopf monoid. 
Assume that for any Hopf module $(V,v^1,v^3,v_2,v_3)$, the morphism 
$$
\xymatrix@C=35pt{
w:=V \ar[r]^-\eta &
[A,AV] \ar[r]^-{[A,c]} &
[A,VA] \ar[r]^-{[A,(v^1)^{-1}]} &
[A,VA] \ar[r]^-{[A,v]} &
[A,V]}
$$
factorizes as $\xymatrix@C=15pt{V\ar[r]^-p &V^c \ar[r]^-i & [A,V]}$, where $p$
is an epimorphism and $i$ is an $A$-pure monomorphism. Then the following hold.
\begin{enumerate}
\item There is a unique morphism $\widetilde n:V \to V^cA$ rendering 
commutative  
$$
\xymatrix{
VA \ar[r]^-{v} \ar[d]_-{v^1} &
V \ar@{-->}[d]^-{\widetilde n} \\
VA \ar[r]_-{p1} &
V^cA\ .}
$$
\item The morphism $\widetilde n$ in part (1) is the inverse of
\begin{equation}\label{eq:theta}
\xymatrix{
n:= V^c A \ar[r]^-{i1} &
[A,V]A \ar[r]^-{c^{-1}} &
A[A,V] \ar[r]^-\epsilon &
V\ .}
\end{equation}
\end{enumerate}
\end{lemma}
\noindent 

\begin{proof}
(1) By assumption, $v:VA\to V$ arises as the coequalizer of some morphisms $f$
  and $g:X \to VA$; thus it is enough to prove that $p1.v^1$ coequalizes $f$
  and $g$. Since $i1:V^cA \to [A,V]A$ is a monomorphism by assumption, this is
  equivalent to showing that $w1.v^1$ coequalizes $f$ and $g$. This requires
  some preparation.

 In the following stringy computation the undecorated bubble stands for
$(v^1)^{-1}$.
$$
\begin{tikzpicture} 
\path (1,.6) node [arr,name=v] {$\ v\ $}
(1,1.2) node [arr,name=v1i] {$\quad\  $}
(1,2) node [unit,name=e] {};
\draw[braid] (.5,2.5) to[out=270,in=135] (v1i);
\draw[braid] (1,2.5) to[out=270,in=90] (e);
\draw[braid] (1.5,2.5) to[out=270,in=45] (v1i);
\draw[braid] (v1i) to[out=225,in=135] (v);
\draw[braid] (v1i) to[out=315,in=45] (v);
\draw[braid] (v) to[out=270,in=90] (1,0);
\draw (2,1.5) node {$=$};
\end{tikzpicture}
\begin{tikzpicture} 
\path (1,.5) node [arr,name=v] {$\ v\ $}
(1.25,1.5) node [arr,name=v1i] {$\quad\  $}
(1.24,1) node [unit,name=e] {}
(.6,1.5) node [empty,name=z] {};
\path[braid,name path=i>v1i] (.5,2.5) to[out=270,in=135] (v1i);
\draw[braid,name path=i>e] (1,2.5) to[out=270,in=90] (z) to[out=270,in=135] (e);
\draw[braid] (1.5,2.5) to[out=270,in=45] (v1i);
\path[braid,name path=v1i>v] (v1i) to[out=225,in=135] (v);
\draw[braid] (v1i) to[out=315,in=45] (v);
\draw[braid] (v) to[out=270,in=90] (1,0);
\fill[white,name intersections={of=i>v1i and i>e}] (intersection-1) circle(0.1);
\fill[white,name intersections={of=v1i>v and i>e}] (intersection-1) circle(0.1);
\draw[braid] (.5,2.5) to[out=270,in=135] (v1i);
\draw[braid] (v1i) to[out=225,in=135] (v);
\draw (2,1.5) node {$=$};
\end{tikzpicture}
\begin{tikzpicture} 
\path (1,.5) node [arr,name=v] {$\ v\ $}
(1.25,2.1) node [arr,name=v1i] {$\quad\  $}
(1.25,1.4) node [arr,name=t1i] {${t_1^{\scriptscriptstyle-1}}$}
(1.25,.9) node [empty,name=m] {}
(.7,2.1) node [empty,name=z] {};
\path[braid,name path=i>v1i] (.5,2.5) to[out=270,in=135] (v1i);
\draw[braid,name path=i>t1i] (1,2.5) to[out=270,in=90] (z) to[out=270,in=135] (t1i);
\draw[braid] (1.5,2.5) to[out=270,in=45] (v1i);
\path[braid,name path=v1i>v] (v1i) to[out=225,in=135] (v);
\draw[braid] (v1i) to[out=315,in=45] (t1i);
\draw[braid] (t1i) to[out=225,in=180] (m) to[out=0,in=315] (t1i);
\draw[braid] (m) to[out=270,in=45] (v);
\draw[braid] (v) to[out=270,in=90] (1,0);
\fill[white,name intersections={of=i>v1i and i>t1i}] (intersection-1) circle(0.1);
\fill[white,name intersections={of=v1i>v and i>t1i}] (intersection-1) circle(0.1);
\draw[braid] (.5,2.5) to[out=270,in=135] (v1i);
\draw[braid] (v1i) to[out=225,in=135] (v);
\draw (2,1.63) node {$\stackrel{\mathrm{(ass)}}{=}$};
\end{tikzpicture}
\begin{tikzpicture} 
\path (.75,.9) node [arr,name=vu] {$\ v\ $}
(1,.4) node [arr,name=vd] {$\ v\ $}
(1.25,2.1) node [arr,name=v1i] {$\quad\  $}
(1.25,1.4) node [arr,name=t1i] {${t_1^{\scriptscriptstyle-1}}$}
(.7,2.1) node [empty,name=z] {};
\path[braid,name path=i>v1i] (.5,2.5) to[out=270,in=135] (v1i);
\draw[braid,name path=i>t1i] (1,2.5) to[out=270,in=90] (z) to[out=270,in=135] (t1i);
\draw[braid] (1.5,2.5) to[out=270,in=45] (v1i);
\path[braid,name path=v1i>v] (v1i) to[out=225,in=135] (vu);
\draw[braid] (v1i) to[out=315,in=45] (t1i);
\draw[braid] (t1i) to[out=225,in=45] (vu);
\draw[braid] (t1i) to[out=315,in=45] (vd);
\draw[braid] (vu) to[out=270,in=135] (vd);
\draw[braid] (vd) to[out=270,in=90] (1,0);
\fill[white,name intersections={of=i>v1i and i>t1i}] (intersection-1) circle(0.1);
\fill[white,name intersections={of=v1i>v and i>t1i}] (intersection-1) circle(0.1);
\draw[braid] (v1i) to[out=225,in=135] (vu);
\draw[braid] (.5,2.5) to[out=270,in=135] (v1i);
\draw (2,1.7) node {$\stackrel{\eqref{eq:Hopf_module}}{=}$};
\end{tikzpicture}
\begin{tikzpicture} 
\path (.75,2) node [arr,name=vu] {$\ v\ $}
(1.1,.5) node [arr,name=vd] {$\ v\ $}
(1.1,1.25) node [arr,name=v1i] {$\quad\  $};
\draw[braid] (.5,2.5) to[out=270,in=135] (vu);
\draw[braid] (1,2.5) to[out=270,in=45] (vu);
\draw[braid] (1.5,2.5) to[out=270,in=45] (v1i);
\draw[braid] (vu) to[out=270,in=135] (v1i);
\draw[braid] (v1i) to[out=315,in=45] (vd);
\draw[braid] (v1i) to[out=225,in=135] (vd);
\draw[braid] (vd) to[out=270,in=90] (1.1,0);
\end{tikzpicture}
$$
This proves that 
\begin{equation}\label{eq:theta-tilde_a}
\xymatrix{
VA^2 \ar[r]^-{1e1} \ar[d]_-{v1} &
VA \ar[r]^-{(v^1)^{-1}} &
VA \ar[d]^-v \\
VA \ar[r]_-{(v^1)^{-1}} &
VA \ar[r]_-v &
V}
\end{equation}
commutes, and now   the commutativity of the next diagram also follows. 
$$
\scalebox{.95}{
\xymatrix@C=15pt{
VA^2 \ar@{=}[r]
\ar[ddddddd]^-{v1} &
VA^2 \ar[d]^-{1t_1} \ar@{=}[r] &
VA^2 \ar[rr]^-{\raisebox{8pt}{${}_{v^11}$}} \ar[dd]^-{1m} 
\ar@{}[rrdd]|-{\eqref{eq:v^1-3_module_maps}} &&
VA^2 \ar[r]^-{\raisebox{8pt}{${}_{\eta 11}$}} \ar[dd]^-{1m} &
[A,AV]A^2 \ar[r]^-{\raisebox{8pt}{${}_{[A,c]11}$}} &
[A,VA]A^2 \ar[dd]_-{[A,(v^1)^{-1}]11} \\
&
VA^2 \ar[d]^-{c1}  \ar[rd]^-{1e1} &&&\\
&
AVA \ar[r]^-{e11} \ar[d]^-{1v^1} &
VA \ar[rr]^-{v^1} &&
VA \ar@{=}[d] &&
[A,VA]A^2 \ar[dd]_-{[A,v]11} \\
&
AVA \ar[rrr]^-{e11} \ar[d]^-{c^{-1}1} &&&
VA \ar@{=}[d] \\
&
VA^2 \ar[rrr]^-{1e1} \ar@{=}[d] &&&
VA \ar[d]^-{\eta 1} &&
[A,V]A^2 \ar[dd]_-{1m} \\
&
VA^2 \ar[r]^-{\raisebox{8pt}{${}_{\eta 1}$}} \ar[dd]^-{v1} &
[A,AVA]A \ar@{=}[r] \ar[dd]^-{[A,1v]1} &
[A,AVA]A \ar[r]^-{\raisebox{8pt}{${}_{[A,11e]1}$}} \ar[d]^(.3){[A,c_{A,VA}]1} &
[A,AV]A \ar[d]^(.3){[A,c]1} \\
&&&
[A,VA^2]A \ar[r]^-{\raisebox{8pt}{${}_{[A,1e1]1}$}} \ar[d]^-{[A,v1]1}
\ar@{}[rrrd]|-{\eqref{eq:theta-tilde_a}} &
[A,VA]A \ar[r]^-{\raisebox{8pt}{${}_{[A,(v^1)^{-1}]1}$}} &
[A,VA]A \ar[r]^-{\raisebox{8pt}{${}_{[A,v]1}$}} &
[A,V]A \ar@{=}[d]\\
VA \ar[r]_-{v^1} \ar@{}[ruu]|-{\eqref{eq:Hopf_module}}&
VA \ar[r]_-{\eta 1} &
[A,AV]A \ar[r]_-{[A,c]1} &
[A,VA]A \ar[rr]_-{[A,(v^1)^{-1}]1} &&
[A,AV]A \ar[r]_-{[A,v]1} &
[A,V]A\ .}}
$$
Since the left-bottom path coequalizes $f1$ and $g1$, so does the top-right
path; from which we conclude using the non-degeneracy of $m=e1.t_1$.

(2) Since $v:VA \to V$ is an epimorphism by assumption, it follows by the
commutativity of 
$$
\xymatrix@C=40pt{
VA \ar[rrrr]^-v \ar[d]_-{v^1} &&&&
V \ar[d]^-{\widetilde n} \\
VA \ar[rrrr]^-{p1} \ar@{=}[d] &&&&
V^c A \ar[d]_-{i1} \ar@/^2.2pc/[ddd]^-n \\
VA \ar[r]^-{\eta 1} \ar[d]_-{c^{-1}} &
[A,AV]A \ar[r]^-{[A,c]1} &
[A,VA]A \ar[r]^-{[A,(v^1)^{-1}]1} &
[A,VA]A \ar[r]^-{[A,v]1} &
[A,V]A \ar[d]_-{c^{-1}} \\
AV \ar[r]^-{1\eta} \ar@/_1.5pc/@{=}[rd] &
A[A,AV] \ar[r]^-{1[A,c]} \ar[d]^-\epsilon &
A[A,VA] \ar[r]^-{1[A,(v^1)^{-1}]} &
A[A,VA] \ar[r]^-{1[A,v]} &
A[A,V] \ar[d]_-\epsilon \\
&
AV \ar[r]_-c &
VA \ar[r]_-{(v^1)^{-1}} &
VA \ar[r]_-v &
V}
$$
that $n.\widetilde n =1$.

In order to compute $\widetilde n.n$, consider the commutative diagram
\begin{equation}\label{eq:theta-tilde_d}
\xymatrix{
VA^2 \ar[rrrrr]^-{v^31} \ar[d]_-{1s_1} &&&&&
VA^2 \ar[dd]^-{1s_1} \\
VA \ar@{=}[r] \ar[dddd]_-{p1} &
VA \ar[r]^-{c^{-1}} \ar[d]^-{\eta 1} &
AV \ar[d]_-{1\eta} \ar@/^1.3pc/@{=}[rd] \\
&
[A,AV]A\ar[d]^-{[A,c]1} &
A[A,AV] \ar[r]^-\epsilon \ar[d]^-{1[A,c]} &
AV \ar[r]^-c &
VA \ar[r]^-{(v^1)^{-1}} &
VA \ar[ddd]^-v \\
&
[A,VA]A \ar[d]^-{[A,(v^1)^{-1}]1} &
A[A,VA] \ar[d]^-{1[A,(v^1)^{-1}]} \\
&
[A,VA]A \ar[d]^-{[A,v]1} &
A[A,VA] \ar[d]^-{1[A,v]} \\
V^cA \ar[r]^-{i1} \ar@/_1.3pc/[rrrrr]_-n &
[A,V]A \ar[r]^-{c^{-1}} &
A[A,V] \ar[rrr]^-\epsilon &&&
V}
\end{equation}
whose top region is seen to commute by applying \eqref{eq:g_*} to
$g=s$. 
With its help we see that 
$$
\xymatrix{
VA^2 \ar[rrr]^-{1s_1} \ar[dd]_-{1s_1} \ar[rd]^-{v^31}&&&
VA \ar[dd]^-{p1} \\
&
VA^2 \ar[r]^-{1s_1} \ar@{}[d]|-{\eqref{eq:theta-tilde_d}}&
VA \ar[ru]^-{v^1} \ar[d]^-v\\
VA \ar[r]_-{p1} &
V^cA \ar[r]_-n &
V \ar[r]_-{\widetilde n} &
V^cA}
$$
commutes; where commutativity of the region at the top follows postcomposing
by $v^1$ the defining identity $(v^1)^{-1}.1s_1=1s_1.v^31$ of
$(v^1)^{-1}=(s_*v)^1$ in \eqref{eq:g_*}. Since $1s_1$ and $p1$ are
epimorphisms, this proves $\widetilde n.n=1$.   
\end{proof}

\begin{theorem}[The fundamental theorem of Hopf modules] \label{thm:fthm}
Let $\cc$ be a closed braided monoidal category satisfying our standing
assumptions of Section~\ref{sect:assumptions}. Let
$(A,t_1,t_2,t_3,$ $t_4,e)$ be a regular multiplier bimonoid in $\cc$ with
non-degenerate multiplication, dense counit, and such that $(A,t_1,t_2,e)$ is
a multiplier Hopf monoid. Assume that, for any Hopf module
$(V,v^1,v^3,v_2,v_3)$, the morphism 
\begin{equation}\label{eq:w}
\xymatrix@C=35pt{
V \ar[r]^-\eta &
[A,AV] \ar[r]^-{[A,c]} &
[A,VA] \ar[r]^-{[A,(v^1)^{-1}]} &
[A,VA] \ar[r]^-{[A,v]} &
[A,V]}
\end{equation}
factorizes as $\xymatrix@C=15pt{V\ar[r]^-p &V^c \ar[r]^-i & [A,V]}$ through
some object $V^c$, where $p$ is a strong epimorphism and $i$ is an $A$-pure
monomorphism. 
Then the functor $(-)A$ in Proposition \ref{prop:comparison_functor}, from
$\cc$ to the category of Hopf modules, is an equivalence with inverse $(-)^c$.
\end{theorem}

\begin{proof}
Since the strong epi-mono factorization is unique up to isomorphism, the
object part of the functor $(-)^c$ is well-defined. Concerning its action on
the morphisms, we use the orthogonality property of the strong epimorphism
$p$. For any morphism of Hopf modules $h:V \to W$, in the diagram 
$$
\xymatrix@C=40pt{
&& V^c \ar@/^1pc/[rrd]^-i \\
V \ar[r]^-\eta \ar[d]_-h \ar@/^1pc/[rru]^-p&
[A,AV] \ar[r]^-{[A,c]} &
[A,VA] \ar[r]^-{[A,(v^1)^{-1}]} \ar[d]^-{[A,h1]}&
[A,VA] \ar[r]^-{[A,v]} \ar[d]^-{[A,h1]} &
[A,V] \ar[d]^-{[A,h]} \\
 W \ar[r]_-\eta \ar@/_1pc/[rrd]_-{ p}&
[A,A W] \ar[r]_-{[A,c]} &
[A, W  A] \ar[r]_-{[A, ( w^1)^{-1}]} &
[A, W  A] \ar[r]_-{[A, w]} &
[A, W ]\\
&&  W^c \ar@/_1pc/[rru]_-{ i}}
$$
the leftmost region commutes by naturality. The middle region commutes since
$h$ is a morphism of comodules and the region on the right commutes because
$h$ is a morphism of modules. 
Then orthogonality of the strong epimorphism $p:V\to V^c$ and the
monomorphism $i:W^c \to [A,W]$ implies the existence of a unique
morphism $h^c:V^c \to  W^c$ such that $ p.h=h^c.p$ (and consequently
also $ i.h^c=[A,h].i$). This defines the functor $(-)^c$ (up to an
unimportant natural isomorphism).    
 
For any object $X$ of $\cc$, the top-right path of the commutative diagram
$$
\xymatrix@C=40pt{
XA \ar[r]^-\eta \ar[d]_-{1e} &
[A,AXA] \ar[r]^-{[A,c_{A,XA}]} \ar[d]^-{[A,11e]} &
[A,XA^2] \ar[r]^-{[A,1t_1^{-1}]} \ar[rd]_-{[A,1e1]} &
[A,XA^2] \ar[d]^-{[A,1m]} \\
X \ar[r]_-\eta &
[A,AX] \ar[rr]_-{[A,c]} &&
[A,XA]}
$$
is the morphism \eqref{eq:w} for the Hopf module $XA$ in
Proposition \ref{prop:comparison_functor}. Now since $e\in \cq$, 
$$
\xymatrix{
XA \ar[r]^-{1e} &
X}
$$
is a regular, hence strong, epimorphism. On the other hand,
$e1\colon AX\to X$ is also a (regular) epimorphism, so the functor
$A(-)\colon\cc\to\cc$ is faithful, from which it follows that  
$$
\xymatrix{
X \ar[r]^-\eta &
[A,AX] \ar[r]^-{[A,c]} &
[A,XA]}
$$
is a monomorphism; thus these must give the factorization of \eqref{eq:w}, and
so $X$ is indeed isomorphic to $(XA)^c$. 

The natural isomorphism $V^cA\cong V$, for any Hopf module $V$, is provided by
the morphisms $n$ in \eqref{eq:theta} of Lemma \ref{lem:theta-tilde}. Their
naturality is an obvious consequence of the naturality of the constituent
morphisms.  
\end{proof}

\begin{remark}
If $A$ is a Hopf monoid with unit $u:I\to A$ in a closed braided monoidal
category $\cc$ in which idempotent morphisms split, then our assumptions in
Theorem \ref{thm:fthm} hold with any choice of \cq containing the split
epimorphisms. 

Indeed, the multiplication $m$ is evidently non-degenerate and the
counit $e$ is an epimorphism split by $u$. 


Furthermore, for any Hopf module $V$ there is an idempotent
morphism 
$$
\xymatrix{
V \ar[r]^-{1u} &
VA \ar[r]^-{(v^1)^{-1}} &
VA \ar[r]^-v &
V \ .}
$$
By assumption it admits a splitting
$
\xymatrix{
V \ar@{->>}[r]^-p &
V^c \ \ar@{>->}[r]^-j &
V}
$;
and we claim that it gives rise to a splitting 
\begin{equation}\label{eq:V^c_unital}
\xymatrix{
V  \ar[r]^-p &
V^c \ar[r]^-j &
V \ar[r]^-\eta &
[A,AV] \ar[r]^-{[A,c]} &
[A,VA] \ar[r]^-{[A,v]} &
[A,V]}
\end{equation}
as required in Theorem \ref{thm:fthm}.

Indeed, the morphism of \eqref{eq:V^c_unital} is equal to \eqref{eq:w} by the
commutativity of
$$
\xymatrix@C=40pt{
V \ar[r]^-{1u} \ar[d]_-\eta &
VA \ar[r]^-{(v^1)^{-1}} &
VA \ar[r]^-v &
V \ar[d]^-\eta \\
[A,AV] \ar[r]^-{[A,11u]} \ar[d]_-{[A,c]} &
[A,AVA] \ar[r]^-{[A,1(v^1)^{-1}]} \ar[d]^-{[A,c_{A,VA}]} &
[A,AVA] \ar[r]^-{[A,1v]} \ar[d]^-{[A,c_{A,VA}]} &
[A,AV] \ar[d]^-{[A,c]} \\
[A,VA] \ar[r]^-{[A,1u1]} \ar@{=}@/_1.3pc/[rd] &
[A,VA^2] \ar[r]^-{[A,(v^1)^{-1}1]} \ar[d]^-{[A,1m]} 
\ar@{}[rd]|-{\eqref{eq:v^1-3_module_maps}}&
[A,VA^2] \ar[r]^-{[A,v1]} \ar[d]^-{[A,1m]} &
[A,VA] \ar[d]^-{[A,v]} \\
&
[A,VA] \ar[r]_-{[A,(v^1)^{-1}]} &
[A,VA] \ar[r]_-{[A,v]} &
[A,V]}
$$
%
On the other hand, 
$$
\xymatrix{
V^c \ar[r]^-j &
V \ar[r]^-\eta &
[A,AV] \ar[r]^-{[A,c]} &
[A,VA] \ar[r]^-{[A,v]} &
[A,V]}
$$
is a split monomorphism since $j$ is so and also $[A,v.c].\eta$ is so by the
commutativity of
$$
\xymatrix{
V \ar[r]^-\eta \ar[d]^-{u1} \ar@/_1.3pc/[ddd]_-{1u} &
[A,AV] \ar[r]^-{[A,c]} &
[A,VA ] \ar[r]^-{[A,v]} &
[A,V] \ar[d]^-{u1} \\
AV \ar[r]^-{1\eta} \ar@{=}[d] &
A[A,AV] \ar[ld]^-\epsilon \ar[r]^-{1[A,c]} &
A[A,VA] \ar[r]^-{1[A,v]} &
A[A,V] \ar[dd]^-\epsilon \\
AV \ar[d]^-c \\
VA \ar[rrr]_-v &&&
V}
$$
in which the left-bottom path is an identity arrow by the unitality of $v$;
cf. Remark \ref{rmk:action-epi}. Since $p$ is 
a split hence strong epimorphism, this completes the proof. 
\end{remark}

\begin{remark}
Our assumptions in Theorem \ref{thm:fthm} also hold in another important
example, when $A$ is a non-zero multiplier Hopf algebra
\cite{VanDaele:multiplier_Hopf} in the closed symmetric monoidal
category of vector spaces over an arbitrary field. In this case 
\cq can be chosen to consist of all surjective linear transformations. (In
this category, the epimorphisms, the strong epimorphisms, the regular
epimorphisms, and the split epimorphisms all coincide: they are the
surjective linear maps.) 

The multiplication is non-degenerate hence non-zero
by assumption. This implies that the counit $e$ is a non-zero map to
the base field; hence it is surjective. 


On the other hand, in the category of vector spaces any morphism (that
is, linear map) $f\colon X\to Y$ factorizes through the surjection 
$\xymatrix@C=15pt{X \ar@{->>}[r] & \mathsf{Im}(f)}$ via the inclusion
$\xymatrix@C=12pt{\mathsf{Im}(f)\ \ar@{>->}[r] & Y}$. In the (abelian)
category of vector spaces this is in turn a strong epi-mono
factorization. 
Since vector spaces are in particular flat modules over the ground field, for
any vector space $X$ the monoidal product functor $X \otimes (-)$ preserves
monomorphisms.  
\end{remark}


\bibliographystyle{plain}

\end{document}